\documentclass[a4paper,11pt,english]{article}

\usepackage{a4wide,bm,epsfig,booktabs}

\usepackage{amsmath}
\usepackage{tikz}
\usepackage{setspace}
\usepackage{relsize}  
\usepackage{graphicx}  
\usepackage{cancel} 
\usepackage{slashed}
\usepackage{verbatim} 
\usepackage{enumerate} 
\usepackage{stmaryrd} 
\usepackage{cite}
\usepackage[flushmargin,hang]{footmisc}
\usepackage{fancyhdr} 
\usepackage[arrow, matrix,curve]{xy}
\usepackage{hyperref}
\usepackage{amssymb}
\usepackage{amsthm}
\usepackage{eucal}
\usepackage{xcomment}

\newcommand{\inverse}[1] {{\mathfrak{inv}(#1)}}

\newcommand{\NN} {\mathbb{N}}

\newcommand{\CC} {\mathbb{C}}
\newcommand{\RR} {\mathbb{R}}

\newcommand{\bl} {\boldsymbol{[}}
\newcommand{\br} {\boldsymbol{]}}

\newcommand{\JI} {\mathrm{J}}

\newcommand{\lip}  {\mathrm{lip}}

\newcommand{\uu} {\mathfrak{u}}
\newcommand{\mm} {\mathfrak{m}}
\newcommand{\nn} {\mathfrak{n}}
\newcommand{\oo} {\mathfrak{o}}
\newcommand{\pp} {\mathfrak{p}}
\newcommand{\qq} {\mathfrak{q}}
\newcommand{\ff} {\mathfrak{f}}
\newcommand{\vv} {\mathfrak{v}}
\newcommand{\ww} {\mathfrak{w}}

\newcommand{\symb} {\centerdot}
\newcommand{\bg} {\boldsymbol{\gamma}}

\newcommand{\mmm} {{\symb}\mm}
\newcommand{\nnn} {{\symb}\nn}

\newcommand{\ppp} {{\symb}\pp}
\newcommand{\qqq} {{\symb}\qq}

\newcommand{\vvv} {{\symb}\vv}
\newcommand{\www} {{\symb}\ww}

\newcommand{\lleq} {\prec}
\newcommand{\ggeq} {\succ}

\newcommand{\cpp} {\comp{\pp}}
\newcommand{\cqq} {\comp{\qq}}

\newcommand{\SEM} {\mathfrak{P}}
\newcommand{\SEML} {\symb\mathfrak{P}}
\newcommand{\SEMM} {\mathfrak{Q}}

\newcommand{\inv} {\mathrm{inv}}
\DeclareMathOperator*{\innt}{\ThisStyle{\vstretch{0.9}{\hstretch{1.5}{\rotatebox{10}{$\SavedStyle\hspace{-0.5pt}\!\int\!\hspace{-0.5pt}$}}}}}
\DeclareMathOperator*{\Der}{\ThisStyle{\hstretch{1.2}{\rotatebox{0}{$\SavedStyle\delta^r$}}}}
\usepackage{scalerel}

\newcommand{\DIDE} {\mathfrak{D}}
\newcommand{\DIDED} {\mathrm{D}}
\newcommand{\EV} {\mathrm{Evol}}
\newcommand{\EVE} {{\mathrm{evol}}}

\newcommand{\MX} {\mathfrak{X}}

\newcommand{\finite} {\mathrm{F}}

\newcommand{\B} {\mathrm{B}}
\newcommand{\OB} {\ovl{\B}}

\newcommand{\chart} {\Xi}
\newcommand{\chartinv} {\Xi^{-1}}
\newcommand{\RT} {\mathrm{R}}
\newcommand{\LT} {\mathrm{L}}
\newcommand{\ad} {\mathrm{ad}}
\newcommand{\Ad} {\mathrm{Ad}}
\newcommand{\com}[1] {\ad\:{#1}}

\newcommand{\pkt} {\boldsymbol{.}}

\newcommand{\CP} {\mathrm{CP}}
\newcommand{\COP} {\mathrm{CoP}}
\newcommand{\Poly} {\mathrm{Poly}}
\newcommand{\DP} {\DIDE\mathrm{P}}

\newcommand{\rhon} {\boldsymbol{\rho}}

\newcommand{\dindp}  {\mathrm{p}}
\newcommand{\dind}  {\mathrm{s}}
\newcommand{\dindo}  {\mathrm{o}}
\newcommand{\dindu}  {\mathrm{u}}

\newcommand{\mackeyconst} {\mathfrak{c}}

\newcommand{\mackeyindex} {\mathfrak{l}}
\newcommand{\wmackeyindex} {\mathfrak{j}}
\newcommand{\rad} {\mathfrak{r}}

\newcommand{\EINS} {\mathbf{1}}

\newcommand{\comp}[1] {\ovl{#1}}
\newcommand{\mgc}  {\ovl{\mg}}

\newcommand{\llleq} {\preceq}

\newcommand{\exxp}  {\mathfrak{exp}}

\newcommand{\conj}  {\mathrm{Conj}}

\newcommand{\id} {\mathrm{id}}

\newcommand{\dermap}  {\Omega}
\newcommand{\dermapinv}  {\Upsilon}

\newcommand{\dermapdiff}  {\omega}
\newcommand{\dermapinvdiff}  {\upsilon}

\newcommand{\oee} {\"o}

\newcommand{\sequ} {\mathfrak{Sequ}}
\newcommand{\mack}{\mathfrak{Mack}}

\newcommand{\sequy} {\mathfrak{s}}
\newcommand{\mackey}{\mathfrak{m}}

\newcommand{\MA} {\mathcal{A}}
\newcommand{\MAU} {\MA^\times}

\newcommand{\ovl}[1] {\overline{#1}}
\newcommand{\wt}[1] {\widetilde{#1}}
\newcommand{\wh}[1] {\widehat{#1}}

\newcommand{\dd} {\mathrm{d}}
\newcommand{\im} {\mathrm{im}}
\newcommand{\dom} {\mathrm{dom}}

\newcommand{\OO}  {\mathcal{O}}
\newcommand{\U}  {\mathcal{U}}
\newcommand{\V}  {\mathcal{V}}

\newcommand{\bound}  {\mathrm{B}}

\newcommand{\rcK}[1] {(#1)} %Klammern
\newcommand{\rcf}[1] {\frac{1}{#1}} %Klammern
\newcommand{\rcfsp} {\hspace{6.4pt}} %Klammern 	
\newcommand{\rcfstrich} {f'} %	
\newcommand{\rcp} {\partial_1} %partial
\newcommand{\rccf}[1] {(cf.\ #1)} %confer
\newcommand{\evek}[1] {\EVE_{#1}} %partial

\newcommand{\w} {\omega}
\newcommand{\deff} {{\it iff\hspace{0.7pt} }}%Used for equivalences, Theorems, etc.            
\newcommand{\defff} {\deff}    %Used for definitions

\newcommand{\mg} {\mathfrak{g}}

\newcommand{\cp} {\circ}
\newcommand{\COMP} {\mathfrak{K}}
\newcommand{\INT} {\mathfrak{I}}
\newcommand{\INTE} {\mathfrak{S}}

\newcommand{\mult}  {\mathrm{m}}

\newcommand{\op} {\mathrm{op}}

\newcommand{\compact} {\mathrm{C}}
\newcommand{\compacto} {\mathrm{K}}

\newcommand{\he} {\hspace{1pt}}

\renewcommand{\theenumi}{\arabic{enumi})} 
\renewcommand{\labelenumi}{\theenumi}

\let\origenumerate\enumerate
\def\enumerate{\origenumerate\itemsep0pt}
\let\origitemize\itemize
\def\itemize{\origitemize\itemsep0pt}

\newtheorem{innercustomthm}{Theorem}
\newenvironment{customthm}[1]
  {\renewcommand\theinnercustomthm{#1}\innercustomthm}
  {\endinnercustomthm}
  
 \newtheorem{innercustomtpr}{Proposition}
\newenvironment{customtpr}[1]
  {\renewcommand\theinnercustomtpr{#1}\innercustomtpr}
  {\endinnercustomtpr}
  
   \newtheorem{innercustomlem}{Lemma}
\newenvironment{customlem}[1]
  {\renewcommand\theinnercustomlem{#1}\innercustomlem}
  {\endinnercustomlem}
  
     \newtheorem{innercustomco}{Corollary}
\newenvironment{customco}[1]
  {\renewcommand\theinnercustomco{#1}\innercustomco}
  {\endinnercustomco}
  
     \newtheorem{innercustomrem}{Remark}
\newenvironment{customrem}[1]
  {\renewcommand\theinnercustomrem{#1}\innercustomrem}
  {\endinnercustomrem}

\newtheorem*{satz}{Theorem}

\newtheorem{theorem}{Theorem}
\newtheorem{proposition}{Proposition}
\newtheorem{lemma}{Lemma}
\newtheorem{corollary}{Corollary}
\newtheorem{remark}{Remark}

\newtheorem{example}{Example}

\makeatletter
\def\blfootnote{\gdef\@thefnmark{}\@footnotetext}
\makeatother

\begin{document}
\title{Regularity of Lie Groups}
\author{
  \textbf{Maximilian Hanusch}\thanks{\texttt{mhanusch@math.upb.de}}
  \\[1cm]
  Institut f\"ur Mathematik \\
  Lehrstuhl f\"ur Mathematik X \\
  Universit\"at W\"urzburg \\
  Campus Hubland Nord \\
  Emil-Fischer-Stra\ss e 31 \\
  97074 W\"urzburg \\
  Germany\\[0,9cm] 
      {\em Research done in part at} \\[0,3cm]
    Physics Department \\
  Florida Atlantic University \\
  777 Glades Road \\
  Boca Raton \\
    FL 33431 \\
  USA
}
%\date{\today}
\date{August 1, 2022}
\maketitle

\begin{abstract}  
We solve the regularity problem for Milnor's infinite dimensional Lie groups in the $C^0$-topological context, and provide necessary and sufficient regularity conditions for the (standard) $C^k$-topological setting. Specifically, we prove that if $G$ is an infinite dimensional Lie group in Milnor's sense, then the evolution map is $C^0$-continuous on its domain \deff $G$ is locally $\mu$-convex -- This is a continuity condition imposed on the Lie group multiplication that generalizes the triangle inequality for locally convex vector spaces.    
We furthermore show that if the evolution map is defined on all smooth curves, then $G$ is Mackey complete -- This is a completeness condition formulated in terms of the Lie group operations that generalizes Mackey completeness as defined for locally convex vector spaces; hence, we generalize the well known fact that a locally convex vector space is Mackey complete if each smooth (compactly supported) curve is Riemann integrable. 
Then, under the assumption that $G$ is locally $\mu$-convex, we show that each $C^k$-curve for $k\in \NN_{\geq 1}\sqcup\{\lip,\infty\}$ is integrable (contained in the domain of the evolution map) \deff $G$ is Mackey complete and $\mathrm{k}$-confined. The latter condition states that each $C^k$-curve in the Lie algebra $\mg$ of $G$ can be uniformly approximated by a special type of sequence that consists of piecewise integrable curves. A similar result is proven for the case $k\equiv 0$; and, we provide several mild conditions that ensure that $G$ is $\mathrm{k}$-confined for each $k\in \NN\sqcup\{\lip, \infty\}$.  
We finally discuss the differentiation of parameter-dependent integrals in the (standard) $C^k$-topological context. In particular, we show that if the evolution map is defined and continuous on $C^k([0,1],\mg)$ for $k\in \NN\sqcup\{\infty\}$, then it is smooth thereon: 
\begingroup
\setlength{\leftmarginii}{12pt}
\begin{itemize}
\item
	For $k=0$:\hspace{46.95pt} \deff it is differentiable at zero \deff $\mg$ is integral complete.
\item
	For $k\in \NN_{\geq 1}\sqcup\{\infty\}$:\hspace{4pt} \deff it is differentiable at zero \deff $\mg$ is Mackey complete.
\end{itemize}
\endgroup
\noindent
This result is obtained by calculating the directional derivatives explicitly, recovering the standard formulas (Duhamel) that hold, e.g., in the Banach (finite dimensional) case.
\end{abstract}
\noindent
{MSC Class: 22E65}
\tableofcontents 
\section{Introduction}
\label{intro}
The right logarithmic derivative and its inverse -- the evolution map -- play a central role in Lie theory. For instance, the existence of the exponential map -- indispensable for structure theory of Lie groups -- is based on the integrability of each constant curve (each such curve is contained in the domain of the evolution map). Moreover, given a principal fibre bundle, the existence of holonomies -- essential for gauge field theories -- is based on the integrability of curves that are pairings of a smooth connection with the derivative of a smooth curve in the  base manifold. In this paper, we discuss the evolution map in the infinite dimensional setting introduced by Milnor \cite{HG,HA,MIL,KHN}. Specifically, we consider an infinite dimensional Lie group $G$ as defined in \cite{HG} that is modeled over a Hausdorff locally convex vector space $E$, with system of continuous seminorms $\SEM$. We  denote the Lie algebra of $G$ by $(\mg,\bl\cdot,\cdot\br)$, 
the inversion of $G$ by $\inv\colon G\ni g\mapsto g^{-1}\in G$, the Lie group multiplication by $\mult\colon G\times G\rightarrow G$; and define $\LT_g:=\mult(g,\cdot)$ as well as $\RT_g:=\mult(\cdot, g)$ for each $g\in G$. We furthermore let $\Ad\colon G\times \mg\rightarrow\mg$ denote the adjoint action; and fix a chart $\chart\colon G\supseteq \U\rightarrow \V\in E$  with $\V$ convex, $e\in \U$, and $\chart(e)=0$. 
The right logarithmic derivative is defined by 
\begin{align*}
	\Der\colon C^1(D,G) \rightarrow C^0(D,\mg),\qquad \mu\mapsto \dd_{\mu}\RT_{\mu^{-1}}(\dot\mu)
\end{align*}
for $D\subseteq \RR$ a non-singleton interval and $\mu^{-1}\equiv\inv\cp\mu$; and, the evolution maps by 
\begin{align*}
	\EV\colon \DIDED\rightarrow C^1([0,1],G), \qquad     &\Der(\mu)\mapsto \mu\cdot \mu^{-1}(0)\\
	\EVE\colon \DIDED\rightarrow G,\hspace{48.7pt}\qquad &\Der(\mu)\mapsto \mu(1)\cdot \mu^{-1}(0)
\end{align*}
for $\mu\in \DIDED:=\Der(C^1([0,1],G))$. Then, the differential equation to be investigated is
\begin{align}
\label{popdsodpodspofffpods}
	\phi= \Der(\mu) \qquad\quad\text{for}\qquad\quad \phi \in C^0(D,\mg),\:\: \mu\in C^1(D,G);
\end{align}
whereby, in contrast to the Banach case, no theory of ODE's is available 
in the generic locally convex case --  
The core of this problem is rather the ``infinite dimensionality'' of the locally convex topology than the infinite dimensionality of the vector space $E$ itself. More specifically, in the context of a given continuous (linear) map $\phi\colon E\rightarrow E$, continuous seminorms can usually only be estimated against each other but not against themselves -- In general, this prevents the Banach fixed-point theorem (Picard-Lindel{\"o}f) and the Gr\oee{nwall} lemma to work.\footnote{Even if $E$ is metrizable via $\dd\colon E\times E\rightarrow \RR_{\geq 0}$, this metric usually fails to have the important property that $\dd(\lambda\cdot X+\lambda'\cdot X',0)\leq |\lambda|\cdot \dd(X,0)+ |\lambda'|\cdot \dd(X',0)$ holds for all $\lambda,\lambda'\in \RR$ and $X,X'\in E$ \cite{Rudin} -- making it incompatible with the Riemann integral (mean values).}  
Thus, given a specific differential equation, one has to use its particular ``symmetries'' in order to prove existence and uniqueness of solutions for arbitrary initial values. 
The ``symmetries'' hidden in \eqref{popdsodpodspofffpods} are
\begin{align}
\label{fgfggfssscvvccvcvcvcvccvcv}
\begin{split}
	\Der(\mu\cdot g)=\Der(\mu)\qquad\quad\text{and}\qquad\quad \Der(\mu|_{D'})=\Der(\mu)|_{D'}\hspace{90pt}\\[5pt]
	\Der( \mu\cp\varrho)=\dot\varrho\cdot\rcK{\Der(\mu)\cp\varrho}\hspace{170pt}\\[5pt]
	\Der(\mu\cdot \nu)= \Der(\mu)+\Ad_\mu(\Der(\nu))\qquad\text{implying}\qquad \Der(\mu^{-1}\nu)=\Ad_{\mu^{-1}}(\Der(\nu) -\Der(\mu))\hspace{18pt}
	\end{split}
\end{align}
for all $\mu,\nu\in C^1(D,G)$, $g\in G$, $\INT\ni D'\subseteq D\in \INT$, and each $\rho\colon \INT\ni D''\rightarrow D$ of class $C^1$; where $\INT$ denotes the set of all non-singleton intervals in $\RR$.
\vspace{6pt}

\noindent 
For instance, already in the Banach (finite dimensional) case, the first line in \eqref{fgfggfssscvvccvcvcvcvccvcv} 
is used to glue together local solutions that are provided by the Picard-Lindel{\"o}f theorem in this context. Following this philosophy, we will apply the second line in \eqref{fgfggfssscvvccvcvcvcvccvcv} to Riemann integrals of suitable bump functions to prove that \rccf{Theorem \ref{dsjkhjsdjhsd}}:
\begin{satz}
$G$ is Mackey complete if $C^\infty([0,1],\mg)\subseteq \DIDED$ holds; i.e., if  each smooth curve is integrable. 
\end{satz}
Here, \emph{Mackey completeness} is a condition formulated in terms of the Lie group operations that generalizes Mackey completeness as defined for locally convex vector spaces. The above theorem thus  generalizes the well-known fact that (cf., e.g., Theorem 2.14 in \cite{COS}) a Hausdorff locally convex vector space $E$ is Mackey complete if the Riemann integral of each (compactly supported) smooth curve (in $E$) exists in $E$. 
\vspace{6pt}

\noindent
Now, there is a further property of the evolution map that can be encoded in a topological condition imposed on the Lie group operations: We consider the restriction $\evek{k}\colon \DIDED\cap C^k([0,1],\mg)\rightarrow G$ for each $k\in \NN\sqcup\{\lip,\infty\}$; and say that $\evek{k}$ is $C^p$-continuous for $p\leq k$ ($p=0$ for $k\equiv \lip$) \defff it is continuous w.r.t.\ the subspace topology that is inherited by the $C^p$-topology on $C^k([0,1],\mg)$. 
We furthermore say that $G$ is \emph{locally $\mu$-convex} \defff for each $\uu\in \SEM$, there exists some $\uu\leq \oo\in \SEM$ with
\begin{align}
\label{aaajjhguoiuouoaaa}
	(\uu\cp\chart)(\chart^{-1}(X_1)\cdot {\dots}\cdot \chart^{-1}(X_n))\leq \oo(X_1)+{\dots}+\oo(X_n)
\end{align} 
for all $X_1,\dots,X_n\in E$, with $\oo(X_1)+{\dots}+\oo(X_n) \leq 1$.\footnote{This notion was originally introduced in \cite{HGGG} as a tool to investigate regularity properties of weak direct products of Lie groups.} 
Then, using the second line in \eqref{fgfggfssscvvccvcvcvcvccvcv}, we will show that \rccf{Theorem \ref{lfdskfddflkfdfd}}:  
\begin{satz}
   $\evek{0}$ is $C^0$-continuous \deff $G$ is locally $\mu$-convex \deff $\evek{\infty}$ is $C^0$-continuous.
\end{satz}
\emph{Evidently, \eqref{aaajjhguoiuouoaaa} generalizes the triangle inequality for locally convex vector spaces; and, due to the above theorem, it is independent of the explicit choice of the chart $\chart$.}
\vspace{6pt}

\noindent
Then, using the above two theorems, we will be able to partially answer the question under which circumstances $G$ is $C^k$-semiregular \cite{HGGG} for some given $k\in \NN\sqcup\{\lip,\infty\}$; i.e., under which circumstances $C^k([0,1],\mg)\subseteq \DIDED$ holds \rccf{Theorem \ref{confev}}:
\begin{satz}
Suppose that $G$ is locally $\mu$-convex. Then, $G$ is $C^k$-semiregular for $k\in \NN_{\geq 1}\sqcup \{\lip,\infty\}$ \deff $G$ is Mackey complete and $\mathrm{k}$-confined. Moreover, $G$ is $C^0$-semiregular if $G$ is sequentially complete and $\mathrm{0}$-confined. 
\end{satz}
\noindent
Here, k-confinedness is an approximation property for $C^k$-curves 
that is automatically fulfilled, 
e.g., if $(\mg,\bl\cdot,\cdot\br)$ is submultiplicative; or, if $G$ admits an exponential map, and $(\mg,\bl\cdot,\cdot\br)$ is constricted. The precise definitions, and more conditions can be found in Sect.\ \ref{patcases} or in Sect.\ \ref{dslklkslkdslkdsyxyyx}.
\vspace{6pt}

\noindent
In the last part of this paper, we will discuss the differentiation of parameter-dependent integrals in the standard topological setting. We first show that 
if $G$ is $C^k$-semiregular and $\evek{k}$ is $C^k$-continuous for $k\in \NN\sqcup\{\lip,\infty\}$, then the directional derivative (w.r.t.\ the $C^k$-topology) of $\evek{k}$ at zero along some $\phi\in C^k([0,1],\mg)$ exists in the completion $\mgc$ of $\mg$, as it is explicitly given by 
\begin{align*}
	\textstyle\frac{\dd}{\dd h}\big|_{h=0}\: \evek{k}(h\cdot \phi)=\int \phi(s)\:\dd s\in \mgc.
\end{align*}  
More generally: Recall that $\mg$ is said to be integral complete \cite{HGGG} \defff $\int \phi(s)\:\dd s\in \mg$ exists for each $\phi\in C^0([0,1],\mg)$; and let
\begin{align*}
	\textstyle\innt^s \phi := \EV(\phi)(s)\qquad\quad\text{as well as}\qquad\quad \innt\phi:=\innt^1\phi\qquad\qquad\forall\: \phi\in \DIDED,\:\: s\in [0,1].
\end{align*} 	
	Then, the above statement generalizes to \rccf{Theorem \ref{kdfklfklfdkjljljjllk}}: 
\begin{satz} 
\noindent
\vspace{-5pt}
\begingroup
\setlength{\leftmargini}{17pt}
{
\renewcommand{\theenumi}{{\arabic{enumi}})} 
\renewcommand{\labelenumi}{\theenumi}
\begin{enumerate}
\item
	Suppose that $G$ is $C^0$-semiregular and that $\evek{0}$ is $C^0$-continuous. Then, $\evek{0}$ is of class $C^1$ \deff $\mg$ is integral complete \deff $\evek{0}$ is differentiable at zero.
\item
	Suppose that $G$ is $C^k$-semiregular and that $\evek{k}$ is $C^k$-continuous, for $k\in \NN_{\geq 1}\sqcup\{\lip,\infty\}$. Then, $\evek{k}$ is of class $C^1$ \deff $\mg$ is Mackey complete \deff $\evek{k}$ is differentiable at zero.
\end{enumerate}}
\endgroup
\noindent
Here, for $k=0$ in the first-, and $k\in \NN_{\geq 1} \sqcup\{\lip,\infty\}$ in the  second case, we have 
\begin{align*}
	\textstyle\big(\dd_\phi\he\evek{k}\big)(\psi)\textstyle 
	=\dd_e\LT_{\innt \phi}\big(\int \Ad_{[\innt^s\!\phi]^{-1}}(\psi(s))\:\dd s\big)\qquad\quad\forall\: \phi,\psi\in C^k([0,1],\mg).
\end{align*}
\noindent
For $k\in \NN\sqcup\{\infty\}$, Theorem E in \cite{HGGG} additionally shows that $\evek{k}$ is even smooth.
\end{satz}
\noindent
Recall that $G$ is said to be \emph{$C^k$-regular} for $k\in \NN\sqcup\{\lip,\infty\}$ \defff $G$ is $C^k$-semiregular and $\evek{k}$ is smooth (w.r.t.\ the $C^k$-topology). Then,
\begingroup
\setlength{\leftmargini}{12pt}
\begin{itemize}
\item
the first point in the above theorem generalizes Theorem C.(a) in \cite{HGGG}, stating that each $C^0$-regular Lie group has an integral complete Lie algebra (modeling space). It furthermore generalizes  Theorem F in \cite{HGGG}, as it drops the presumption that  there  exists a point-separating family $(\alpha_j)_{j\in J}$ of smooth Lie group homomorphisms $\alpha_j\colon G\rightarrow H_j$ to $C^0$-regular Lie groups $H_j$. 
\item
since $C^\infty$-regular Lie groups are $C^\infty$-semiregular, the second point in the above theorem generalizes the result announced in Remark II.5.3.(b) in \cite{KHN2}, stating that each $C^\infty$-regular Lie group has a Mackey complete Lie algebra. 
\end{itemize}
\endgroup
\noindent
Actually, the last theorem is a consequence of a more general theorem concerning differentiation of parameter-dependent integrals.  
We write $\dind\llleq k$ for $\dind\in \NN$ and 
\begingroup
\setlength{\leftmargini}{12pt}
\begin{itemize}
\item
	$k\in \NN$\hspace{5.5pt}\:\:\he\defff\:\:$\dind\leq k$\:\: holds,
\item
	$k\equiv\lip$\:\:\he\defff\:\:$\dind=0$\:\: holds,
	\item
	$k\equiv\infty$\hspace{1.3pt}\:\:\he\defff\:\:$\dind\in \NN$\:\: holds.
\end{itemize}
\endgroup
\noindent
We furthermore recall that the $C^k$-topology on $C^k([r,r'],\mg)$ for $r<r'$ is generated by the seminorms
\begin{align}
\label{lkdlkfdlkfdlkfdlkfdlkfd}
	\ppp^\dind_\infty(\phi):=\sup\{(\pp\cp \dd_e\chart)\big(\phi^{(m)}(t)\big)\:|\: 0\leq m\leq \dind,\:t\in [r,r']\}\qquad\quad\forall\: \phi\in C^k([r,r'],\mg)
\end{align}
for $\pp\in \SEM$ and $\dind \llleq k$; and let $\ppp_\infty\equiv \ppp_\infty^0$ for each $\pp\in \SEM$. Then, we will show that \rccf{Theorem \ref{kckjckjs}}:
\begin{satz}
Suppose that $G$ is $C^k$-semiregular and that $\evek{k}$ is $C^k$-continuous, for some $k\in \NN\sqcup\{\lip,\infty\}$. Let furthermore $\Phi\colon I\times [0,1]\rightarrow \mg$ ($I\subseteq \RR$ open) be given with $\Phi(z,\cdot)\in C^k([0,1],\mg)$ for each $z\in I$.  
Then, 
\begin{align*}
	\textstyle\frac{\dd}{\dd h}\big|_{h=0} \big([\innt \Phi(x,\cdot)]^{-1}[\innt\Phi(x+h,\cdot)]\big)=\textstyle\int \Ad_{[\innt^s\Phi(x,\cdot)]^{-1}}(\partial_1\Phi(x,s))\:\dd s\hspace{1pt}\in \comp{\mg}
\end{align*}
holds for $x\in I$, provided that
\begingroup
\setlength{\leftmargini}{17pt}{
\renewcommand{\theenumi}{{\alph{enumi}})} 
\renewcommand{\labelenumi}{\theenumi}
\begin{enumerate}
\item
We have 
$(\partial_1 \Phi)(x,\cdot)\in C^k([0,1],\mg)$.
\item
For each $\pp\in \SEM$ and $\dind\leq k$, there exists $L_{\pp,\dind}\geq 0$, as well as $I_{\pp,\dind}\subseteq I$ open with $x\in I_{\pp,\dind}$, such that
\begin{align*}
	\textstyle\rcf{|h|}\cdot\ppp^\dind_\infty(\Phi(x+h,\cdot)-\Phi(x,\cdot))\leq L_{\pp,\dind}\qquad\quad \forall\: h\in \RR_{\neq 0}\:\text{ with }\: x+h\in I_{\pp,\dind}.
\end{align*}
\end{enumerate}}
\endgroup
\end{satz}
In particular, we will derive Duhamel's formula from this theorem in Sect.\ \ref{kjdkjdjkdkjd}.

This paper is organized as follows:
\begingroup
\setlength{\leftmargini}{12pt}
\begin{itemize}
\item
In Sect.\ \ref{pofpofdpofdpofdpfdfd}, we give a precise synopsis of the results obtained in this paper, and compare them to the results obtained in the literature so far. 
\item 
In Sect.\ \ref{prelim}, we provide the basic definitions, and prove the most elementary properties of the core mathematical objects of this paper. 
\item
In Sect.\ \ref{dsjshdhkjshkhjsd}, we prove certain continuity properties of the evolution map; and discuss piecewise integrable curves.
\item
In Sect.\ \ref{podposddpospodpods}, we show equivalence of $C^0$-continuity and local $\mu$-convexity.
\item
In Sect.\ \ref{COMPAPPR}, we show that each $C^\infty$-semiregular Lie group is Mackey complete; and prove certain approximation statements that are relevant for our discussion in Sect.\ \ref{patcases}.
\item
In Sect.\ \ref{sec:confined}, we show that, under the presumption that $G$ is locally $\mu$-convex, $G$ is $C^k$-semiregular for $k\in \NN_{\geq 1}\sqcup \{\lip,\infty\}$ \deff $G$ is Mackey complete and k-confined. Similar statements are proven for the case $k\equiv 0$.
\item
In Sect.\ \ref{asopsopdsopsdpoosdp}, we discuss the differentiation of parameter-dependent integrals.
\end{itemize}
\endgroup

\section{Precise Synopsis of the Results}
\label{pofpofdpofdpofdpfdfd}
In this section, we give a precise synopsis of the most important results obtained in this paper, and compare them to the results obtained in the literature so far, primarily in \cite{HG}. 
\subsection{Setting the Stage}
We are concerned with the following situation in this paper. We are given a Lie group $G$ in the sense of \cite{HG} (cf.\ Definition 3.1 and Definition 3.3 in \cite{HG}) that is modeled over a Hausdorff locally convex vector space $E$, with system of continuous seminorms $\SEM$. We denote Lie algebra of $G$ by $\mg$, fix a chart $\chart\colon G\supseteq \U\rightarrow \V\subseteq E$ with $\V$ convex, $e\in \U$, $\chart(e)=0$; and identify $\mg$ with $E$ via $\dd_e\chart\colon \mg\rightarrow E$ -- more specifically, this means that we define the seminorms $\{\ppp:=\pp\cp\dd_e\chart \:|\: \pp\in \SEM \}$ on $\mg$. We denote the inversion in $G$ by $\inv\colon G\ni g\mapsto g^{-1}\in G$, the Lie group multiplication by $\mult\colon G\times G\rightarrow G$; and let $\LT_g:=\mult(g,\cdot)$ as well as $\RT_g:=\mult(\cdot, g)$ for each $g\in G$. The adjoint action is denoted by  
$\Ad\colon G\times \mg\rightarrow \mg$; i.e., we have 
\begin{align*}
	 \Ad(g,X)\equiv \Ad_g(X):=\dd_e\conj_g(X)\qquad\quad\text{with}\qquad\quad \conj_g\colon G\ni h\mapsto g\cdot  h\cdot g^{-1}
\end{align*}
for each $X\in  \mg$ and $g\in G$. 
The differential equation under consideration then is
\begin{align}
\label{popdsodpodspopods}
	\phi=\Der(\mu)\equiv \dd_{\mu}\RT_{\mu^{-1}}(\dot\mu)\qquad\quad\text{for}\qquad\quad\: \phi \in C^0(D,\mg),\:\: \mu\in C^1(D,G),\:\: D\in \INT,
\end{align}
where $\INT$ denotes the set of all non-singleton intervals $D\subseteq \RR$.  
It is immediate from the definitions that
\vspace{5pt}
\begin{align}
\label{fgfggfsss}
	\Der(\mu\cdot g)=\Der(\mu)\qquad\quad\text{and}\qquad\quad \Der(\mu|_{D'})=\Der(\mu)|_{D'}\hspace{90pt}\\[5pt]
	\label{substi}
	\Der( \mu\cp\varrho)=\dot\varrho\cdot\rcK{\Der(\mu)\cp\varrho}\hspace{170pt}\\[5pt]
	\label{fgfggfssss}
	\Der(\mu\cdot \nu)= \Der(\mu)+\Ad_\mu(\Der(\nu))\qquad\text{implying}\qquad \Der(\mu^{-1}\nu)=\Ad_{\mu^{-1}}(\Der(\nu) -\Der(\mu))\hspace{17pt}
\end{align}
\vspace{-10pt}

\noindent
holds, for all $\mu,\nu\in C^1(D,G)$, $g\in G$, $\INT\ni D'\subseteq D\in \INT$, and each $\rho\colon \INT\ni D''\rightarrow D$ of class $C^1$ (we write $\mu^{-1}\equiv\inv\cp\mu$). 
Together with smoothness of the Lie group operations, and 
\begin{align}
\label{isdsdoisdiosdsss}
	\textstyle\gamma(t)-\gamma(r)=\int_r^t \dot\gamma(s)\:\dd s\in F\qquad\quad\forall\:t\in [r,r'],\:\: \gamma\in C^1([r,r'],F)
\end{align}
for $F$ a Hausdorff locally convex vector space,  
these are the only properties we have in hand to investigate Equation \eqref{popdsodpodspopods}. Let now $\COMP\subseteq \INT$ denote the set of all non-singleton compact intervals $[r,r'] \subseteq \RR$. 
Then, 
\begingroup
\setlength{\leftmargini}{12pt}
\begin{itemize}
\item
It follows from \eqref{fgfggfsss} that \rccf{Lemma \ref{evk}} for $k\in \NN$, we have 
\begingroup
\setlength{\leftmarginii}{12pt}
\begin{itemize}
\item[$\cp$]
	\vspace{-4pt}
	$\Der\colon C^{k+1}([r,r'],G)\rightarrow C^k([r,r'],\mg)$. 
	\vspace{2pt}
\item[$\cp$]
	$\mu\in C^{k+1}([r,r'],G)$ for each  $\mu\in C^1([r,r'],G)$ with $\Der(\mu)\in C^k([r,r'],\mg)$.
\end{itemize}
\endgroup
	\vspace{-4pt}
\item
It is then immediate from \eqref{isdsdoisdiosdsss} and the right side of \eqref{fgfggfssss} that (cf.\ Lemma \ref{xckxklxc})
\begin{align*}
	\Der\colon C_*^{k+1}([r,r'],G)\rightarrow C^k([r,r'],\mg)
\end{align*}
is injective for $k\geq 0$, with $C_*^{k+1}([r,r'],G):=\{\mu\in C^{k+1}([r,r'],G)\: |\: \mu(r)=e\}$. 
\end{itemize}
\endgroup
\noindent
We let $\DIDE_{[r,r']}:=\Der(C^{1}([r,r'],G))$ for each $[r,r']\in \COMP$, as well as $\DIDE_{[r,r']}^k:=\DIDE_{[r,r']}\cap C^k([r,r'],\mg)$ for each $k\in \NN\sqcup\{\lip,\infty\}$. Then, (we let $\lip+1:=1$,  $\infty +1:=\infty$) 
\begin{align*}
	\EV_{[r,r']}^k\colon \DIDE_{[r,r']}^k\rightarrow C_*^{k+1}([r,r'],G),\qquad \Der(\mu)\mapsto \mu\cdot \mu^{-1}(r)
\end{align*}
is well defined for $k\in \NN\sqcup\{\lip,\infty\}$, as well as surjective for $k\in \NN\sqcup\{\infty\}$. We define
\begin{align}
\label{podspodspopodspodpods}
	\textstyle\EVE_{[r,r']}^k\colon \DIDE_{[r,r']}^k\rightarrow G,\qquad \phi\mapsto \EV_{[r,r']}(\phi)(r') 
\end{align} 
with $\EV_{[r,r']}\equiv \EV_{[r,r']}^0$, for each $k\in \NN\sqcup\{\lip,\infty\}$; and denote
		\begin{align*}
		\textstyle\innt_a^b \phi := \EV_{[a,b]}(\phi|_{[a,b]})(b),\qquad\innt\phi:=\innt_r^{r'}\phi,\qquad\innt_c^c\phi:=e,\qquad \innt_r^\bullet\phi\colon [r,r']\ni t\mapsto \innt_r^t\phi 
	\end{align*} 
	for each $\phi\in \DIDE_{[r,r']}$, with $r\leq a<b\leq r'$ and $c\in [r,r']$. 
There are now several issues to be clarified. We first discuss
\subsection{Semiregularity and Mackey Completeness}
We say that $G$ is \emph{$C^k$-semiregular} for $k\in \NN\sqcup\{\lip,\infty\}$ \defff $\DIDE_{[0,1]}^k=C^k([0,1],\mg)$ holds. In this case, 
\begingroup
\setlength{\leftmargini}{12pt}
\begin{itemize}
\item
$G$ is $C^p$-semiregular for each $p\geq k$ (we let $1\geq \lip\geq \lip\geq 0$).
\item
it is straightforward from \eqref{substi} that $\DIDE_{[r,r']}^k=C^k([r,r'],\mg)$ holds for each $[r,r']\in \COMP$ \rccf{Lemma \ref{assasaasas}}.
\end{itemize}
\endgroup
\noindent
One then clearly wants to have criteria in hand for $G$ to be $C^k$-semiregular for some given $k\in \NN\sqcup \{\lip,\infty\}$.  
We provide the following necessary condition \rccf{Theorem \ref{dsjkhjsdjhsd}}:
\begin{customthm}{I}\label{A}
$G$ is Mackey complete if $G$ is $C^\infty$-semiregular. 
\end{customthm}
\noindent
Here, $G$ is said to be \emph{Mackey complete} \defff each \emph{Mackey-Cauchy sequence} converges in $G$; i.e., each sequence $\{g_n\}_{n\in \NN}\subseteq G$, such that
\begin{align*}
 	(\pp\cp\chart)(g^{-1}_m\cdot g_{n})\leq \mackeyconst_\pp\cdot \lambda_{m,n}\qquad\quad\forall\: m,n\geq \mackeyindex_\pp,\:\:\pp\in\SEM
 \end{align*}
holds for certain $\{\mackeyconst_\pp\}_{\pp\in \SEM}\subseteq \RR_{\geq 0}$, $\{\mackeyindex_\pp\}_{\pp\in \SEM}\subseteq \NN$, and $\RR_{\geq 0}\supseteq \{\lambda_{m,n}\}_{(m,n)\in \NN\times \NN}\rightarrow 0$. 
 \begingroup
\setlength{\leftmargini}{12pt}
\begin{itemize}
\item
This definition is independent of the explicit choice of the chart $\chart$ \rccf{Remark \ref{popodspodspodssds}}.
\item
This definition specializes to Mackey completeness as defined for locally convex vector spaces; i.e., the case where $(G,\cdot)\equiv (E,+)$ equals the additive group of a locally convex vector space $E$. 

Theorem \ref{A} thus generalizes the well-known fact (cf.\ Theorem 2.14 in \cite{COS}) that a locally convex vector space is Mackey complete if each smooth (compactly supported) curve is Riemann integrable.
\item
Mackey completeness is exemplarily verified in Example \ref{opfdopdfpofddfop} for Banach Lie groups; and the setting considered in \cite{HGIA}.
\end{itemize}
\endgroup 
\begin{customrem}{I}\label{opodfopd}
	The idea of the proof of Theorem \ref{A} is to construct some $\phi\in C^\infty([0,1],\mg)$ whose integral $\innt\phi$ is the limit of a (subsequence of a) given Mackey-Cauchy sequence $\{g_n\}_{n\in\NN}\subseteq G$. 
 	Roughly speaking, we will use \eqref{substi} to glue together smooth curves whose integrals equal $g_n^{-1}\cdot g_{n-1}$ via suitable bump functions. Here, it is important that (1.) a Mackey-Cauchy sequence converges \deff one of its subsequences converges, and (2.) passing to a subsequence if necessary, we can achieve that $(\pp\cp\chart)(g_n^{-1}\cdot g_{n-1})$ decreases suitably fast -- namely, (up to a factor $\mackeyconst_\pp$) in the same way for all seminorms $\pp\in \SEM$: This ensures that the so-constructed $\phi$ is defined and smooth at $1$ (where all of its derivatives must necessarily be zero).\hspace*{\fill}$\ddagger$
\end{customrem}
\subsection{Topologies and Continuity}
We say that $\EVE_{[r,r']}^k$ is $C^p$-continuous for $p\leq k\in \NN\sqcup\{\lip,\infty\}$ and $[r,r']\in \COMP$ \defff it is continuous w.r.t.\ the seminorms \eqref{lkdlkfdlkfdlkfdlkfdlkfd}, for $\dind \llleq p$. We say that $G$ is 
\begingroup
\setlength{\leftmargini}{12pt}
\begin{itemize}
\item
	p$\pkt$k-continuous for $p\llleq k$ \defff $\EVE_{[r,r']}^k$ is $C^p$-continuous for each $[r,r']\in \COMP$,
\item
	\hspace{6pt} k-continuous\hspace{46.5pt} \defff $G$ is k$\pkt$k-continuous.
\end{itemize}
\endgroup
\noindent
It is straightforward from \eqref{substi} and the right side of \eqref{fgfggfssss} that \rccf{Lemma \ref{klklllkjlaaa}} 
\begin{customlem}{I}\label{opfdopfdop}
$G$ is {\rm p$\pkt$k}-continuous \deff $\EVE^k_{[0,1]}$ is $C^p$-continuous at zero.
\end{customlem}
{\it Under the presumtion that $G$ is $C^k$-semiregular (for $k\in \NN\sqcup \{\infty\}$), it had already been shown in Theorem D in \cite{HGGG} that $\EVE^k_{[0,1]}$ is $C^k$-continuous \deff it is $C^k$-continuous at zero.}
\vspace{6pt}

\noindent
Clearly, for $k\geq 1$, the $C^0$-topology is strictly coarser than the $C^k$-topology; so that 0$\pkt$k-continuity implies k-continuity but usually not vice versa. Anyhow, it is straightforward from \eqref{substi} and \eqref{fgfggfssss} that \rccf{Lemma \ref{ssssss}}: 
\begin{customlem}{II}\label{CL}
If $G$ is abelian, then $G$ is {\rm k}-continuous for $k\in \NN\sqcup\{\infty\}$ \deff $G$ is {\rm 0$\pkt$k}-continuous.
 \end{customlem}
 \noindent
The important feature of 0.k-continuity is that it can be encoded in a continuity property of the Lie group multiplication:       
Recall that $G$ is said to be locally $\mu$-convex \defff \eqref{aaajjhguoiuouoaaa} holds. We  
will show that \rccf{Theorem \ref{lfdskfddflkfdfd}}:  
\begin{customthm}{II}\label{B}
$G$ is {\rm 0}-continuous \deff $G$ is locally $\mu$-convex \deff $G$ is {\rm 0}$\pkt\infty$-continuous.
\end{customthm}
\noindent
Specifically, for $k\in \NN\sqcup\{\infty\}$ and $[r,r']\in \COMP$, let   
$\DP^k([r,r'],\mg)$ denote the set of all maps $\phi\colon [r,r']\rightarrow \mg$ such that there exist $r=t_0<{\dots}<t_n=r'$ and $\phi[p]\in \DIDE^k_{[t_p,t_{p+1}]}$ for $p=0,\dots,n-1$ with
\begin{align*}
	\phi|_{(t_p,t_{p+1})}=\phi[p]|_{(t_p,t_{p+1})}\qquad\quad\forall\: p=0,\dots,n-1.
\end{align*}
Moreover, define the integral of $\phi$ by (well-definedness is straightforward from \eqref{fgfggfsss})
\begin{align}
\label{oppopopo}
	\textstyle\innt_r^t\phi&\textstyle:=\innt_{t_{p}}^t \phi[p] \cdot \innt_{t_{p-1}}^{t_p} \phi[p-1]\cdot {\dots} \cdot \innt_{t_0}^{t_1}\phi[0]\qquad\quad \forall\: t\in (t_{p}, t_{p+1}],\:\: p=0,\dots,n-1.
\end{align} 	
Then, the one direction in Theorem \ref{B} is covered by \rccf{Proposition \ref{aaapofdpofdpofdpofd}}:
\begin{customtpr}{I}\label{C}
Suppose that $G$ is locally $\mu$-convex. Then, for each $\pp\in \SEM$, there exists some $\pp\leq \qq\in \SEM$, such that
\begin{align*}
	\textstyle\int\qqq(\phi(s))\:\dd s \leq 1\quad\:\text{for}\quad\: \phi\in \DP^0([r,r'],\mg) \qquad\quad\Longrightarrow\qquad\quad	\textstyle(\pp\cp\chart)\big(\innt_r^\bullet\phi\big)\leq \int_r^\bullet \qqq(\phi(s))\:\dd s, 
\end{align*} 
for each $[r,r']\in \COMP$.
\end{customtpr}
\begin{customrem}{II}\label{khdshdshsdjhsdhjsd}
Apart from Proposition \ref{C}, the set $\DP^k([r,r'],\mg)$ plays an important role in the proof of the other direction in Theorem \ref{B}. Here, the key observation is that $\phi\in \DP^k([r,r'],\mg)$ given with $\qqq_\infty(\phi)\leq 1/2$ for some $\qq\in \SEM$, it is possible to construct $\varrho\colon [r,r']\rightarrow [r,r']$ smooth with $|\dot\varrho|\leq 2$, such that 
\begin{align*}
	\textstyle\dot\varrho\cdot \rcK{\phi\cp\varrho}\in \DIDE^k_{[r,r']}\qquad\quad\text{as well as}\qquad\quad\innt\phi=\innt\dot\varrho\cdot \rcK{\phi\cp\varrho} 
\end{align*}
holds; i.e., $\qqq_\infty(\dot\varrho\cdot \rcK{\phi\cp\varrho})\leq 1$ \rccf{Lemma \ref{pofdspospods}}. Continuity of $\EVE_{[r,r']}^k$ w.r.t.\ to the seminorms $\{\ppp_\infty\}_{\pp\in \SEM}$ thus carries over to the set $\DP^k([r,r'],\mg)$; and, then  \eqref{aaajjhguoiuouoaaa} is a straightforward consequence of \eqref{oppopopo}. 
Here, $\varrho$ is obtained by glueing together (and then integrating) suitable bump functions; so that the argument fails on the level of the $C^k$-topology for $k\geq 1$, just because the higher derivatives of a so-constructed $\varrho$ become that larger that finer the decomposition of $[r,r']$ is made.  
\hspace*{\fill}$\ddagger$
\end{customrem}
\noindent
Finally, we say that $G$ is $\rm L^1$-continuous \defff $\EVE_{[r,r']}^0$ is continuous w.r.t.\ the $L^1$-seminorms
\begin{align}
\label{ofdpofdpofdpofd}
	\textstyle\ppp_{\int}(\phi):=\int \ppp(\phi(s))\: \dd s\qquad\quad\forall\: \pp\in \SEM,\:\: \phi\in C^0([r,r'],\mg)
\end{align}
for each $[r,r']\in \COMP$. 
Then, Theorem \ref{B} and Proposition \ref{C} show that $G$ is $\rm L^1$-continuous \deff $G$ is locally $\mu$-convex \deff $G$ is 0$\pkt\infty$-continuous; which generalizes Lemma 14.9 in \cite{HGGG}. Here, the equivalence of $\rm L^1$-continuity and 0-continuity is already straightforward from \eqref{substi} \rccf{cf.\ Lemma \ref{sdsddsdsdsdsds}}.

\subsection{Integrability}
\label{dslklkslkdslkdsyxyyx}
We now come back to the question under which circumstances a given $\phi\in C^0([0,1],\mg)$ is integrale, i.e., contained in $\DIDE_{[0,1]}^0$. A sequence $\{\phi_n\}_{n\in \NN}\subseteq\DP^0([0,1],\mg)$ is said to be \emph{tame} \defff for each $\vv\in \SEM$, there exists some $\vv\leq \ww\in \SEM$, such that 
\begin{align*}	 
	 \vvv\cp \Ad_{[\innt_0^\bullet \phi_n]^{-1}}\leq  \www\qquad\quad\forall\: n\in \NN 
\end{align*}
holds. 
Moreover, 
$\{\phi_n\}_{n\in \NN}\subseteq\DP^0([0,1],\mg)$ is said to be a
	 \begingroup
\setlength{\leftmargini}{12pt}
\begin{itemize}
\item
	Cauchy sequence \defff to each $\pp\in \SEM$ and $\varepsilon>0$, there exists some $p\in \NN$ with 
	$\ppp_\infty(\phi_m-\phi_n)\leq \varepsilon$ for all $m,n\geq p$.
\item
	 Mackey-Cauchy sequence \defff there exists a net $\RR_{\geq 0}\supseteq \{\lambda_{m,n}\}_{(m,n)\in \NN\times \NN}\rightarrow 0$, as well as constants $\{\mackeyconst_\pp\}_{\pp\in \SEM}\subseteq \RR_{\geq 0}$ and $\{\mackeyindex_\pp\}_{\pp\in \SEM}\subseteq \NN$, such that for each $\pp\in\SEM$, we have
	\begin{align*}
	\ppp_\infty(\phi_m-\phi_n)\leq \mackeyconst_\pp\cdot \lambda_{m,n} \qquad\quad\forall\: m,n\geq \mackeyindex_\pp.
\end{align*}
\end{itemize}
\endgroup
\noindent 
Then, $\phi\in C^0([0,1],\mg)$ is said to be
\begingroup
\setlength{\leftmargini}{12pt}
\begin{itemize}
\item
		\emph{$\sequy$-integrable} \defff there exists a tame Cauchy sequence 
	$\{\phi_n\}_{n\in \NN}\subseteq\DP^0([0,1],\mg)$ with $\{\phi_n\}_{n\in \NN}\rightarrow\phi$ uniformly (i.e., w.r.t\ the seminorms $\{\ppp_\infty\}_{\pp\in \SEM}$).
\item
	\emph{$\mackey$-integrable} \defff there exists a tame Mackey-Cauchy sequence 
	$\{\phi_n\}_{n\in \NN}\subseteq\DP^0([0,1],\mg)$ with $\{\phi_n\}_{n\in \NN}\rightarrow\phi$ uniformly. 
\end{itemize}
\endgroup
\noindent
We will show that \rccf{Lemma \ref{sdsddsd} and Proposition \ref{dfopfdpoofdp}}:
\begin{customtpr}{II}\label{popodfopdfopdfoppfdo}
Suppose that $G$ is locally $\mu$-convex.  
\begingroup
\setlength{\leftmargini}{17pt}
\begin{enumerate}
\item
If $G$ is sequentially complete, then $\phi\in \DIDE_{[0,1]}^0$ holds for $\phi\in C^0([0,1],\mg)$ \deff $\phi$ is $\sequy$-integrable.
\item
If $G$ is Mackey complete, then $\phi\in \DIDE_{[0,1]}^\lip$ holds for $\phi\in C^\lip([0,1],\mg)$ \deff $\phi$ is $\mackey$-integrable.
\end{enumerate}
\endgroup
\end{customtpr}
\begin{customrem}{III}\label{fddffdd}
The one direction in Proposition \ref{popodfopdfopdfoppfdo} is immediate from the fact that $\mu\colon t\mapsto \innt^t_0\phi$ has compact image, for each $\phi\in \DIDE_{[0,1]}^0$. For the other direction (in analogy to the Riemann integral) one defines 
\begin{align*}
	\textstyle\mu(t) := \lim_n \innt_0^t \phi_n\qquad\quad\forall\: t\in [0,1];
\end{align*}
and then has to verify (1.) that the limit exists pointwise, i.e., that $\mu$ is defined, (2.) that $\mu$ is continuous, (3.) that $\{\innt_0^\bullet \phi_n\}_{n\in \NN}\rightarrow \mu$ converges uniformly, and (4.) that $\mu$ is of class $C^1$ with $\Der(\mu)=\phi$. \hspace*{\fill}$\ddagger$
\end{customrem}
\noindent 
We say that $G$ is \emph{$\mathrm{k}$-confined} 
\begingroup
\setlength{\leftmargini}{12pt}
\begin{itemize}
\item
	for $k\equiv 0$:\hspace{63.6pt}\quad\:\: \defff each $\phi\in C^0([0,1],\mg)$ is $\sequy$-integrable,
\item
	for $k\in \NN_{\geq 1}\sqcup\{\lip,\infty\}$:\quad\:\: \defff each $\phi\in C^k([0,1],\mg)$ is $\mackey$-integrable;
\end{itemize}
\endgroup
\noindent
and obtain from Theorem \ref{A} that \rccf{Theorem \ref{confev}}:	
\begin{customthm}{III}\label{D}
Suppose that $G$ is locally $\mu$-convex. Then, $G$ is $C^k$-semiregular for $k\in \NN_{\geq 1}\sqcup \{\lip,\infty\}$ \deff $G$ is Mackey complete and $\mathrm{k}$-confined. Moreover, $G$ is $C^0$-semiregular if $G$ is sequentially complete and $\mathrm{0}$-confined. 
\end{customthm}	
\noindent
For instance, $G$ is k-confined for each $k\in \NN\sqcup\{\lip,\infty\}$ \rccf{Sect.\ \ref{patcases}}:
\begingroup
\setlength{\leftmargini}{12pt}
\begin{itemize}
\item
If $G$ is abelian; or, more generally, if $(\mg,\bl\cdot,\cdot\br)$ is submultiplicative.
\item
If $G$ is locally $\mu$-convex and \emph{reliable}; i.e., if for each $\vv\in \SEM$, there exists a symmetric neighbourhood $V\subseteq G$ of $e$, and a sequence $\{\ww_n\}_{n\in \NN_{\geq 1}}\subseteq \SEM$, 
such that 
\begin{align*}
	\vvv\cp \Ad_{g_1}\cp{\dots}\cp\Ad_{g_n}\leq \www_{n}\qquad\quad\forall\: g_1,\dots,g_n\in V,\:\:n\geq 1.
\end{align*}
	This is the case, e.g., 
	for the unit group $\MAU$ of a  continuous inverse algebra $\MA$ fulfilling the condition $(*)$ (cf.\ \eqref{invACond}) formulated in the theorem proven in \cite{HGIA}.
\item
If $G$ admits an exponential map, is constricted, and has a sequentially complete Lie algebra. 

Here, the first condition means that $\phi_X|_{[0,1]}\in \DIDE_{[0,1]}$ holds, for each constant curve $\phi_X\colon \RR\ni t\mapsto X\in \mg$; i.e., that 
\begin{align}
\label{pofdpofdpofdpofdfpodfpopofdpofd}
	\textstyle\exp\colon \mg\ni X\mapsto \innt_0^1 \phi_X\in G
\end{align}
is defined. Moreover, constrictedness states that for each bounded subset $\bound\subseteq \mg$, and each $\vv\in \SEM$, there exist $C\geq 0$ and $\vv\leq \ww\in \SEM$, such that 
\begin{align*}
	\vvv\cp \com{X_1}\cp {\dots}\cp \com{X_n}\leq  C^n\cdot \www\qquad\quad\forall\: X_1,\dots,X_n\in \bound,\:\: n\geq 1
\end{align*}
holds, with $\com{X} \colon \mg\ni Y\mapsto \bl X,Y\br\in \mg$ for each $X\in \mg$. 
\end{itemize}
\endgroup
\noindent 
In particular, 
\begin{customco}{I}\label{CCA}
If $G$ is abelian, then $G$ is $C^\infty$-semiregular and $\rm\infty$-continuous \deff $G$ is Mackey complete and locally $\mu$-convex \deff $G$ is $C^k$-semiregular and {\rm k}-continuous for each $k\in \NN_{\geq 1}\sqcup\{\lip,\infty\}$.  
\end{customco}
\begin{proof}
If $G$ is $C^\infty$-semiregular and $\infty$-continuous, then $G$ is Mackey complete by Theorem \ref{A}, as well as 0$\pkt\infty$-continuous by Lemma \ref{CL}; thus, locally $\mu$-convex by Theorem \ref{B}. 
Conversely, if $G$ is locally $\mu$-convex, then $G$ is (even 0$\pkt$)k-continuous for each $k\in \NN_{\geq 1}\sqcup \{\lip,\infty\}$ by Theorem \ref{B}. Since $G$ is $\lip$-confined, Theorem \ref{D} shows that $G$ is $C^k$-semiregular for each $k\in \NN_{\geq 1}\sqcup \{\lip,\infty\}$ if $G$ is additionally Mackey complete.
\end{proof}
\subsection{Smoothness and Differentiation}
In Sect.\ \ref{asopsopdsopsdpoosdp}, we will discuss the differentiation of parameter-dependent integrals in the standard setting; i.e., w.r.t.\ the $C^k$-topology. Our key observation there is \rccf{Proposition \ref{rererererr}}: 
\begin{customtpr}{III}\label{pcvcvcvvccvopodfopdfopdfoppfdo}
Suppose that $G$ is \emph{k-continuous} for $k\in \NN\sqcup \{\lip,\infty\}$; and that $(-\delta,\delta)\cdot \phi\subseteq \DIDE^k_{[r,r']}$ holds for some $\phi\in C^k([r,r'],\mg)$ for $[r,r']\in \COMP$ and $\delta>0$.
Then, we have
\begin{align*}
	\textstyle\frac{\dd}{\dd h}\big|_{h=0}\: \EVE_{[r,r']}^k(h\cdot \phi)=\int \phi(s)\:\dd s\in \comp{\mg}.
\end{align*}
Thus, the directional derivative of $\EVE_{[r,r']}^k$ at zero along such a $\phi\in C^k([r,r'],\mg)$ always exists; namely, in the completion $\mgc$ of $\mg$.
\end{customtpr}
\noindent
We say that $\mg$ is \emph{integral complete} \cite{HGGG} \defff $\int \phi(s)\:\dd s\in \mg$ exists for each $\phi\in C^0([0,1],\mg)$;  and recall that $\mg$ is Mackey complete \deff $\int \phi(s)\:\dd s\in \mg$ exists for each $\phi\in C^\infty([0,1],\mg)$.  
Then, the above proposition immediately shows that \rccf{Corollary \ref{MC}}:
\begin{customco}{II}\label{lkflkf}
\noindent
\vspace{-5pt}
\begingroup
\setlength{\leftmargini}{17pt}
\begin{enumerate}
\item
Suppose that $G$ is {\rm 0}-continuous and $C^0$-semiregular. Then,   
 $\EVE^0_{[0,1]}$ is differentiable at zero \deff $\mg$ is \emph{integral complete}.
\item
Suppose that $G$ is {\rm k}-continuous for some $k\in \NN_{\geq 1}\sqcup\{\lip,\infty\}$, as well as $C^\infty$-semiregular. Then,   
 $\EVE^k_{[0,1]}\big|_{C^\infty([0,1],\mg)}$ is differentiable at zero \deff $\mg$ is Mackey complete.
\end{enumerate}
\endgroup
\end{customco}
\noindent
Here, the first point generalizes Theorem C.(a) in \cite{HGGG} stating that each $C^0$-regular Lie group has an integral complete Lie algebra (modeling space). The  second point generalizes the analogous result announced in Remark II.5.3.(b) in \cite{KHN2} stating that each $C^\infty$-regular Lie group has a Mackey complete Lie algebra -- Recall that $G$ is said to be \emph{$C^k$-regular} for $k\in \NN\sqcup\{\lip,\infty\}$ \defff $G$ is $C^k$-semiregular and $\EVE^k_{[0,1]}$ is smooth w.r.t.\ the $C^k$-topology.
\vspace{6pt}

\noindent
Next, using the above proposition, we show that, cf. Theorem \ref{kckjckjs}: 
\begin{customthm}{IV}\label{E}
Suppose that $G$ is {\rm k}-continuous and $C^k$-semiregular for some $k\in \NN\sqcup\{\lip,\infty\}$; and let $\Phi\colon I\times [r,r']\rightarrow \mg$ ($I\subseteq \RR$ open) be fixed with $\Phi(z,\cdot)\in C^k([r,r'],\mg)$ for each $z\in I$.  
Then, 
\begin{align*}
	\textstyle\frac{\dd}{\dd h}\big|_{h=0} \big([\innt \Phi(x,\cdot)]^{-1}[\innt\Phi(x+h,\cdot)]\big)=\textstyle\int \Ad_{[\innt_r^s\Phi(x,\cdot)]^{-1}}(\rcp\Phi(x,s))\:\dd s\hspace{1pt}\in \comp{\mg}
\end{align*}
holds for $x\in I$, provided that
\begingroup
\setlength{\leftmargini}{17pt}{
\renewcommand{\theenumi}{{\alph{enumi}})} 
\renewcommand{\labelenumi}{\theenumi}
\begin{enumerate}
\item
We have 
$(\rcp \Phi)(x,\cdot)\in C^k([r,r'],\mg)$.
\item
For each $\pp\in \SEM$ and $\dind\llleq k$, there exists $L_{\pp,\dind}\geq 0$, as well as $I_{\pp,\dind}\subseteq I$ open with $x\in I_{\pp,\dind}$, such that
\begin{align*}
	\textstyle \rcf{|h|}\cdot\ppp^\dind_\infty(\Phi(x+h,\cdot)-\Phi(x,\cdot))\leq L_{\pp,\dind}\qquad\quad \forall\: h\in \RR_{\neq 0}\:\text{ with }\: x+h\in I_{\pp,\dind}.
\end{align*}
\end{enumerate}}
\endgroup
\end{customthm}
\vspace{-4pt}
\noindent
For instance, we obtain \rccf{Corollary \ref{sasassasasa}}:
\begin{customco}{III}\label{ospospddopsopsdopdp}
Suppose that $G$ is $\rm\infty$-continuous, and $C^\infty$-semiregular; and that $\mg$ is Mackey complete. Then, for $\MX\colon I\rightarrow \mg$ of class $C^1$, we have
\begin{align*}
	\textstyle\partial_z \exp(\MX(x))=\dd_e\LT_{\exp(\MX(x))}\big(\int \Ad_{\exp(-s\cdot \MX(x))}(\partial_z\MX(x)) \:\dd s\big)\qquad\quad\forall\: x\in I.
\end{align*}
\end{customco}
Imposing further presumptions, this specializes to \emph{Duhamel's formula} \rccf{Proposition \ref{sahjhjsahjsahjsahjsaqqwppowqpowq}}: 

{\it(Actually, in Proposition \ref{sahjhjsahjsahjsahjsaqqwppowqpowq}, a slightly more general situation is considered.)}
\begin{customtpr}{IV}\label{ajshahjshaashsahasasasas}
Suppose that $G$ is $\infty$-continuous, $C^\infty$-semiregular, and constricted; and that $\mg$ is sequentially complete.    
Then, for each $\MX\colon I\rightarrow \mg$ of class $C^1$, we have\footnote{The precise definition of the expression in the parentheses on the left side can be found in Sect.\ \ref{kjdkjdjkdkjd} \rccf{Equation \eqref{odsoidoioisdoidsoids}}.}
\begin{align*}
	\textstyle\partial_z \exp(\MX(x))\textstyle=\dd_e\LT_{\exp(\MX(x))}\Big(\frac{\id_\mg-\exp(-\com{\MX(x)})}{\com{\MX(x)}}(\partial_z\MX(x))\Big)\qquad\quad\forall\: x\in I.
\end{align*}
\end{customtpr}
\noindent
Now, Theorem E in \cite{HGGG} states that $\EVE_{[0,1]}^k$ is smooth if $G$ is $C^k$-semiregular, and $\EVE_{[0,1]}^k$ is of class $C^1$. Combining this with Corollary \ref{lkflkf} and Theorem \ref{E}, we  obtain \rccf{Theorem \ref{kdfklfklfdkjljljjllk}}:
\begin{customthm}{V}\label{F} 
\noindent
\vspace{-5pt}
\begingroup
\setlength{\leftmargini}{17pt}
{
\renewcommand{\theenumi}{{\arabic{enumi}})} 
\renewcommand{\labelenumi}{\theenumi}
\begin{enumerate}
\item
\label{aaaaa22}
		If $G$ is {\rm 0}-continuous and $C^0$-semiregular, then $\EVE_{[r,r']}^0$ is smooth for each $[r,r']\in \COMP$ \deff $\mg$ is integral complete \deff $\EVE_{[0,1]}^0$ is differentiable at zero.
\item
\label{aaaaa12}
	If $G$ is {\rm k}-continuous and $C^k$-semiregular for $k\in \NN_{\geq 1}\sqcup\{\infty\}$, then $\EVE_{[r,r']}^k$ is smooth for each $[r,r']\in \COMP$ \deff $\mg$ is Mackey complete \deff $\EVE_{[0,1]}^k$ is differentiable at zero.
\end{enumerate}}
\endgroup
\noindent
Here, for $k=0$ in the first-, and $k\in \NN_{\geq 1} \sqcup\{\infty\}$ in the second case, we have 
\begin{align}
\label{poasopsaosappoasposa}
	\textstyle\big(\dd_\phi\:\EVE_{[r,r']}^k\big)(\psi)\textstyle 
	=\dd_e\LT_{\innt \phi}\big(\int \Ad_{[\innt_r^s\phi]^{-1}}(\psi(s))\:\dd s\big)\qquad\quad\forall\: \phi,\psi\in C^k([r,r'],\mg),\:\: [r,r']\in \COMP.
\end{align}
\end{customthm}
\begin{proof}[Sketch of the Proof given in Sect.\ \ref{osposdopopsd}]
	By Corollary \ref{lkflkf}, it suffices to show that (under the given presumptions) $\EVE_{[r,r']}^k$ is smooth and  fulfills \eqref{poasopsaosappoasposa}, namely
	\begingroup
\setlength{\leftmargini}{12pt}
\begin{itemize}
\item
	for $k\equiv 0$\hspace{46.5pt}\quad  
	if $\mg$ is integral complete,
\item
	for $k\in \NN_{\geq 1}\sqcup\{\infty\}$\quad 
	if $\mg$ is Mackey complete.
\end{itemize}
\endgroup
\noindent
	Now, in both cases, formula \eqref{poasopsaosappoasposa} is immediate from Theorem \ref{E} applied to
\begin{align*}
	\Phi[\phi,\psi]\colon  (0,1)\times [r,r']\ni(h,t)\mapsto \phi(t)+ h\cdot \psi(t)\qquad\quad\forall\: \phi,\psi\in C^k([r,r'],\mg),
\end{align*}	
	whereby the right side of \eqref{poasopsaosappoasposa} is easily seen to be continuous (cf.\ Lemma \ref{fdfdfdf}). It thus follows from Theorem E in \cite{HGGG} that $\EVE_{[0,1]}^k$ is smooth. Then, smoothness of $\EVE_{[r,r']}^k$ for $[r,r']\in \COMP$ is clear from
	\begin{align*}
		\textstyle\EVE_{[r,r']}^k\stackrel{\eqref{substi}}{=}\EVE_{[0,1]}^k\cp \:\eta,
	\end{align*}
	for $\eta\colon C^k([r,r'],\mg)\rightarrow C^k([0,1],\mg)$ given by 
		\begin{align*} 
		&\eta(\phi)\mapsto \dot\varrho\cdot \rcK{\phi\cp\varrho}\equiv |r'-r|\cdot \rcK{\phi\cp\varrho} \\
&\text{with}\quad \varrho\colon [0,1]\rightarrow [r,r'],\quad t\mapsto r+ t\cdot |r'-r|,
	\end{align*}
as $\eta$ is evidently smooth.
\end{proof}
\begin{customrem}{IV}\label{khdsdsdsdsshdshsdjhsdhjsd}
\noindent
\vspace{-5pt}
\begingroup
\setlength{\leftmargini}{12pt}
\begin{itemize}
\item
Up to the point where Theorem E from \cite{HGGG} is applied, the above argument also works for the Lipschitz case (cf.\ Corollary \ref{sddsdsdsds}); i.e., we have:
\begingroup
\setlength{\leftmarginii}{20pt}
\begin{enumerate}
\item[2')]
If $G$ is $\lip$-continuous and $C^\lip$-semiregular, then  $\EVE^\lip_{[r,r']}$ is of class $C^1$ for each $[r,r']\in \COMP$ \deff $\mg$ is Mackey complete \deff $\EVE^\lip_{[0,1]}$ is differentiable at zero.	
\end{enumerate}
\endgroup
\item
Then, instead of using Theorem E from \cite{HGGG} in the above argument, one might use the explicit formula \eqref{poasopsaosappoasposa} to prove smoothness of $\EVE_{[0,1]}^k$ inductively for $k\in \NN\sqcup\{\lip,\infty\}$, which would strengthen the statement in the previous point of course.  
The details, however, seem to be quite elaborate and technical; so that we leave this issue to another paper.
\end{itemize}
\endgroup
\end{customrem}
\noindent
Now, Theorem \ref{F} shows:
\begingroup
\setlength{\leftmargini}{17pt}
{
\renewcommand{\theenumi}{{\Alph{enumi}})} 
\renewcommand{\labelenumi}{\theenumi}
\begin{enumerate}
\item
\label{sasaassasasa1}
	$G$ is $C^0$-regular \deff $G$ is $C^0$-semiregular and 0-continuous, with $\mg$ integral complete.	
\item
\label{sasaassasasa2}
$G$ is $C^k$-regular for $k\in \NN_{\geq 1}\sqcup \{\infty\}$ \deff $G$ is $C^k$-semiregular and k-continuous, with $\mg$ Mackey complete.
\end{enumerate}}
\endgroup
\noindent
Here, \ref{sasaassasasa1} generalizes Theorem F in \cite{HGGG} stating that $G$ is $C^0$-regular if $G$ is $C^0$-semiregular and 0-continuous with integral complete Lie algebra, such that there exists a point-separating family $(\alpha_j)_{j\in J}$ of smooth Lie group homomorphisms $\alpha_j\colon G\rightarrow H_j$ to $C^0$-regular Lie groups $H_j$. 
\vspace{6pt}

\noindent
Moreover, let us say that $G$ admits a $C^1$-exponential map \deff $\exp$ as defined in \eqref{pofdpofdpofdpofdfpodfpopofdpofd} is of class $C^1$. 
We then have
\begin{customlem}{III}\label{lkfdlkfdlkfd}
Suppose that $G$ is abelian. Then, 
\begingroup
\setlength{\leftmargini}{17pt}
{
\renewcommand{\theenumi}{{\arabic{enumi}})} 
\renewcommand{\labelenumi}{\theenumi}
\begin{enumerate}
\item
\label{sasasasasasadd1}
$G$ is $C^0$-regular\hspace{6pt} \deff\he\he $G$ admits a $C^1$-exponential map, and $\mg$ is integral complete.
\item
\label{sasasasasasadd2}
$G$ is $C^\infty$-regular \he\he \deff \he\he $G$ admits a $C^1$-exponential map, and $\mg$ is Mackey complete

\hspace{80pt} \deff \he\he $G$ is $C^k$-regular for each $k\in \NN_{\geq 1}\sqcup \{\lip,\infty\}$.
\end{enumerate}}
\endgroup
\end{customlem}
\begin{proof}
	If $G$ is abelian and $\exp\colon \mg\rightarrow G$ is of class $C^1$, then we have \rccf{Remark \ref{exponentialmap}.\ref{exponentialmap3}}
	\begin{align*}
		\textstyle\innt \phi=\exp(\int \phi(s)\: \dd s)\qquad\text{for each}\qquad \phi\in C^0([0,1],\mg)\qquad\text{with}\qquad\int_0^t \phi(s)\: \dd s\in \mg\quad\:\:\forall\: t\in [0,1];   
	\end{align*}		
	which is obviously continuous w.r.t.\ the seminorms $\{\ppp_\infty\}_{\pp\in \SEM}$. It is thus clear that $G$ is
	\begingroup
\setlength{\leftmargini}{12pt}
\begin{itemize}
\item
$C^0$-semiregular and 0-continuous if $\mg$ is integral complete; thus, $C^0$-regular by \ref{sasaassasasa1}.
\item
$C^k$-semiregular and k-continuous for each $k\in \NN_{\geq 1}\sqcup \{\lip,\infty\}$ if $\mg$ is Mackey complete; thus, $C^k$-regular for each $k\in \NN_{\geq 1}\sqcup \{\lip,\infty\}$ by \ref{sasaassasasa2}.
\end{itemize}
\endgroup
\noindent	
Since $\exp$ is of class $C^1$ if $G$ is $C^\infty$-regular (cf.\ Remark \ref{exponentialmap}.\ref{exponentialmap2}), the rest is clear from \ref{sasaassasasa1} and \ref{sasaassasasa2}.
\end{proof}
Then, using Proposition V.1.9 in \cite{KHNM}, we obtain:
\begin{customtpr}{V}\label{lkflkfdfdf}
	Suppose that $G$ is connected and abelian. 
	Then, 
\begingroup
\setlength{\leftmargini}{17pt}
{
\renewcommand{\theenumi}{{\arabic{enumi}})} 
\renewcommand{\labelenumi}{\theenumi}
\begin{enumerate}
\item
\label{fdfdfdfdfdfdfdllllq1}
$G$ is \emph{$C^0$-regular} \hspace{2.4pt} \deff \he\he $G\cong E\slash \Gamma$ holds for a discrete subgroup $\Gamma\subseteq E$, with $E$ integral complete 

\hspace{82.1pt}					\deff \he\he $G$ admits a $C^1$-exponential map, and $E$ is integral complete. 
\item
\label{fdfdfdfdfdfdfdllllq2}
$G$ is \emph{$C^\infty$-regular} \he\he \deff \he\he $G\cong E\slash \Gamma$ holds for a discrete subgroup $\Gamma\subseteq E$, with $E$ \emph{Mackey complete}

\hspace{80.2pt}	\he\he				\deff \he\he $G$ admits a $C^1$-exponential map, and $E$ is Mackey complete 

\hspace{80.2pt}	\he\he				\deff \he\he $G$ is $C^k$-regular for each $k\in \NN_{\geq 1}\sqcup\{\lip,\infty\}$.
\end{enumerate}}
\endgroup
\end{customtpr}
\begin{proof}
Observe that $E$ is integral/Mackey complete \deff $\mg$ is integral/Mackey complete; and that, by Lemma \ref{lkfdlkfdlkfd}, it suffices to show the equivalences in the first line of \ref{fdfdfdfdfdfdfdllllq1} and \ref{fdfdfdfdfdfdfdllllq2}:
	\begingroup
\setlength{\leftmargini}{12pt}
\begin{itemize}
\item
If $G$ is $C^k$-regular for $k\in \{0,\infty\}$, then $G$ is $C^\infty$-regular. Then, $\exp$ is smooth \rccf{Remark \ref{exponentialmap}.\ref{exponentialmap2}}; and, by \ref{sasaassasasa1} and \ref{sasaassasasa2}, $E$ is Mackey complete -- even integral complete for $k\equiv 0$. Consequently, $G\cong E\slash \Gamma$ holds for a discrete subgroup $\Gamma\subseteq E$, by Proposition V.1.9 in \cite{KHNM}.
\item
Suppose that $G= E\slash \Gamma$ holds for a discrete subgroup $\Gamma\subseteq E=\mg$; and let $k\in \{0,\infty\}$ be fixed. Suppose furthermore that 
$E$ is Mackey complete for $k\equiv \infty$, and integral complete for $k\equiv 0$. Then, the evolution map \eqref{podspodspopodspodpods} (for $[r,r']\equiv [0,1]$) of $(E,+)$ is given by
\begin{align*}
	\textstyle\innt^k_{\!E}\colon C^k([0,1],E)\rightarrow E,\qquad \phi \mapsto \int \phi(s)\:\dd s,
\end{align*}
which is obviously smooth w.r.t.\ the $C^0$-topology (as it is linear and continuous therein); and the canonical projection $\pi\colon E\rightarrow E\slash \Gamma$ is a smooth Lie group homomorphism -- confer Example \ref{exxxcon}.\ref{exxxcon0} for more details concerning the Lie group structure on $E\slash\Gamma$. We thus have (confer, e.g., statement \ref{homtausch} in Sect.\ \ref{opsopdsospdosdpdsoppods})
\begin{align*}
	\textstyle\EVE_{[0,1]}^k (\phi)=\big(\pi\cp \innt^k_{\!E}\big)(\phi)\qquad\quad\forall\:\phi\in C^k([0,1],E);
\end{align*}
which is evidently smooth w.r.t.\ the $C^0$-topology. It is thus clear that $G$ is $C^k$-regular.
\end{itemize}
\endgroup
\noindent	
The claim now follows from Lemma \ref{lkfdlkfdlkfd}.
\end{proof}
\noindent
Evidently, Proposition \ref{lkflkfdfdf} generalizes Theorem C.(b),(c) in \cite{HGGG} stating that $(E,+)$ is $C^0$-regular \deff $E$ is integral complete; and that $(E,+)$ is $C^1$-regular \deff $E$ is Mackey complete.

\section{Preliminaries}
\label{prelim}
In this section, we fix the notations; and recall the most important facts concerning locally convex vector spaces, differentiable maps, and Lie groups that we will need in the main text. 
\subsection{Conventions}
Intervals are non-empty, non-singleton, connected subsets of $\RR$. In the following, $D$ always denotes an arbitrary-, $I$ an open-, and $K$ a compact interval. The set of all intervals is denoted by $\INT$, and the set of all compact ones by $\COMP$. 
Let $F$ be a (Hausdorff) locally convex vector space with corresponding system of continuous seminorms $\SEMM$. We recall that $\SEMM$ is filtrating, i.e., that for $\qq_1,\dots, \qq_n\in \SEMM$ with $n\geq 1$ given, there exists some $\qq\in \SEMM$ with $\qq_1,\dots,\qq_n\leq \qq$. For $\varepsilon>0$ and $\qq\in \SEMM$, we define
\begin{align*}
	\B_{\qq,\varepsilon}:=\{X\in F\:|\: \qq(X)<\varepsilon\}\qquad\qquad\qquad \OB_{\qq,\varepsilon}:=\{X\in F\:|\: \qq(X)\leq\varepsilon\};
\end{align*}
and write $\qq\lleq V$ (or $V\ggeq \qq$) for $\qq\in \SEMM$ and $V\subseteq F$ \defff \he$\OB_{\qq,1}\subseteq V$ holds. We say that $\bound\subseteq F$ is bounded \defff it is von Neumann bounded, i.e., \defff we have
\begin{align*}
	\sup\{\qq(X)\:|\: X\in \bound\}<\infty\qquad\quad\forall\:\qq\in \SEMM.
\end{align*} 
We let $\comp{F}$ denote the completion of $F$; as well as $\cqq$ the (unique) extension of $\qq\in \SEMM$ to $\comp{F}$. 
\vspace{6pt}

Manifolds and Lie groups are always assumed to be in the sense of \cite{HG} (cf.\ Definition 3.1 and Definition 3.3 in \cite{HG}); i.e., smooth, Hausdorff, and modeled over a Hausdorff locally convex vector space: The corresponding differential calculus is reviewed in Sect.\ \ref{dAINT}. 
If $f\colon M\rightarrow N$ is a $C^1$-map between the manifolds $M$ and $N$, then $\dd f\colon TM \rightarrow TN$ denotes the corresponding tangent map between their tangent manifolds; and we write $\dd_xf\equiv\dd f(x,\cdot)\colon T_xM\rightarrow T_{f(x)}N$ for each $x\in M$. 
A curve is a continuous map $\gamma\colon D\rightarrow M$, where $M$ is a manifold and $D\in \INT$ an interval. If $D\equiv I$ is open, then $\gamma$ is said to be of class $C^k$ for $k\in \NN\sqcup \{\infty\}$ \defff it is of class $C^k$ when considered as a map between the manifolds $I$ and $M$. 
We say that $\gamma\colon D\rightarrow M$ is of class $C^k$ for $k\in \NN\sqcup \{\infty\}$ --  and write $\gamma\in C^k(D,M)$ -- \defff $\gamma=\wt{\gamma}|_D$ holds for some $\wt{\gamma}\colon I\rightarrow M$ of class $C^k$ with $D\subseteq I$.  
If $\gamma\colon D\rightarrow M$ is of class $C^1$ (or differentiable), 
we let 
$\dot\gamma(t)\in T_{\gamma(t)}M$ denote the corresponding tangent vector at $\gamma(t)\in M$. 
The same conventions also hold if $M\equiv F$ is a Hausdorff locally convex vector space -- In this case, 
	we let $C^\lip([r,r'],F)$ denote the set of all Lipschitz curves on $[r,r']\in \COMP$; i.e., all curves $\gamma\colon [r,r']\rightarrow F$ with
\begin{align*}
	\qq(\gamma(t)-\gamma(t'))\leq L_\qq\cdot |t-t'|\qquad\quad\forall\: t,t'\in [r,r'],\:\: \qq\in \SEMM
\end{align*}
for certain Lipschitz constants $\{L_\qq\}_{\qq\in \SEMM}\subseteq \RR_{\geq 0}$. 
We let 
$\infty +1:=\infty$ as well as $\lip+1:=1$; and, for $k\in \NN\sqcup\{\lip,\infty\}$ and $[r,r']\in \COMP$, we define
\begin{align*}
	\qq^\dind(\gamma):=\qq\big(\gamma^{(\dind)}\big),\qquad\:\: \qq^\dind_\infty(\gamma):=\sup\big\{\qq\big(\gamma^{(m)}(t)\big)\:\big|\: 0\leq m\leq \dind,\:\:t\in [r,r']\big\},\qquad\:\: \qq_\infty:=\qq_\infty^0
\end{align*}
for each $s\llleq k$ and $\gamma\in C^k([r,r'],F)$ -- Here, $s\llleq k$ means 
\begingroup
\setlength{\leftmargini}{12pt}
\begin{itemize}
\item
	$\dind\leq k$\hspace{2pt}\: for\: $k\in \NN$, 
\item
	$\dind=0$\:\: for\: $k\equiv\lip$,
	\item
	$\dind\in \NN$\hspace{1.2pt}\: for\: $k\equiv\infty$.
\end{itemize}
\endgroup
\noindent
The $C^k$-topology on $C^k([r,r'],F)$ is the Hausdorff locally convex topology that is generated by the seminorms $\qq_\infty^\dind$, for each $\qq\in \SEMM$ and $\dind\llleq k$.
\vspace{6pt}

\noindent
In this paper, $G$ will always denote an infinite dimensional Lie group (in Milnor's sense) that is modeled over a Hausdorff locally convex vector space $E$, with system of continuous seminorms $\SEM$. We denote the Lie algebra of $G$ by $(\mg,\bl\cdot,\cdot\br)$, fix a chart $\chart\colon G\supseteq\U\rightarrow \V\subseteq E$ with $\V$ convex, $e\in \U$, $\chart(e)=0$; and identify $\mg$ with $E$ via $\dd_e\chart\colon E\rightarrow \mg$ -- specifically, we define 
\begin{align*}
	\SEML:=\{\ppp\equiv\pp\cp\dd_e\chart\colon \mg\rightarrow \RR_{\geq 0}\: |\: \pp\in \SEM\}.
\end{align*}
We denote the inversion and the Lie group multiplication by
\begin{align*}
\inv\colon G\rightarrow G,\qquad g\mapsto g^{-1}\qquad\qquad\text{and}\qquad\qquad\mult\colon G\times G\rightarrow G,\qquad (g,g')\mapsto g\cdot g',
\end{align*}
respectively, say that $\mathrm{A}\subseteq G$ is symmetric \defff $\inv(\mathrm{A})=\mathrm{A}$ holds; and recall the product rule\footnote{Confer, e.g., \ref{productrule} in Sect.\ \ref{Diffmaps}.} 
\begin{align}
\label{LGPR}
	\dd_{(g,h)}\mult(v,w)= \dd_g\RT_h(v) + \dd_h\LT_g(w)\qquad\quad\forall\: g,h\in G,\:\: v\in T_gG,\:\: w\in T_h G.
\end{align} 
We let $\conj\colon G\times G\ni (g,h)\mapsto \conj_g(h)\in G$ with
	\begin{align*}
		\conj_g:=\LT_g\cp\RT_{g^{-1}}\qquad\text{for}\qquad \RT_g:=\mult(\cdot,g)\qquad \text{and}\qquad \LT_g:=\mult(g,\cdot) \qquad\quad\:\:\forall\: g\in G,
	\end{align*} 
define $\Ad_g:=\dd_e \conj_g\colon \mg\rightarrow \mg$ for each $g\in G$; and let $\Ad\colon G\times \mg\ni (g,X)\mapsto \Ad_g(X)\in \mg$ denote the adjoint action. 
We furthermore let
\begin{align*}
	\ad\:X(Y):=\dd_e\Ad\bl Y\br(X)\quad\:\:\forall\: X\in \mg \qquad\quad\:\:\text{for}\qquad\quad\:\: \Ad\bl Y\br\colon G\ni g\mapsto \Ad_g(Y)\in \mg;
\end{align*}
and recall that 
$\ad\:X(Y)=\bl X,Y\br$ holds for each $X,Y\in \mg$.
\subsection{Locally Convex Vector Spaces}
Let $F_1,\dots,F_n$ be (Hausdorff) locally convex vector spaces with corresponding system of continuous seminorms $\SEMM_1,\dots,\SEMM_n$. Obviously, the Tychonoff topology on $F:=F_1\times{\dots}\times F_n$ is the (Hausdorff) locally convex topology that is generated by the seminorms
\begin{align}
\label{rttrrttrtr}
	\mm[\qq_1,\dots,\qq_n]\colon F\ni (X_1,\dots,X_n)\mapsto \max\{\qq_k(X_k)\:|\: k=1,\dots,n\},
\end{align}
with $\qq_k\in \SEMM_k$ for $k=1,\dots,n$. Let $E$ be a further locally convex vector space with system of continuous seminorms $\SEM$. We then have
\begin{lemma}
\label{kldskldsksdklsdl}
	Let $X$ be a topological space; and let $\Phi\colon X\times F_1\times{\dots}\times F_n\rightarrow E$ be continuous with $\Phi(x,\cdot)$ 
	$n$-multilinear for each $x\in X$. Then, for each $x\in X$ and $\pp\in \SEM$, there exist seminorms $\qq_1\in \SEMM_1,\dots,\qq_n\in \SEMM_n$ as well as $V\subseteq X$ open with $x\in V$, such that 
	\begin{align*}
		(\pp\cp\Phi)(y,X_1,\dots,X_n) \leq \qq_1(X_1)\cdot {\dots}\cdot \qq_n(X_n)\qquad\quad\forall\: y\in V 
	\end{align*} 
	holds for all $X_1\in F_1,\dots,X_n\in F_n$.
\end{lemma} 
\begin{proof}
The proof is elementary, and can be found in Appendix \ref{appA1}.
\end{proof}
\begin{corollary}
\label{ofdpopfdofdp}
	Let $X$ be a topological space; and let $\Phi\colon X\times F_1\times{\dots}\times F_n\rightarrow E$ be continuous with $\Phi(x,\cdot)$ 
	$n$-multilinear for each $x\in X$. Then, for each compact $\compacto\subseteq X$ and each $\pp\in \SEM$, there exist seminorms $\qq_1\in \SEMM_1,\dots,\qq_n\in \SEMM_n$ as well as $O\subseteq X$ open with $\compacto\subseteq O$, such that
	\begin{align}
	\label{oaosapop}
		(\pp\cp\Phi)(y,X_1,\dots,X_n) \leq \qq_1(X_1)\cdot {\dots}\cdot \qq_n(X_n)\qquad\quad\forall\: y\in O
	\end{align} 
	holds for all $X_1\in F_1,\dots,X_n\in F_n$.
\end{corollary}
\begin{proof}
	The proof is elementary, and can be found in Appendix \ref{appA2}.
\end{proof}
Let us finally recall the following standard result concerning completions.
\begin{lemma}
\label{pofsdisfdodjjxcycxj}
	Let $F_1,\dots,F_n,E$ be Hausdorff locally convex vector spaces; and let $\Phi\colon F_1\times{\dots}\times F_n\rightarrow E$  be continuous and $n$-multilinear. Then, $\Phi$ extends uniquely to a continuous $n$-multilinear map
	$\comp{\Phi}\colon \comp{F}_1\times {\dots}\times \comp{F}_n\rightarrow \comp{E}$. 
\end{lemma}

\subsection{Differentiation and Integrals}
\label{dAINT}
In this subsection, we recall the differential calculus from \cite{HA,HG,MIL,KHN}; and provide some facts that we will need to work efficiently in the main text.

\subsubsection{Differentiable Maps}
\label{Diffmaps}
Let $E$ and $F$ be Hausdorff locally convex vector spaces with systems of continuous seminorms $\SEM$ and $\SEMM$, respectively.  
Let $U\subseteq F$ be open, and $f\colon U\rightarrow E$ be a map.

We say that $f$ is differentiable at $x\in U$ \defff 
\begin{align*}
	\textstyle(D_v f)(x):=\lim_{t\rightarrow 0} \rcf{t}\cdot (f(x+t\cdot v)-f(x))\in E
\end{align*} 
exists for each $v\in F$. Moreover, 
\begingroup
\setlength{\leftmargini}{12pt}
\begin{itemize}
\item
$f$ is said to be differentiable \defff it is differentiable at each $x\in U$; i.e., \defff  $D_v f\colon U\rightarrow E$ is defined for each $v\in F$.
\item
$f$ is said to be $k$-times differentiable for $k\geq 1$ \defff 
	\begin{align*}
	D_{v_k,\dots,v_1}f\equiv D_{v_k}(D_{v_{k-1}}( {\dots} (D_{v_1}(f))\dots))\colon U\rightarrow E
\end{align*}
is defined for each $v_1,\dots,v_k\in F$; implicitly meaning that $f$ is $p$-times differentiable for each $1\leq p\leq k$. In this case, we define
\begin{align*}
	\dd^p_xf(v_1,\dots,v_p)\equiv \dd^p f(x,v_1,\dots,v_p):=D_{v_p,\dots,v_1}f(x)\qquad\quad\forall\: x\in U,\:\:v_1,\dots,v_p\in F,
\end{align*} 	
for $p=1,\dots,k$; and let $\dd f\equiv \dd^1 f$, as well as $\dd_x f\equiv \dd^1_x f$ for each $x\in U$.
\end{itemize}
\endgroup
\noindent
Then,
\begingroup
\setlength{\leftmargini}{12pt}
\begin{itemize}
\item
	$f$ is said to be of class $C^0$ \defff it is continuous; and we let $\dd^0 f\equiv f$ in this case.
\item
	$f$ is said to be of class $C^k$ for $k\geq 1$ \defff it is $k$-times differentiable, such that 
\begin{align*}
	\dd^pf\colon U\times F^p\rightarrow E,\qquad (x,v_1,\dots,v_p)\mapsto D_{v_p,\dots,v_1}f(x)
\end{align*} 
is continuous for $p=0,\dots,k$. 

In this case, $\dd^p_x f$ is symmetric and $p$-multilinear for each $x\in U$ and $p=1,\dots,k$, cf.\ \cite{HG}.  
\item
$f$ is said to be of class $C^\infty$ \defff it is of class $C^k$ for each $k\in \NN$. 
\end{itemize}
\endgroup
\noindent
We have the following differentiation rules, cf.\ \cite{HG}:
\begingroup
\setlength{\leftmargini}{17pt}
{
\renewcommand{\theenumi}{{\bf \alph{enumi}})} 
\renewcommand{\labelenumi}{\theenumi}
\begin{enumerate}
\item
\label{iterated}
A map $f\colon F\supseteq U\rightarrow E$ is of class $C^k$ for $k\geq 1$ \defff $\dd f$ is of class $C^{k-1}$ when considered as a map $F' \supseteq U' \rightarrow E$ for $F'\equiv F \times F$ and $U'\equiv U\times F$.
\item
\label{linear}
If $f\colon F\rightarrow E$ is linear and continuous, then $f$ is smooth; with $\dd^1_xf=f$ for each $x\in F$, as well as $\dd^kf=0$ for each $k\geq 2$.  
\item
\label{speccombo}
Let $E_1,\dots,E_m$ be Hausdorff locally convex vector spaces; and 
$f_u\colon F\supseteq U\rightarrow E_u$ be of class $C^k$ for $k\geq 1$ and $u=1,\dots,m$. Then, 
\begin{align*}
	f=f_1\times{\dots}\times f_m \colon U\rightarrow E_1\times{\dots}\times E_m,\qquad x\mapsto (f_1(x),\dots,f_m(x))
\end{align*}
if of class $C^k$ with $\dd^p f=\dd^pf_1{\times} \dots\times \dd^pf_m$ for $p=1,\dots,k$.
\item
\label{chainrule}
	Suppose that $f\colon F\supseteq U\rightarrow U'\subseteq F'$ and $\rcfstrich \colon F'\supseteq U'\rightarrow  F''$ are of class $C^k$ for $k\geq 1$, for Hausdorff locally convex vector spaces $F,F',F''$. Then, $\rcfstrich\cp f\colon U\rightarrow F''$ is of class $C^k$ with 
	\begin{align*}
		\dd_x(\rcfstrich\cp f)=\dd_{f(x)}\rcfstrich\cp \dd_x f\qquad\quad \forall\: x\in U.
	\end{align*}
\item
\label{productrule}
	Let $F_1,\dots,F_m,E$ be Hausdorff locally convex vector spaces, and $f\colon F_1\times {\dots} \times F_m\supseteq U\rightarrow E$ be of class $C^0$. Then, $f$ is of class $C^1$ \defff the ``partial derivatives''
	\begin{align*}
		\partial_u f \colon U\times F_u\ni((x_1,\dots,x_m),v_u)&\textstyle\mapsto \lim_{t\rightarrow 0} \rcf{t}\cdot (f(x_1,\dots, x_u+t\cdot v_u,\dots,x_m)-f(x_1,\dots,x_m))
	\end{align*}
	exist in $E$ and are continuous, for $u=1,\dots,m$. In this case, we have
	\begin{align*}
		\textstyle\dd_{(x_1,\dots,x_m)} f(v_1,\dots,v_m)&\textstyle=\sum_{u=1}^m\partial_u f((x_1,\dots,x_m),v_u)\\
		&\textstyle= \sum_{u=1}^m \dd f((x_1,\dots,x_m),(0,\dots,0, v_u,0,\dots,0))
	\end{align*}
	for all $(x_1,\dots,x_m)\in U$, and $v_u\in F_u$ for $u=1,\dots,m$.
\end{enumerate}}
\endgroup
\noindent
Finally, for $f\colon F\supseteq U\rightarrow E$ of class $C^{k}$ for $k\geq 1$, we have Taylor's formula, cf.\ \cite{HG}
\begin{align}
\label{Taylor}
\begin{split}
 \textstyle f(x+\Delta)= f(x)+ \dd^1_x f(\Delta) + {\dots} &\textstyle+ \frac{1}{(k-1)!} \cdot\dd^{k-1}_xf(\Delta,\dots,\Delta) \\
 &\textstyle+ \frac{1}{(k-1)!}\cdot \int_0^1 (1-s)^{k-1} \cdot \dd^k_{x+s\cdot \Delta}f(\Delta,\dots,\Delta)\:\dd s
\end{split}
\end{align}
for each $x\in U$ and $\Delta\in F$ with $x+[0,1]\cdot \Delta\subseteq U$.
Here, $\int\dd s$ denotes the Riemann integral, discussed in Sect.\ \ref{lkdslkdslkdslkdslkdslkds} below.
\subsubsection{Differentiable Curves}
\label{sdsdsddsdssdsdaaaa}
We now consider the situation where $f\equiv\gamma\colon I\rightarrow E$ holds -- i.e., we have $F\equiv\RR$, and $U\equiv I$ is an open interval. It is then not hard to see that $\gamma$ is of class $C^k$ for $k\geq 1$ \defff $\gamma^{(p)}$, inductively defined by $\gamma^{(0)}:=\gamma$ and\footnote{We have $\gamma^{(p)}(t)=\dd^p_t\gamma(1,\dots,1)$ for $p=1,\dots,k$, $t\in I$.} 
\begin{align*}
	\gamma^{(p)}(t):=\textstyle\lim_{h\rightarrow 0}\rcf{h}\cdot (\gamma^{(p-1)}(t+h)-\gamma^{(p-1)}(t))\qquad\quad\forall\: t\in I,\:\:p=1,\dots,k, 
\end{align*}
exists and is continuous for $p=0,\dots,k$. Then, 
for $\gamma\in C^k(D,E)$ with extension $\wt{\gamma}\colon D\supseteq I\rightarrow E$, we define 
	$\gamma^{(p)}:=\wt{\gamma}^{(p)}|_D$ for $p=0,\dots,k$, and let $\dot\gamma\equiv\wt{\gamma}^{(1)}|_D$. 
\begin{lemma}
\label{aaasasdswewe}
Suppose that $\gamma\in C^k(D,E)$ holds for $k\geq 1$ and $D\in \INT$. Then, $\gamma$ is of class $C^{k+1}$ \deff $\gamma$ is of class $C^k$ with $\gamma^{(k)}$ of class $C^1$.
\end{lemma}
\begin{proof}
The proof is elementary, and can be found in Appendix \ref{appdiffssddsssdds}.
\end{proof}
\begin{lemma}
\label{sddsdssd}
Let $F_1,F_2,E$ be Hausdorff locally convex vector spaces; and $\gamma_i\colon D\rightarrow W_i\subseteq F_i$ be of class $C^k$ for $i=1,2$, for some $k\geq 1$. Suppose furthermore that $\Omega\colon W_1\times W_2\rightarrow E$ is smooth. 
Then, $\delta\colon D\ni t\mapsto \Omega(\gamma_1(t),\gamma_2(t))$ is of class $C^k$; and  $\delta^{(p)}$, for $0\leq p \leq k$, can be written as a finite sum of terms of the form
\begin{align}
\label{lkdslkdslkdslkdsds}
	\alpha=\Psi\big(\gamma^{(z_1)}_{i_1},\dots,\gamma_{i_m}^{(z_m)}\big)\qquad\text{for some}\qquad 0\leq z_1,\dots,z_m\leq p,\:\: 1\leq i_1,\dots,i_m \leq 2,\:\:m\geq 2,
\end{align}
where $\Psi\colon V\equiv V_{i_1}\times {\dots}\times V_{i_m}\rightarrow E$ is smooth with open neighbourhoods $V_{i_u}\subseteq F_{i_u}$ for $u=1,\dots,m$.
\end{lemma}
\begin{proof}
The proof is elementary, and can be found in 
Appendix \ref{appdiff}. 
\end{proof}
\begin{corollary}
\label{bhsbsshshdkksjdhjsd}
Let $F,E$ be Hausdorff locally convex vector spaces; and suppose that $\gamma_1\colon D\rightarrow W_1\subseteq E$ is of class $C^1$, $\gamma_2\colon D\rightarrow W_2\subseteq F$ is of class $C^k$ for some $k\geq 1$, and that $\dot\gamma_1=\Omega(\gamma_1,\gamma_2)$ holds for a smooth map $\Omega\colon W_1\times W_2\rightarrow E$. Then, $\gamma_1$ is of class $C^{k+1}$.
\end{corollary}
\begin{proof}
	This follows inductively from Lemma \ref{aaasasdswewe} and Lemma \ref{sddsdssd}. 
\end{proof}
\begin{corollary}
\label{fddfd}
Let $F_1,F_2,E$ be Hausdorff locally convex vector spaces; and $\gamma_i\colon D\rightarrow W_i\subseteq F_i$ be of class $C^k$ for $i=1,2$, for some $k\geq 1$. Suppose furthermore that $\Omega\colon W_1\times F_2\rightarrow E$ is smooth, as well as linear in the second argument. Then, $\delta\colon D\ni t\mapsto \Omega(\gamma_1,\gamma_2)$ is of class $C^k$; and $\delta^{(p)}$, for $1\leq p\leq k$, can be written as a finite sum of terms of the form\footnote{Evidently, $[\partial_1]^m\Omega$ is continuous, as well as  multilinear in the last $m+1$ arguments.}
\begin{align*}
	([\partial_1]^m\Omega)\big(\gamma_1,\gamma^{(z_1)}_{1},\dots,\gamma_{1}^{(z_m)},\gamma_2^{(q)}\big)\qquad\quad\text{for certain}\qquad\quad 0\leq z_1,\dots,z_m,q\leq p,\:\: m\geq 1.
\end{align*} 
\end{corollary}
\begin{proof}
	Lemma \ref{sddsdssd} shows that $\delta$ is of class $C^k$; and the rest follows inductively from \ref{linear}, \ref{chainrule}, \ref{productrule}.
\end{proof}
\begin{lemma}
\label{oopxcxopcoxpopcx}
Let $F_1,F_2,E$ be Hausdorff locally convex vector spaces with systems of continuous seminorms $\SEMM_1,\SEMM_2,\SEM$. Suppose that $W_1\subseteq F_1$ is open; and that $\Omega\colon W_1\times F_2\rightarrow E$ is smooth, as well as linear in the second argument. Then, the following statements hold: 
\begingroup
\setlength{\leftmargini}{15pt}
{
\renewcommand{\theenumi}{{\arabic{enumi}})} 
\renewcommand{\labelenumi}{\theenumi}
\begin{enumerate}
\item
\label{oopxcxopcoxpopcx1}
	For $\pp\in \SEM$ and $\dindu\in \NN$ fixed, there exist $\mm\in \SEMM_1$ and $\qq\in \SEMM_2$, such that for each $[r,r']\in \COMP$ and $\gamma\in C^\dindu([r,r'],W_1)$ with $\mm_\infty^\dindu(\gamma)\leq 1$, we have
\begin{align*}
	\pp^\dindp(\Omega(\gamma,\psi))\leq \qq^\dindp(\psi)\qquad\quad\forall\:\psi\in C^\dindu([r,r'],F_2),\:\: 0\leq \dindp\leq \dindu.
\end{align*}
\item
\label{oopxcxopcoxpopcx2}
	For $\pp\in \SEM$, $\dindu\in \NN$, and $\gamma\in C^\dindu([r,r'],W_1)$ fixed, there exists some $\qq\in \SEMM_2$ with
\begin{align*}
	\pp^\dindp(\Omega(\gamma,\psi))\leq \qq^\dindp(\psi)\qquad\quad\forall\:\psi\in C^\dindu([r,r'],F_2),\:\: 0\leq \dindp\leq \dindu.
\end{align*} 
\end{enumerate}}
\endgroup
\end{lemma}
\begin{proof}
The proof is elementary, and can be found in 
Appendix \ref{appdifff}. 
\end{proof}

\subsubsection{The Riemann Integral}
\label{lkdslkdslkdslkdslkdslkds}
Let $F$ be a Hausdorff locally convex vector space with system of continuous seminorms $\SEMM$, and completion $\comp{F}$. 
We denote the Riemann integral of $\gamma\in C^0([r,r'],F)$ by $\int \gamma(s) \:\dd s\in \comp{F}$; 
and define 
\begin{align}
\label{fdopfdpo}
	\textstyle\int_a^b \gamma(s)\:\dd s:= \int \gamma|_{[a,b]}(s) \:\dd s,\qquad\:\:\int_b^a \gamma(s) \:\dd s:= - \int_a^b \gamma(s) \:\dd s,\qquad\:\:
	 \int_c^c \gamma(s)\: \dd s:=0\qquad
\end{align}
for $r\leq a<b\leq r'$, $c\in [r,r']\in \COMP$. Clearly, the Riemann integral is linear, with	
\begin{align}
\label{absch1}
	\textstyle\int_a^c \gamma(s) \:\dd s&\textstyle=\int_a^b \gamma(s)\:\dd s+ \int_b^c \gamma(s)\:\dd s\qquad\qquad \forall\: r\leq a<b<c\leq r',\\[2pt]
	\label{absch2}
	\textstyle\comp{\qq}\big(\int_r^t \gamma(s) \:\dd s\big)&\textstyle\leq  \int_r^t \qq(\gamma(s))\:\dd s\quad	\hspace{79.8pt}\forall\: t\in [r,r'],\:\:\qq\in \SEMM.
\end{align}
	It is furthermore not hard to see that 
	\begin{align}
	\label{opgfgofppof}
		\textstyle\Gamma\in C^1([r,r'],\comp{F})\qquad\text{with}\qquad\dot\Gamma=\gamma\quad\qquad\text{holds for}\quad\qquad \Gamma\colon [r,r']\ni t\mapsto \int_r^t \gamma(s)\:\dd s.
	\end{align}
	More importantly, we have, cf.\ \cite{HG} 
	 \begin{align}
	\label{isdsdoisdiosd}
		\textstyle\gamma(t)-\gamma(r)=\int_r^t \dot\gamma(s)\:\dd s\in F\qquad\quad\forall\:t\in [r,r'],\:\: \gamma\in C^1([r,r'],F).
	\end{align}
	From this, we obtain 
\begin{align}
\label{substitRI}
	\textstyle\int \gamma(s)\: \dd s=\Gamma(\varrho(\ell'))-\Gamma(\varrho(\ell))\stackrel{\eqref{isdsdoisdiosd}}{=}\int \partial_t (\Gamma\cp \varrho)(s) \: \dd s\stackrel{\ref{chainrule}}{=}\int \dot\varrho(s)\cdot \gamma(\varrho(s))\:\dd s
\end{align}
for each $\gamma\in C^0([r,r'],F)$, and each $\varrho\colon [\ell,\ell'] \rightarrow [r,r']$ of class $C^1$ with $\varrho(\ell)=r$ and $\varrho(\ell')=r'$. Moreover,
	\begin{lemma}
\label{ofdpofdpopssssaaaasfffff}
For each $\gamma\in C^1([r,r'],F)$, we have
\begin{align*}
	\textstyle\qq(\gamma(t)-\gamma(r))\leq \int_r^t \qq(\dot\gamma(s))\: \dd s\qquad\quad\forall\:t\in [r,r'],\:\:\qq\in \SEMM.
\end{align*}
\end{lemma}
\begin{proof}
	Combine \eqref{absch2} with \eqref{isdsdoisdiosd}.  
\end{proof}
\begin{remark}[Banach Spaces]
\label{banachchchch}
Suppose that $E,F$ are Banach spaces; and that $f\colon F\supseteq U\rightarrow E$ is of class $C^{n+1}$ for some $n\geq 1$. Then, using Lemma \ref{kldskldsksdklsdl} and Lemma \ref{ofdpofdpopssssaaaasfffff}, one can show that $f$ is of class $C^n$ in the Fr\'{e}chet sense, cf.\ also \cite{KHN}. 
In particular, if $f$ is of class $C^\infty$, then $f$ is smooth in the Fr\'{e}chet sense.\hspace*{\fill}$\ddagger$
\end{remark}
\begin{lemma}
\label{sdsdds}
Let $F,E$ be Hausdorff locally convex vector spaces; and let $f\colon F\supseteq U\rightarrow E$ be of class $C^{2}$. Suppose that 
$\gamma\colon D\rightarrow F\subseteq \ovl{F}$ is continuous at $t\in D$, such that $\lim_{h\rightarrow 0} \rcf{h}\cdot (\gamma(t+h)-\gamma(t))=:X\in \ovl{F}$ exists. Then, we have
\begin{align*}
	\textstyle\lim_{h\rightarrow 0} \rcf{h}\cdot (f(\gamma(t+h))-f(\gamma(t)))=\ovl{\dd_{\gamma(t)}f}\he(X).
\end{align*}   
\end{lemma}
\begin{proof}
The proof is elementary, and can be found in Appendix \ref{dkldksldkslsdlkklsdkldskl}.
\end{proof}
\noindent
We finally need to discuss the Riemann integral for piecewise continuous curves: 
\begingroup
\setlength{\leftmargini}{12pt}
\begin{itemize}
\item
We let $\CP^0([r,r'],F)$ denote the set of piecewise $C^0$-curves on $[r,r']\in \COMP$; i.e., all $\gamma\colon [r,r']\rightarrow F$  such that there exist $r=t_0<{\dots}<t_n=r'$ as well as $\gamma[p]\in C^0([t_p,t_{p+1}],F)$ for $p=0,\dots,n-1$ with
\begin{align*}
	\gamma|_{(t_p,t_{p+1})}=\gamma[p]|_{(t_p,t_{p+1})}\qquad\quad\forall\: p=0,\dots,n-1.
\end{align*}
\item
We let  
$\COP([r,r'],F)$ denote the set of piecewise constant curves on $[r,r']\in \COMP$; i.e., all $\gamma\colon [r,r']\rightarrow F$, such that there exist $r=t_0<{\dots}<t_n=r'$ as well as $X_0,\dots,X_{n-1}\in F$ with
\begin{align*}
	\gamma|_{(t_p,t_{p+1})}=X_p\qquad\quad\forall\: p=0,\dots,n-1.
\end{align*}  
\end{itemize}
\endgroup
\noindent
We clearly have $\COP([r,r'],F)\subseteq \CP^0([r,r'],F)$; and for $\gamma\in \CP^0([r,r'],F)$ as above, we define
\begin{align}
\label{opofdpopfd}
	\textstyle\int \gamma(s)\:\dd s:=\sum_{p=0}^{n-1}\int \gamma[p](s)\:\dd s. 
\end{align} 
A standard refinement argument in combination with \eqref{absch1} then shows that this is well defined; i.e., independent of any choices we have made. We define $\int_a^b\gamma(s)\: \dd s$ and $\int_c^c \gamma(s)\: \dd s$ as in \eqref{fdopfdpo}; and observe that \eqref{opofdpopfd} is linear and fulfills \eqref{absch1}.

\subsection{Some Estimates for Lie Groups}
In this subsection, we collect some elementary estimates concerning Lie group operations and coordinate changes that will be relevant for our argumentation in the main text. Let thus $G$ be an infinite dimensional Lie group (in Milnor's sense) that is modeled over the Hausdorff locally convex vector space $E$, with system of continuous seminorms $\SEM$ in the following. 
\subsubsection{Lie Group Operations}
We observe that $\Ad\colon G\times \mg\rightarrow \mg$ is smooth (continuous) by \ref{iterated}, because $\conj$ smooth with
\begin{align*}
	\Ad_g(X)= \dd_{(g,e)} \conj(0,X)\qquad\quad\forall\: g\in G,\:\:X\in \mg.
\end{align*}
 We thus obtain from Lemma \ref{kldskldsksdklsdl} and Corollary \ref{ofdpopfdofdp} 
that:
\begingroup
\setlength{\leftmargini}{12pt}
\begin{itemize}
\item
For each $\qq\in \SEM$, there exists some $\qq\leq \nn\in \SEM$, as well as $V\subseteq G$ symmetric open with $e\in V$, such that 
\begin{align}
\label{odspospodpof}
\qqq(\Ad_{g}(X))\leq \nnn(X)\qquad\quad\forall\: g\in V,\:\: X\in \mg.
\end{align} 
\item
For each $\nn\in\SEM$, and each compact $\compact\subseteq G$, there exists some $\nn\leq \mm\in \SEM$, as well as $O\subseteq G$ open with $\compact\subseteq O$, such that 
\begin{align}
\label{askasjkjksaasqwqqwasw}
	\nnn\cp\Ad_g\leq \mmm\qquad\quad\forall\: g\in O.
\end{align}
\end{itemize}
\endgroup
\noindent
Similarly, the maps
\begin{align}
\label{opopopop1}
	\dermapdiff\colon& \V\times E\rightarrow \mg,\qquad (x,X)\mapsto \dd_{\chartinv(x)}\RT_{[\chartinv(x)]^{-1}}(\dd_x\chartinv(X))\\
\label{opopopop2}
	\dermapinvdiff\colon& \V\times \mg\rightarrow E,\qquad (x,X)\mapsto \big(\dd_{\chartinv(x)}\chart\cp \dd_{e}\RT_{\chartinv(x)}\big)(X)
\end{align}
are smooth, as they can be written as
\begin{align}
\label{cmmcmcmcmcmc}
	\dermapdiff(x,X)=\dd_{(x,x)}\dermap(0,X)\qquad\quad\:\:\text{and}\qquad\quad\:\: \dermapinvdiff(x,X)=\dd_{(x,e)}\dermapinv(0,X)
\end{align}
for the smooth maps
\begin{align*}
	\dermap\colon& \V\times \V\rightarrow G,\qquad (x,y)\mapsto \mult(\chartinv(y),[\chartinv(x)]^{-1})\\
\dermapinv\colon&  \V\times \U\rightarrow E,\qquad (x,g)\mapsto (\chart\cp\mult)\big(g,\chartinv(x)\big).
\end{align*}
Thus, by Lemma \ref{kldskldsksdklsdl}, for each $\vv\in \SEM$, there exists some $\V\ggeq \ww\in \SEM$ with $\vv\leq \ww$, such that 
\begin{align}
\label{cpocpoxjdsndscxaaabc}
	\vvv(\dermapdiff(x,X))&\leq \:\ww(X)\qquad\quad\forall\:x\in \OB_{\ww,1}\hspace{2.7pt},\:\:  X\in E\\
\label{dssdsdsdsdsddsds}
	\vv(\dermapinvdiff(x,X)\hspace{0.5pt})&\leq \www(X)\qquad\quad\forall\:x\in \OB_{\www,1},\:\:  X\in \mg.
\end{align}
More generally, we obtain from \ref{iterated} and \ref{productrule} that $\dermapdiff$ is smooth with
\begin{align}
\label{omegakl}
	\dermapdiff[n]:=[\partial_1]^n\dermapdiff\colon \V\times E^{n+1}\rightarrow \mg
\end{align}
continuous as well as multilinear in the last $n+1$ arguments, for each $n\in \NN$. For each $p\in \NN$ and $\vv\in \SEM$, there thus exists some $\V\lleq \ww\in \SEM$ with $\vv\leq \ww$, such that
\begin{align}
\label{omegakll}
	(\vvv\cp\dermapdiff[q])(x,X_1,\dots,X_{q+1})\leq \ww(X_1)\cdot {\dots}\cdot \ww(X_{q+1})
\end{align} 
holds for all $x\in \OB_{\ww,1}$, $X_1,{\dots},X_{q+1}\in E$, and $0\leq q\leq p$.
\vspace{6pt}

\noindent
Finally, since $\inv\colon G\rightarrow G$ is smooth, for each $\mm\in \SEM$, there exists some $\V\lleq \nn\in \SEM$ with $\mm \leq \nn$, such that $\mm\cp\dd_x(\chart\cp\inv\cp\chartinv)\leq \nn$ holds for each $x\in \OB_{\nn,1}$. We thus obtain from Lemma \ref{ofdpofdpopssssaaaasfffff} that
\begin{align}
\label{invrel}
	\mm\cp\chart\cp\inv\cp\chartinv\leq \nn\qquad\quad\text{holds on}\qquad\quad \OB_{\nn,1},
\end{align}
just by considering the curve $\gamma_X\colon [0,1]\ni t\mapsto t\cdot X$ for each $X\in \OB_{\nn,1}$. 
\subsubsection{Coordinate Changes} 
For $h\in G$, we define $\chart_h(g):=\chart(h^{-1}\cdot g)$ for each $g\in h\cdot \U$; i.e.,
\begin{align*}
	[\chart_h]^{-1}(x)=h\cdot \chartinv(x)\qquad\quad\forall\: x\in \V.
\end{align*}
Let now $\compact\subseteq \U$ be a fixed compact:
\begingroup
\setlength{\leftmargini}{12pt}
\begin{itemize}
\item
 We choose $\wt{\compact},U\subseteq \U$ open with $\compact\subseteq \wt{\compact}$ and $e\in U$, such that $\wt{\compact}\cdot U\subseteq\U$ holds.
\item 
We let $U':=\chart(U)$, and observe that 
\begin{align*}
	\xi\colon \wt{\compact}\times U'\rightarrow \V,\qquad (\he\wt{c},u')\mapsto (\chart\cp\mult)(\wt{c}, \chartinv(u'))
\end{align*}
is defined and smooth; i.e., that 
	$\Theta\equiv\partial_2\he\xi\colon \wt{C}\times U'\times E\rightarrow E$ 
is continuous, and linear in $E$. 
\end{itemize}
\endgroup
\noindent
Let now $\pp\in \SEM$ be fixed:
\begingroup
\setlength{\leftmargini}{12pt}
\begin{itemize}
\item
Corollary \ref{ofdpopfdofdp}, applied to $\Phi\equiv \Theta$,  $X\equiv \wt{C}\times U'$, $F_1\equiv E$, and $\compacto\equiv\compact\times \{0\}$, provides us with an open subset $O\subseteq \wt{C}\times U'$ containing $\compact\times \{0\}$, as well as $\uu\equiv\qq_1\in \SEM$, such that
\begin{align*}
	(\pp\cp \Theta)(z,X)\leq \uu(X)\qquad\quad\forall\: z\in O,\:\: X\in E
\end{align*}   
holds. We fix an open neighbourhood $W\subseteq \U$ of $e$ with $\compact\cdot W\times \chart(W)\subseteq O$, and obtain   
 \begin{align}
 \label{opdfdodpof}
	(\pp\cp \Theta)(g\cdot h,\chart(q),X)\leq \uu(X)\qquad\quad\forall\: g\in \compact,\:\: h\in W,\:\: q \in W,\:\: X\in E.
\end{align} 
\item
Here, we can assume that $\chart(W)$ is convex; and choose $V\subseteq W$ symmetric open with $e\in V$ and $V\cdot V\subseteq W$.  
Moreover, since $\SEM$ is filtrating, we can additionally assume that $\OB_{\uu,1}\subseteq \chart(V)$ holds.
\end{itemize}
\endgroup
\noindent
We obtain
\begin{lemma}
\label{fhfhfhffhaaaa}
Let $\compact\subseteq \U$ be compact. Then, for each $\pp\in\SEM$, there exists some $\pp\leq \uu\in \SEM$, and a symmetric open neighbourhood $V\subseteq\U$ of $e$ with $\compact\cdot V\subseteq \U$ and $\OB_{\uu,1}\subseteq \chart(V)$, 
such that 
\begin{align*}
	\pp(\chart(q)-\chart(q'))\leq \uu(\chart_{g\cdot h}(q)-\chart_{g\cdot h}(q'))\qquad\quad \forall\: q,q'\in g\cdot V,\:\: h\in V
\end{align*} 
holds for each $g\in \compact$. 
\end{lemma}
\begin{proof}
We choose $V,W$, $\uu$ as above. Then, for  
 $g\in \compact$, $q,q'\in g\cdot V$, and $h\in V$ fixed, we define 
 \begingroup
\setlength{\leftmargini}{12pt}
\begin{itemize}
\item
	$x:=\chart_{g\cdot h}(q),\:  x':=\chart_{g\cdot h}(q')\: \in \: \chart(V\cdot V)\subseteq \chart(W)$,
\item
	$\delta\colon [0,1]\rightarrow \chart(W),\quad t\mapsto x'+t\cdot (x-x')$,
\item
	$\gamma:=\xi(g\cdot h,\delta)$.
\end{itemize}
\endgroup
\noindent 
We conclude from \eqref{opdfdodpof} and Lemma \ref{ofdpofdpopssssaaaasfffff} that
\begin{align*}
	\pp(\chart(q)-\chart(q'))&= \pp(\xi(g\cdot h,\delta(1))-\xi(g\cdot h,\delta(0)))\\
	&=\pp(\gamma(1)-\gamma(0))
	=\textstyle \pp\big(\int \dot\gamma(s) \:\dd s\big)\\
	&\textstyle=\pp\big(\int \Theta(g\cdot h,\delta(s),\dot\delta(s)) \:\dd s\big)\\
   &\textstyle\leq \int \uu(\dot\delta(s))\:\dd s\textstyle=\int \uu(x-x')\:\dd s\\
  &=\uu(\chart_{g\cdot h}(q)-\chart_{g\cdot h}(q'))
\end{align*}	  
holds, which shows the claim.
\end{proof} 

\subsection{The Evolution Map}
\label{dskjdskjdskjdsdsdsdsds}
We now introduce the central object of this paper -- the evolution map -- and discuss its most important properties. 
\subsubsection{The Right Logarithmic Derivative}
The right logarithmic derivative is defined by
\begin{align*}
	\Der\colon  C^1(D,G)\rightarrow C^0(D,\mg),\quad\:\: \mu \mapsto \dd_\mu\RT_{\mu^{-1}}(\dot\mu)\qquad\quad\: \forall\: D\in \INT.  
\end{align*}
Then, for each $\mu\in C^1(D,G)$, $g\in G$,  $\INT\ni D'\subseteq D$, and each $\varrho\colon \INT\ni D''\rightarrow D$ of class $C^1$, we have
\begin{align}
\label{fgfggf}
	\Der(\mu\cdot g)=\Der(\mu)\qquad\quad\: \Der(\mu|_{D'})=\Der(\mu)|_{D'}\qquad\quad\: \Der( \mu\cp\varrho)=\dot\varrho\cdot\rcK{\Der(\mu)\cp\varrho}.\quad
\end{align}
 Moreover, for $\mu,\nu\in C^1(D,G)$, we  
	 conclude from the product rule \eqref{LGPR} that
	\begin{align}
	\label{alsalsalsa}
		\Der(\mu\cdot \nu)= \Der(\mu)+\Ad_\mu(\Der(\nu))
	\end{align}
	holds; thus, 
\begin{align}
\label{FORM1}
	0=\Der(\mu^{-1}\mu)&=\Der(\mu^{-1}) + \Ad_{\mu^{-1}}(\Der(\mu))\qquad\:\:\Longrightarrow \qquad\:\:\Der(\mu^{-1})=-\Ad_{\mu^{-1}}(\Der(\mu))\quad\\
\label{FORM2}
	\Der(\mu^{-1}\nu)&=\Der(\mu^{-1}) + \Ad_{\mu^{-1}}(\Der(\nu)).
\end{align}
Here, we denote $\mu^{-1}:=\inv\cp \mu$ for each $\mu\in C^0(D,\mg)$ in the following.  
Then, combining \eqref{FORM2} with the right side of \eqref{FORM1}, 
we obtain
	 \begin{align}
	\label{popopo}
		\Der(\mu^{-1}\nu)=\Ad_{\mu^{-1}}(\Der(\nu) -\Der(\mu))\qquad\quad\forall\: \mu,\nu\in C^1(D,G).	
	\end{align} 
	We conclude that
\begin{lemma}
\label{xckxklxc}
	Let $\mu,\nu\in C^1(D,G)$ for $D\in \INT$ be given. Then, we have
	\begin{align*}
		\Der(\mu)=\Der(\nu)\qquad\quad\:\Longleftrightarrow\qquad\quad\:  \nu=\mu\cdot g\quad\:\:\text{holds for some}\quad\:\: g\in G.
	\end{align*}	
\end{lemma}
\begin{proof}
By \eqref{fgfggf}, we have $\Der(\mu)=\Der(\mu \cdot g)$ for each $g\in G$; which shows the one direction. 
For the other direction, we fix $\tau\in D$, define $\alpha:=\mu^{-1}\nu\cdot g$ for $g:=\nu^{-1}(\tau)\cdot \mu(\tau)$, and obtain 
\begin{align*}
	\Der(\alpha)\stackrel{\eqref{fgfggf}}{=}\Der(\mu^{-1}\nu)\stackrel{\eqref{popopo}}{=} 0;
\end{align*}
thus, $\dot\alpha=0$ as $\dd_q\RT_{q^{-1}}$ is bijective for each $q\in G$. 
For each $[r,r']\subseteq D$ with $\tau\in [r,r']$ and $\alpha([r,r'])\subseteq \U$, we thus obtain from \eqref{isdsdoisdiosd} that $(\chart\cp\alpha)|_{[r,r']}=0$ holds; so that the claim follows from a standard supremum-contradiction argument. 
\end{proof}
\noindent
We furthermore obtain 
\begin{lemma}
\label{evk}
Let $D\in \INT$ and $k\in \NN$ be fixed. Then,
\begingroup
\setlength{\leftmargini}{17pt}
{
\renewcommand{\theenumi}{{\arabic{enumi}})} 
\renewcommand{\labelenumi}{\theenumi}
\begin{enumerate}
\item
\label{aaaaaaaaaaaaaaaa}
	$\Der(\mu)\in C^k(D,\mg)$ holds for each $\mu\in C^{k+1}(D,G)$. 
\item
\label{aaaaaaaaaaaaaaaab}
	$\mu\in C^{k+1}(D,G)$ holds for each $\mu\in C^1(D,G)$ with $\Der(\mu)\in C^k(D,\mg)$.
\end{enumerate}}
\endgroup
\end{lemma}
\begin{proof}
By the second identity in \eqref{fgfggf}, in both situations it suffices to show that for each $t\in D$ there exists an open interval $J\subseteq \RR$ containing $t$, such that the claim holds for $\nu:=\mu|_{D\cap J}$. 
Moreover, by the first identity in \eqref{fgfggf}, we can additionally assume that $\im[\nu]\subseteq \U$ holds, just by shrinking $J$ if necessary.    
We let $\gamma:=\chart\cp\nu$, and obtain 
\begingroup
\setlength{\leftmargini}{12pt}
\begin{itemize}
\item
	$\Der(\nu)=\dermapdiff(\gamma,\dot\gamma)$    
for $\dermapdiff$ defined by \eqref{opopopop1}; so that \ref{aaaaaaaaaaaaaaaa} is clear from Lemma \ref{sddsdssd}. 
\item
	$\dot\gamma=\dermapinvdiff(\gamma,\Der(\nu))$    
for $\dermapinvdiff$ defined by \eqref{opopopop2}; so that \ref{aaaaaaaaaaaaaaaab} 
is clear from Corollary  \ref{bhsbsshshdkksjdhjsd}.
\end{itemize}
\endgroup
\noindent
This establishes the proof.
\end{proof}
\noindent
Finally, if $H$ is a Lie group, and  $\Psi\colon G\rightarrow H$ a $C^1$-Lie group homomorphism, we immediately obtain
	\begin{align}
	\label{asssadxfgg}
		\Der(\Psi\cp\mu)=\dd_e\Psi\cp \Der(\mu)\qquad\quad\forall\: \mu\in C^1(D,H),\:\: D\in \INT.	
	\end{align}

\subsubsection{The Product Integral}
\label{opsopdsospdosdpdsoppods}
We define 
\begin{align*}
	\textstyle\DIDE:=\bigsqcup_{[r,r']\in \COMP}\DIDE_{[r,r']}\qquad\text{with}\qquad \DIDE_{[r,r']}:=\Der(C^1([r,r'],G))\qquad\text{for each}\qquad [r,r']\in \COMP.
\end{align*}
Then, Lemma \ref{xckxklxc} shows that $\EV_{[r,r']}\colon \DIDE_{[r,r']} \rightarrow C^1([r,r'],G)$ given by 
\begin{align*}
	\EV_{[r,r']}(\Der(\mu)):=\mu\cdot \mu^{-1}(r)\qquad\quad 
	\forall\:\mu\in C^1([r,r'],G),\:\:
	[r,r']\in \COMP\hspace{8.25pt}
\end{align*}
is well defined; and we let  
\begin{align*}
	\EV_{[r,r']}^k\equiv \EV|_{\DIDE_{[r,r']}^k}\qquad\quad\text{for}\qquad\quad \DIDE_{[r,r']}^k:=\DIDE_{[r,r']}\cap C^k([r,r'],\mg)\\[5pt]
	\EVE_{[r,r']}^k\colon \DIDE_{[r,r']}^k\ni\phi\mapsto \EV_{[r,r']}(\phi)(r')\in G\hspace{75pt}
\end{align*}
for $[r,r']\in \COMP$ and $k\in \NN\sqcup\{\lip,\infty\}$. Moreover, for $k\in \NN\sqcup\{\infty\}$ and $[r,r']\in \COMP$, we define
\begin{align*}
	C_*^{k}([r,r'],G):=\{\mu\in C^{k}([r,r'],G)\: |\: \mu(r)=e\};\hspace{25pt}
\end{align*}
and obtain that
\begin{corollary}
\label{ofgogpfopg}
	For each $k\in \NN\sqcup\{\lip,\infty\}$ and $[r,r']\in \COMP$, we have  
\begin{align*}
	\EV_{[r,r']}^k\colon \DIDE_{[r,r']}^k\rightarrow C_*^{k+1}([r,r'],G).
\end{align*}
\end{corollary}
\begin{proof}
	The claim is clear from Lemma \ref{evk}.\ref{aaaaaaaaaaaaaaaab}.
\end{proof}
\noindent
The product integral is given by\footnote{Observe that the first expression is defined by the second equality in \eqref{fgfggf} as well as Lemma \ref{evk}.}
		\begin{align*}
		\textstyle\innt_a^b \phi := \EV_{[a,b]}^k(\phi|_{[a,b]})(b),\qquad\innt\phi:=\innt_r^{r'}\phi,\qquad\innt_c^c\phi:=e,\qquad \innt_r^\bullet\phi\colon [r,r']\ni t\mapsto \innt_r^t\phi
	\end{align*} 
	for each $\phi\in \DIDE_{[r,r']}^k$ with $k\in \NN\sqcup\{\lip,\infty\}$, $r\leq a<b\leq r'$, and $c\in [r,r']$. Then, 
\begingroup
\setlength{\leftmargini}{16pt}
{
\renewcommand{\theenumi}{\emph{\alph{enumi})}} 
\renewcommand{\labelenumi}{\theenumi}
\begin{enumerate}
\item
\label{kdsasaasassaas}
We conclude from \eqref{alsalsalsa} that
\begin{align*}
	\textstyle\innt_r^t \phi \cdot \innt_r^t\psi=\innt_r^t \phi+\Ad_{\innt_r^\bullet\phi}(\psi)\qquad\quad\hspace{40pt}\forall\: \phi,\psi\in \DIDE_{[r,r']},\:\: t\in [r,r'].
\end{align*} 
\item
\label{kdskdsdkdslkds}
We conclude from \eqref{popopo} that
\begin{align*}
	\textstyle\big[\innt_r^t \phi\big]^{-1} \big[\innt_r^t\psi\big]=\innt_r^t\Ad_{[\innt_r^\bullet\phi]^{-1}}(\psi-\phi)\qquad\quad\forall\: \phi,\psi\in \DIDE_{[r,r']},\:\: t\in [r,r'].
\end{align*} 
\item
\label{pogfpogfaaa}
We conclude from \eqref{FORM1} that
\begin{align*}
	\textstyle \big[\innt_r^t\phi\big]^{-1}=\innt_r^t -\Ad_{[\innt_r^\bullet\phi]^{-1}}(\phi)\qquad\quad\hspace{32pt}\forall\: \phi\in \DIDE_{[r,r']},\:\: t\in [r,r'].
\end{align*}
\item
\label{pogfpogf}
For $r=t_0<{\dots}<t_n=r'$ and $\phi\in \DIDE_{[r,r']}$, we conclude from the first two identities in \eqref{fgfggf} that
	\begin{align*}
		\textstyle\innt_r^t\phi=\innt_{t_{p}}^t \phi\cdot \innt_{t_{p-1}}^{t_{p}} \phi \cdot {\dots} \cdot \innt_{t_0}^{t_1}\phi\qquad\quad\forall\:t\in (t_p,t_{p+1}],\:\: p=0,\dots,n-1.
	\end{align*}
\item
\label{subst}
	For $\varrho\colon [\ell,\ell']\rightarrow [r,r']$ 
of class $C^1$, we conclude from the last identity in \eqref{fgfggf} that
\begin{align*}
	 \textstyle\innt_r^{\varrho}\phi=\big[\innt_\ell^\bullet\dot\varrho\cdot \rcK{\phi\cp\varrho}\he\big]\cdot \big[\innt_r^{\varrho(\ell)}\phi\he\big]\qquad\quad\forall\:\phi\in \DIDE_{[r,r']}.
\end{align*} 
\item
\label{homtausch}
We conclude from \eqref{asssadxfgg} that for each $C^1$-Lie group homomorphism $\Psi\colon G\rightarrow H$, we have
\begin{align*}
	\textstyle\Psi\cp \innt_r^\bullet \phi = \innt_r^\bullet \dd_e\Psi\cp\phi\qquad\quad\forall\:\phi\in \DIDE_{[r,r']}.
\end{align*}
\end{enumerate}}
\endgroup	
\begin{example}
\label{fdpofdopdpof}
For $[r,r']\in \COMP$ fixed, we let 
	$\varrho\colon [r,r']\rightarrow [r,r'],\:\: t\mapsto r +r' -t$; and define   
\begin{align*}
	\textstyle \DIDE_{[r,r']}\ni \inverse{\phi}:=\dot\varrho\cdot \rcK{\phi\cp\varrho} \colon [r,r']\ni t\mapsto - \phi(r+r'-t)\qquad\quad\forall\:\phi\in \DIDE_{[r,r']}. 
\end{align*}
We let $[\ell,\ell']\equiv[r,r']$; and obtain from \emph{\ref{subst}} that
\begin{align*}
	\textstyle e=\innt_r^{\varrho(r')}\phi\stackrel{\emph{\ref{subst}}}{=}\big[\innt_{r}^{r'} \inverse{\phi}\big] \cdot \big[\innt_r^{r'}\phi  \big]\qquad\quad\text{holds, thus}\qquad\quad [\innt \phi]^{-1}=\innt \inverse{\phi},
\end{align*}
which will be useful for our argumentation in Sect.\ \ref{pofdofdpofdpofdpofdfdofdpodf}.
\hspace*{\fill}$\ddagger$
\end{example}
\begin{lemma}
\label{sddssdsdsd}
Let $[r,r']\in \COMP$, and $k\in \NN\sqcup\{\lip, \infty\}$ be fixed; and suppose that we are given $\phi\in C^k([r,r'],\mg)$ and 
$r=t_0<{\dots}<t_n=r'$, such that $\phi|_{[t_p,t_{p+1}]}\in \DIDE^k_{[t_p,t_{p+1}]}$ holds for $p=0,\dots,n-1$. 
Then, we have $\phi\in \DIDE^k_{[r,r']}$  with
	\begin{align*}
		\textstyle\innt_r^t\phi=\innt_{t_{p}}^t \phi\cdot \innt^{t_p}_{t_{p-1}} \phi \cdot {\dots} \cdot \innt_{t_0}^{t_1}\phi\qquad\quad\forall\:t\in (t_p,t_{p+1}],\:\: p=0,\dots,n-1.
	\end{align*}
\end{lemma}
\begin{proof}
The proof is elementary, and can be found in Appendix \ref{ceinseig}.
\end{proof}
\subsubsection{Semiregularity}
We  say that $G$ is  {\bf $\boldsymbol{C^k}$-semiregular} for $k\in \NN\sqcup\{\lip,\infty\}$ \defff $\DIDE^k_{[0,1]}=C^k([0,1],\mg)$ holds. Then,
\begin{lemma}
\label{assasaasas}
$G$ is $C^k$-semiregular \deff 
\begin{align*}
	\DIDE^k_{[r,r']}=C^k([r,r'],\mg)\qquad\text{holds for each}\qquad [r,r']\in \COMP.
\end{align*}
\end{lemma}
\begin{proof}
The one direction is evident. For the other direction, we fix $[r,r']\in \COMP$, and let 
\begin{align*}
	\varrho\colon [r,r']\rightarrow [0,1],\qquad t\mapsto  |t-r|/|r'-r|.
\end{align*}
Then, for $\phi\in C^k([r,r'],\mg)$ given, we define  
	$\psi:= |r'-r|\cdot \phi\cp \varrho^{-1}\in C^k([0,1],\mg)$,
 and choose $\nu\in C^{k+1}([0,1],\mg)$ with $\Der(\nu)=\psi$. Then, the last identity in \eqref{fgfggf} gives
\begin{align*}
	\Der(\nu\cp\varrho)=|r'-r|^{-1} \cdot \rcK{\psi\cp\varrho}=\phi,
\end{align*}   
which proves the claim.
\end{proof}
\noindent
We say that $G$ {\bf admits an exponential map} \defff\he $\phi_X|_{[0,1]}\in \DIDE_{[0,1]}$ holds for each constant curve $\phi_X\colon \RR\ni t\mapsto X\in \mg$; i.e., \defff 
\begin{align*}
	\textstyle\exp\colon \mg\ni X\mapsto \innt_0^1 \phi_X\in G
\end{align*} 
is defined. 

\subsubsection{Continuity}
We say that $\EVE_{[r,r']}^k$ is $C^p$-continuous for $p\leq k\in \NN\sqcup\{\lip,\infty\}$ (we let $0\leq\lip\leq\lip\leq 1$) and $[r,r']\in \COMP$ \deff it is continuous w.r.t.\ the seminorms $\{\ppp_\infty^\dind\}_{\pp\in \SEM,\:\dind\llleq p}$. We say that $G$ is 
\begingroup
\setlength{\leftmargini}{12pt}
\begin{itemize}
\item
	{\bf p$\boldsymbol{\pkt}$k-continuous}\:\: for $p\llleq k$\: \defff\:$\EVE_{[r,r']}^k$ is $C^p$-continuous for each $[r,r']\in \COMP$,
\item
	\hspace{10.5pt}{\bf k-continuous}\hspace{46.5pt}\: \:\:\defff\:$G$ is k$\pkt$k-continuous.
\end{itemize}
\endgroup
\noindent
Then,
\begin{lemma}
\label{Adlip}
We have $\Ad_\mu(\phi)\in C^k([r,r'],\mg)$ for each $\mu\in C^{k+1}([r,r'],G)$, $\phi\in C^k([r,r'],\mg)$, and $k\in \NN\sqcup\{\lip,\infty\}$.
\end{lemma}
\begin{proof}
Since $\Ad\colon G\times \mg\rightarrow \mg$ is smooth, the claim is clear for $k\in \NN\sqcup \{\infty\}$. The case where $k=\lip$ holds is proven in Appendix \ref{appLip}.
\end{proof}
\begin{lemma}
\label{opopsopsdopds}
	Let $[r,r']\in \COMP$, $k\in \NN\sqcup\{\lip,\infty\}$, and $\phi\in \DIDE_{[r,r']}^k$ be fixed. Then, 
	for each $\pp\in \SEM$ and $\dind\llleq k$, there exists some $\pp\leq \qq\in \SEM$ with
		\begin{align*}
		\ppp^\dindp\big(\Ad_{[\innt_r^\bullet\phi]^{-1}}(\psi)\big)\leq \qqq^\dindp(\psi)\qquad\quad\forall\: \psi\in C^k([r,r'],\mg),\:\: 0\leq \dindp\leq \dind.
	\end{align*}
\end{lemma}
\begin{proof}
Decomposing $[r,r']$ if necessary, we can assume that $\im[\innt_r^\bullet\phi]$ is contained in the domain of a fixed chart $\wt{\chart}$. The claim then follows from Lemma \ref{oopxcxopcoxpopcx}.\ref{oopxcxopcoxpopcx2}, applied to $\Omega\equiv\Ad(\inv\cp\wt{\chart}^{-1}(\cdot),\cdot)$, $\dindu\equiv \dind$, and the $C^\dind$-curve $\gamma\equiv\wt{\chart}\cp\innt_r^\bullet\phi$.
\end{proof}
\noindent
We obtain that
\begin{lemma}
\label{klklllkjlaaa}
$G$ is {\rm p$\pkt$k}-continuous \deff $\EVE^k_{[0,1]}$ is $C^p$-continuous at zero.
\end{lemma}
\begin{proof}
The one direction is evident; and the other direction follows from Lemma \ref{Adlip}, Lemma \ref{opopsopsdopds}, and 
\ref{kdskdsdkdslkds} once we have shown 
that $\EVE_{[r,r']}^k$ is $C^p$-continuous at zero if $\EVE_{[r,r']}^k$ is $C^p$-continuous at zero. 
For this, we apply \ref{subst} to 
\begin{align*}
	\varrho\colon [0,1]\rightarrow [r,r'],\qquad t\mapsto r+ t\cdot |r'-r|; 
\end{align*}
and conclude that $\EVE_{[r,r']}^k=\EVE_{[0,1]}^k\cp \:\eta$ holds, for the $C^p$-continuous map (use \ref{chainrule})
 	\begin{align*}
		\textstyle  
		\eta\colon C^k([r,r'],\mg)\rightarrow C^k([0,1],\mg),\qquad \phi\mapsto \dot\varrho\cdot \rcK{\phi\cp\varrho}\equiv |r'-r|\cdot \rcK{\phi\cp\varrho}.
	\end{align*}
	From this, the claim is clear.
\end{proof}
\subsection{Supplementary Material}
\label{poigfgfoogfigfoigfoioigf}
In this subsection, we provide the proofs of the supplementary statements made but not verified in Sect.\ \ref{pofpofdpofdpofdpfdfd}. First,
\begin{lemma}
\label{ssssss}
Suppose that $G$ is abelian; and let $k\in \NN\sqcup\{\lip, \infty\}$ be fixed. Then, $G$ is {\rm k}-continuous \deff $G$ is $\mathrm{0\pkt k}$-continuous.
 \end{lemma}
\begin{proof}
	The one directions is evident. For the other direction, we suppose that $G$ is k-continuous. Then, $\pp\in \SEM$ given, there exist $\qq\in\SEM$ and $\dind\llleq k$, such that 
\begin{align}
\label{podopfdpofd}
	\textstyle\qqq_\infty^\dind(\psi)\leq 1\quad\:\:\text{for}\quad\:\: \psi\in \DIDE_{[0,1]}^k\qquad\quad\Longrightarrow\qquad\quad (\pp\cp\chart)(\innt\psi)\leq 1.
\end{align}	
Then, for $\phi\in \DIDE_{[0,1]}^k$ with $\qqq_\infty(\phi)\leq 1$, we choose $n\geq 1$ such large that $\qqq_\infty^\dind(\phi)\leq n$ holds; and define 
\begin{align*}
	\psi_p:= \phi\cp \varrho_p\qquad\text{for}\qquad \varrho_p\colon [0,1/n]\ni t\mapsto p/n+t\in [p/n,(p+1)/n]\qquad\quad\:\:\forall\: p=0,\dots,n-1.  
\end{align*}
By \ref{subst}, we have $\innt \phi|_{[p/n,(p+1)/n]}=\innt \psi_p$ for $p=0,\dots n-1$; and obtain from \ref{kdsasaasassaas}, \ref{pogfpogf}, and \ref{subst} that\footnote{It is obvious from the definitions that $\Ad_g=\id_\mg$ holds for each $g\in G$ if $G$ is abelian.}
\begin{align}
\label{fdpofopfdpofdpopofdassa}
	\textstyle\innt \phi=\innt\psi_{n-1}\cdot {\dots}\cdot \innt\psi_{0}=\innt_0^{1/n} \psi_{n-1}+{\dots}+\psi_{0}=\innt_0^1 \underbrace{1/n\cdot (\psi_0+{\dots}+\psi_{n-1})\cp \varrho}_{\psi\in \DIDE_{[0,1]}^k}
\end{align}
holds, for $\varrho\colon [0,1]\ni t\mapsto t/n \in [0,1/n]$. Then, \ref{chainrule} gives $\qq_\infty^\dind(\psi)\leq 1$; so that \eqref{podopfdpofd} provides us with   
\begin{align*}
	\textstyle(\pp\cp\chart)\big(\innt \phi\big)\stackrel{\eqref{fdpofopfdpofdpopofdassa}}{=}(\pp\cp\chart)\big(\innt \psi\big)\leq 1. 
\end{align*}
The rest is clear from Lemma \ref{klklllkjlaaa}. 
\end{proof}
Second, let us say that $G$ is $\rm L^1$-continuous \defff $\EVE_{[r,r']}^0$ is continuous w.r.t.\ the seminorms \eqref{ofdpofdpofdpofd} for each $[r,r']\in \COMP$. Then,  
\begin{lemma}
\label{sdsddsdsdsdsds}
$G$ is {\rm 0}-continuous \deff $G$ is $\rm L^1$-continuous.
\end{lemma}
\begin{proof}
The one direction is evident. Let thus $G$ be 0-continuous, fix $\pp\in \SEM$; and choose $\qq\in \SEM$ with
\begin{align}
\label{oapoapoapoapoapoapoapa}
	\textstyle\qqq_\infty(\psi)\leq 1\quad\:\:\text{for}\quad\:\: \psi\in \DIDE_{[0,2]}^0\qquad\quad\Longrightarrow\qquad\quad (\pp\cp\chart)(\innt \psi)\leq 1.
\end{align} 
Then, for $\phi\in \DIDE_{[r,r']}^0$ with $\qqq_{\int}(\phi)\leq 1$, we define
\begin{align*}
	\textstyle \lambda\colon [r,r']\rightarrow [0,2],\qquad t\mapsto  \frac{t-r}{r'-r}\cdot (2-\qqq_{\int}(\phi)) + \int_0^t \qqq(\phi(s))\:\dd s; 
\end{align*}
and consider the $C^1$-diffeomorphism $\varrho:=\lambda^{-1}\colon [0,2]\rightarrow [r,r']$. 
Then, $\innt\phi=\innt \psi$ holds for $\psi:= \dot\varrho\cdot \rcK{\phi\cp\varrho}\in \DIDE^0_{[0,2]}$ by \ref{subst}, with 
\begin{align*}
	\dot\varrho=(\dot\lambda\cp\varrho)^{-1}=(2-\qqq_{\int}(\phi))/|r'-r|+\qqq(\phi\cp\varrho))^{-1}\leq \qqq(\phi\cp\varrho)^{-1}.
\end{align*}
We thus have $\qqq_\infty(\psi)\leq 1$; so that the claim is clear from \eqref{oapoapoapoapoapoapoapa}.
\end{proof}
Finally, let us collect some properties of the exponential map.
\begin{remark}
\label{exponentialmap}
\noindent
\vspace{-5pt}
\begingroup
\setlength{\leftmargini}{15pt}
\begin{enumerate}
\item
\label{exponentialmap1}
Suppose we have $\phi_X|_{[0,1]}\in \DIDE_{[0,1]}$ for some $X\in \mg$. Then, Lemma \ref{sddssdsdsd} (and \ref{subst}) shows that $\phi_X|_{[0,n]}\in \DIDE_{[0,n]}$ holds for each $n \geq 1$; and, \ref{subst} applied to $\varrho\colon [0,1]\rightarrow [0,s\cdot n],\quad t\mapsto s\cdot n\cdot t$ for $0<s\leq 1$, gives
\begin{align*}
	\textstyle\innt_0^{s\cdot n} \phi_{X}\stackrel{\ref{subst}}{=}\innt_0^1 \phi_{s\cdot n\cdot X}\equiv\exp(s\cdot n\cdot X)
	\qquad\quad\forall\: 0<s\leq 1.
\end{align*} 
We thus have $\RR_{\geq 0}\cdot X\subseteq \dom[\exp]$ with
\begin{align}
\label{lkdsklsdkdl}
	\textstyle\exp(t\cdot X)=\innt_0^t \phi_{X}
	\qquad\quad\forall\: t\geq 0.
\end{align}  
It follows that $\RR\ni t\mapsto \exp(t\cdot X)$ is a smooth Lie group homomorphism, cf.\ Appendix \ref{ceinseigaaaa}.
\item
\label{exponentialmap2}
Suppose that $G$ is $C^\infty$-semiregular; and that $\EVE_{[0,1]}^\infty$ is of class $C^p$ w.r.t.\ the $C^\infty$-topology, for some $p\in \NN\sqcup\{\infty\}$.  Then, $\exp$ is of class $C^p$, because 
\begin{align*}
	\ppp_\infty^\dind(\phi_X)=\ppp(X) \qquad\quad\forall\: \pp\in \SEM,\:\:\dind\in \NN,\:\:X\in \mg
\end{align*} 
shows that $\mg\ni X\mapsto \phi_X\in C^\infty([0,1],\mg)$ is smooth.
\item
\label{exponentialmap3}
If $G$ is abelian with $\exp\colon \mg\rightarrow G$ of class $C^1$, then we have, cf.\ Appendix \ref{ceinseigaaaaa}
	\begin{align*}
		\textstyle\innt \phi=\exp(\int \phi(s)\: \dd s)\quad\:\:\:\text{for each}\quad\:\:\: \phi\in C^0([0,1],\mg)\quad\:\:\:\text{with}\quad\:\:\:\int_0^t \phi(s)\: \dd s\in \mg\quad\:\forall\: t\in [0,1];   
	\end{align*}		
	which is obviously continuous w.r.t.\ the seminorms $\ppp_\infty,\ppp_{\int}$ for $\pp\in \SEM$.\hspace*{\fill}$\ddagger$
\end{enumerate}
\endgroup
\end{remark}

\section{Auxiliary Results}
\label{dsjshdhkjshkhjsd}
In this section, we prove further continuity statements for the evolution map; and discuss piecewise integrable curves. 
\subsection{Continuity of the Evolution Map}
\label{CPOF}
\begin{lemma}
\label{posdpospodspoaaaa}
Suppose that $G$ is \rm{p$\pkt$k}-continuous; and let $[r,r']\in \COMP$ be fixed. Then, for each $\pp\in \SEM$, there exist $\pp\leq \qq\in \SEM$  and $\dind\llleq p$, such that
\begin{align*}
\textstyle\qqq^\dind_\infty(\phi)\leq 1\quad\:\:\text{for}\quad\:\: \phi\in \DIDE^k_{[r,r']}\qquad\quad\:\:\Longrightarrow\qquad\quad\:\:
\textstyle(\pp\cp\chart)(\innt_r^{\bullet}\phi)\leq 1. 
\end{align*}
\end{lemma}
\begin{proof}
By continuity, there exist $\pp\leq \qq\in \SEM$ and $\dind\llleq p$, such that  
\begin{align}
\label{sdsjdahhsdhjdaas}
\textstyle\qqq^\dind_\infty(\psi)\leq 1\quad\:\:\text{for}\quad\:\: \psi\in \DIDE^k_{[r,r']}\qquad\quad\:\:\Longrightarrow\qquad\quad\:\:
\textstyle(\pp\cp\chart)(\innt_r^{\bullet}\psi)\leq 1. 
\end{align}
Let now $\phi\in \DIDE_{[r,r']}^k$ with $\qqq_\infty^\dind(\phi)\leq 1$, and $r<\tau\leq r'$ be fixed.  
We define $\psi:=\phi|_{[r,\tau]}$ as well as 
\begin{align*}
	\textstyle\varrho\colon [r,r']\rightarrow [r,\tau],\qquad t\mapsto r+  |t-r|\cdot c \qquad\qquad\:\:\text{for}\qquad\qquad\:\: c:=\frac{\tau-r}{r'-r}\leq 1.
	\end{align*}
Then, $\innt_r^{\tau}\phi\equiv \innt \psi=\innt \dot\varrho\cdot \rcK{\psi\cp\varrho}$ holds by \ref{subst}, with $\dot\varrho\cdot \rcK{\psi\cp\varrho}\in \DIDE^k_{[r,r']}$ as well as 
\vspace{-4pt}
\begin{align*}
	\textstyle
	\qqq^\dind_\infty(\dot\varrho\cdot \rcK{\psi\cp\varrho})=\qqq^\dind_\infty(c\cdot \rcK{\psi\cp\varrho})\stackrel{\ref{chainrule}}{\leq} \qqq^\dind_\infty(\phi)\leq 1.
\end{align*}
We thus obtain from \eqref{sdsjdahhsdhjdaas} that
\begin{align*}
	\textstyle(\pp\cp\chart)(\innt_r^{\tau}\phi)= (\pp\cp\chart)(\innt\psi)= (\pp\cp\chart)(\innt \dot\varrho\cdot \psi\cp \varrho)\leq 1
\end{align*}
holds, from which the claim is clear.
\end{proof}
We inductively obtain
\begin{lemma}
\label{jlkfdsjlkfdsjklfdsjklfsdjkl}
Suppose that $G$ is {\rm k}-continuous; and let $[r,r']\in \COMP$ be fixed. Then, for each $\pp\in \SEM$ 
and $\dindu\llleq k$, there exist $\pp\leq \qq\in \SEM$ and $\dind\llleq k$, such that
\begin{align*}
	\textstyle\qqq_\infty^\dind(\phi)\leq 1\quad\:\:\text{for}\quad\:\: \phi\in \DIDE^k_{[r,r']}\qquad\quad\:\:\Longrightarrow\qquad\quad\:\:
		\pp_\infty^\dindu\big(\chart\cp\innt_r^\bullet\phi\big)\leq 1.
	\end{align*}
\end{lemma}
Confer \cite{HGGG} for the case that $G$ is $C^k$-semiregular .
\begin{proof}
By Lemma \ref{posdpospodspoaaaa}, we can assume that the claim is proven for some  $0\leq \dindu< k$. In particular, there exist $\mm\in \SEM$ and $\dindo\llleq k$, such that
\begin{align*}
	\textstyle\bg:=\chart\cp\EV_{[r,r']}^k\colon \{\phi\in \DIDE_{[r,r']}^k\:|\: \mmm_\infty^\dindo(\phi)\leq 1\}\rightarrow C_*^{k+1}([r,r'],\V),\qquad \phi\mapsto \chart\cp\innt_r^\bullet \phi
\end{align*}
is defined. Let thus $\phi\in \DIDE_{[r,r']}^k$ with $\mmm_\infty^\dindo(\phi)\leq 1$ be given. Then,
\begingroup
\setlength{\leftmargini}{12pt}
\begin{itemize}
\item
We have $\bg(\phi)^{(1)}=\dermapinvdiff(\bg(\phi),\phi)$, for $\dermapinvdiff$ defined by \eqref{opopopop2}; so that      
 Corollary \ref{fddfd} shows that $\bg(\phi)^{(\dindu+1)}=\sum_{i=1}^d\alpha_i(\phi)$ holds, with
\begin{align*}
	\alpha_i\colon \phi\mapsto([\partial_1]^{m_i}\dermapinvdiff)\big(\bg(\phi),\bg(\phi)^{(z[i]_{1})},\dots,\bg(\phi)^{(z[i]_{m_i})},\phi^{(q_i)}\big)
\end{align*}
for certain $0\leq z[i]_{1},\dots,z[i]_{m_i},q_i  \leq \dindu$ and $m_i\geq 1$, for $i=1,\dots ,d$.
\item
For $\pp\in\SEM$ fixed,  
Lemma \ref{kldskldsksdklsdl} provides us with an open neighbourhood $V\subseteq \V$ of $0$, as well as $\ww\in \SEM$, such that
\begin{align}
\label{opopsdopds}
\begin{split}
	(\pp\cp[\partial_1]^{m_i}\dermapinvdiff)\big(x,\bg(\phi)^{(z[i]_{1})},&\dots,\bg(\phi)^{(z[i]_{m_i})},\phi^{(q_i)}\big)\\
	&\leq \ww\big(\bg(\phi)^{(z[i]_1)}\big)\cdot{\dots}\cdot \ww\big(\bg(\phi)^{(z[i]_{m_i})}\big)\cdot \www\big(\phi^{(q_i)}\big)\\
	&\leq  \big[\ww_\infty^\dindu(\bg(\phi))\big]^{m_i}\cdot \www_\infty^\dindu(\phi)
	\end{split}
\end{align}
holds, for each $x\in V$ and $i=1,\dots,d$.
\end{itemize}
\endgroup
\noindent
 We choose $V\lleq\vv\in \SEM$ with $d\cdot \ww,\pp,\mm\leq\vv$; and apply the induction hypotheses in order to fix $\vv\leq \qq\in \SEM$ and $\dindo\llleq \dind\llleq k$, such that
\begin{align*}
	\textstyle\qqq_\infty^\dind(\phi)\leq 1\quad\:\text{for}\quad\: \phi\in \DIDE^k_{[r,r']}\qquad\quad\:\:\Longrightarrow\qquad\quad\:\:
		\vv^\dindu_\infty(\bg(\phi))\equiv\vv^\dindu_\infty\big(\chart\cp\innt_r^\bullet\phi\big)\leq 1.
	\end{align*}
In particular, then for $\phi\in \DIDE_{[r,r']}^k$ with $\qqq_\infty^\dind(\phi)\leq 1$, we have 
\begingroup
\setlength{\leftmargini}{12pt}
\begin{itemize}
\item
$\im[\bg(\phi)]\subseteq V$ and $\qqq_\infty^\dindo(\phi)\leq 1$; so that \eqref{opopsdopds} gives 
\begin{align*}
	\pp\big(\bg(\phi)^{(\dindu+1)}\big)\leq d\cdot  \www^\dindu_\infty(\phi)\leq \vvv^\dindu_\infty(\phi) \leq \qqq^\dindu_\infty(\phi).
\end{align*}
\item
$\pp_\infty^\dindu(\bg(\phi))\leq \vv_\infty^\dindu(\bg(\phi))\leq 1$.
\end{itemize}
\endgroup
\noindent
For $\wt{\dind}:=\max(\dind,\dindu)$ and $\phi\in \DIDE_{[r,r']}^k$ with $\qqq_\infty^{\wt{\dind}}(\phi)\leq 1$, we thus have
\begin{align*}
		\textstyle\pp\big(\bg(\phi)^{(\dindu+1)}\big)\leq \qqq^{\dindu}_\infty(\phi)\leq\qqq^{\wt{\dind}}_\infty(\phi)\leq 1\qquad\quad\text{and}\qquad\quad
			\pp_\infty^\dindu(\bg(\phi))\leq 1;
\end{align*}
thus, $\pp_\infty^{\dindu+1}(\bg(\phi))\leq 1$. The claim thus follows inductively.  
\end{proof}
\noindent
We furthermore obtain that
\begin{lemma}
\label{opdfopfdopfdpoas}
Suppose that $G$ is {\rm p$\pkt$k}-continuous; and let $[r,r']\in \COMP$ be fixed. Then, for each $\pp\in \SEM$, there exist $\pp\leq\qq\in \SEM$ and $\dind\llleq p$, such that 
\begin{align*}
	\textstyle\qqq^\dind_\infty(\phi)\leq 1\quad\:\:\text{for}\quad\:\: \phi\in \DIDE^k_{[r,r']}\qquad\quad\Longrightarrow\qquad\quad (\pp\cp\chart)\big(\innt_r^\bullet\phi\big)\leq \int_r^\bullet \qqq(\phi(s))\:\dd s. 
\end{align*}
\end{lemma}
\begin{proof}
We choose $\ww$ as in \eqref{dssdsdsdsdsddsds} for $\vv\equiv \pp$ there; and let $\qq$, $\dind$ be as in Lemma \ref{posdpospodspoaaaa} for $\pp\equiv \ww$ there (i.e., we have  $\pp\leq \ww\leq \qq$). 
Then, for $\phi\in \DIDE^k([r,r'],\mg)$ with $\qqq^\dind_\infty(\phi)\leq 1$, we have $(\ww\cp\chart)\big(\innt_r^\bullet\phi\big)\leq 1$ 
by Lemma \ref{posdpospodspoaaaa}; and obtain from \eqref{dssdsdsdsdsddsds} that for $\gamma:= \chart\cp \mu$ with $\mu:=\innt_r^\bullet\phi$ we have
\begin{align*}
	\textstyle\pp(\dot \gamma) =\pp\big(\dermapinvdiff(\gamma,\Der(\mu)))\big)\leq \www(\phi)\leq \qqq(\phi).
\end{align*}
The claim thus follows from Lemma \ref{ofdpofdpopssssaaaasfffff}.
\end{proof}
\subsection{Estimates in Charts}
In this subsection, we prove certain statements that we will need for our differentiability discussions in Sect.\ \ref{asopsopdsopsdpoosdp}. 
We start with a variation of Lemma \ref{opopsopsdopds}.
\begin{lemma}
\label{jlkfdsjasasslkfdsjklfdsjklfsdjkl}
Suppose that $G$ is {\rm k}-continuous; and let $[r,r']\in \COMP$ be fixed. Then, for each $\pp\in \SEM$ and $\dindu\llleq k$, there exist $\pp\leq \mm\in \SEM$ and $\dind\llleq k$, such that	
\begin{align*}
		\ppp^\dindp\big(\Ad_{[\innt_r^\bullet \phi]^{-1}}(\psi)\big)\leq \mmm^\dindp(\psi) \qquad\quad\forall\:\psi\in C^k([r,r'],\mg),\:\: 0\leq \dindp\leq \dindu 
	\end{align*}
	holds for each $\phi\in \DIDE_{[r,r']}^k$ with $\mmm_\infty^\dind(\phi)\leq 1$.
\end{lemma}
\begin{proof}
Since $\SEM$ is filtrating, Lemma \ref{oopxcxopcoxpopcx}.\ref{oopxcxopcoxpopcx1} applied to  
\begin{align*}
	\Omega\colon \V\times \mg\rightarrow \mg,\qquad (x,X)\mapsto \Ad((\inv\cp\chartinv)(x),X)
\end{align*}
provides us with some $\pp\leq \qq\in \SEM$, such that for each $\phi\in \DIDE_{[r,r']}^k$ with $\qqq_\infty^\dindu\big(\chart\cp \innt_r^\bullet\phi\big)\leq 1$, we have
\begin{align*}
	\ppp^\dindp\big(\Ad_{[\innt_r^\bullet \phi]^{-1}}(\psi)\big)\leq \qqq^\dindp(\psi) \qquad\quad\forall\:\psi\in C^k([r,r'],\mg),\:\: 0\leq\dindp\leq\dindu.
\end{align*}
  By Lemma \ref{jlkfdsjlkfdsjklfdsjklfsdjkl}, there exist $\qq\leq \mm\in \SEM$ and $\dind\llleq k$, such that $\qqq_\infty^\dindu\big(\chart\cp\innt_r^\bullet\phi\big)\leq 1$ holds for each $\phi\in \DIDE_{[r,r']}^k$ with $\mmm_\infty^\dind(\phi)\leq 1$; from which the claim is clear. 
\end{proof}
Together with Lemma \ref{jlkfdsjlkfdsjklfdsjklfsdjkl}, this shows 
\begin{lemma}
\label{jlkfdsjlkfdsjklfdsjklfsdjskl}
Suppose that $G$ is {\rm k}-continuous; and let $[r,r']\in \COMP$ be fixed. Then, for each $\pp\in \SEM$, there exist $\pp\leq \mm\in \SEM$ and $\dindu\llleq k$, such that	
\begin{align*}
	\textstyle(\pp\cp\chart)\big([\innt_r^\bullet\phi]^{-1}[\innt_r^\bullet\psi] \he\big)\leq \int_r^{\bullet}\mmm(\psi(s)-\phi(s))\:\dd s 
\end{align*}
holds for all $\phi,\psi \in \DIDE_{[r,r']}^k$ with $\mmm_\infty^\dindu(\phi), \:\mmm^\dindu_\infty(\psi-\phi)\leq 1$. 
\end{lemma}
\begin{proof}
We choose $\pp\leq \qq\in \SEM$ and $\dind\llleq k$ as in Lemma \ref{opdfopfdopfdpoas}. Then, Lemma \ref{jlkfdsjasasslkfdsjklfdsjklfsdjkl} provides us with some $\qq\leq \mm\in \SEM$ and $\dindo\llleq k$, such that for each $\phi\in \DIDE_{[r,r']}^k$ with $\mmm_\infty^\dindo(\phi)\leq 1$, we have
\begin{align}
\label{spodsopodspodsaaaaa}
		\qqq^\dindp\big(\Ad_{[\innt_r^\bullet \phi]^{-1}}(\chi)\big)\leq \mmm^\dindp(\chi) \qquad\quad\forall\:\chi\in C^k([r,r'],\mg),\:\: 0\leq \dindp\leq \dind. 
	\end{align} 
	We let $\dindu:=\max(\dindo,\dind)$; and recall that, cf.\ \ref{kdskdsdkdslkds} 
\begin{align} 
\label{spodsopodspods}
	\textstyle [\innt_r^t\phi]^{-1} [\innt_r^t\psi] =\innt_r^t \Ad_{[\innt_r^\bullet\phi]^{-1}}(\psi-\phi) \qquad\quad\forall\: t\in [r,r'],\:\: \phi,\psi\in \DIDE^k_{[r,r']}
\end{align} 
holds. For $\phi,\psi \in \DIDE_{[r,r']}^k$ with $\mmm_\infty^\dindu(\phi), \:\mmm^\dindu_\infty(\psi-\phi)\leq 1$, we thus have 
\begin{align*}
	\qqq_\infty^\dind\big(\Ad_{[\innt_r^\bullet\phi]^{-1}}(\psi-\phi)\big)\stackrel{\eqref{spodsopodspodsaaaaa}}{\leq} \mm_\infty^\dind(\psi-\phi)\leq \mm_\infty^\dindu(\psi-\phi)\leq 1;
\end{align*}
so that the claim is clear from Lemma \ref{opdfopfdopfdpoas}, \eqref{spodsopodspods}, and \eqref{spodsopodspodsaaaaa} for $\dindp\equiv 0$ there.
\end{proof}
We conclude that
	\begin{proposition}
	\label{hghghggh}
	Suppose that $G$ is {\rm k}-continuous; and let $[r,r']\in \COMP$ be fixed. Then, for each $\pp\in \SEM$, there exist $\pp\leq\mm\in \SEM$ and $\dind\llleq k$, such that 
		\begin{align*}
		\textstyle\pp\big(\chart\big(\innt_r^\bullet\phi\big)-\chart\big(\innt_r^\bullet\psi\big)\big)\leq \int_r^\bullet\mmm(\psi(s)-\phi(s))\:\dd s
	\end{align*}
	holds for all $\phi,\psi\in \DIDE^k_{[r,r']}$ with $\mmm_\infty^\dind(\phi),\: \mmm_\infty^\dind(\psi),\:\mmm_\infty^\dind(\psi-\phi)\leq 1$.
	\end{proposition}
	\begin{proof}
	We let $\pp\leq\uu\in \SEM$ and $V$ be as in Lemma \ref{fhfhfhffhaaaa} for $\compact\equiv \{e\}$ there, i.e., we have $\OB_{\uu,1}\subseteq \chart(V)$. By Lemma \ref{posdpospodspoaaaa}, there exist $\uu\leq \qq\in \SEM$ and $\dindo\llleq k$, such that 
\begin{align*}
	\textstyle\qqq_\infty^\dindo(\chi)\leq 1\quad\:\:\text{for}\quad\:\:\chi\in \DIDE_{[r,r']}^k \qquad\quad\Longrightarrow\qquad\quad (\uu\cp\chart)\big(\innt_r^\bullet \chi\big) \leq 1\qquad\quad\Longrightarrow\qquad\quad\innt_r^\bullet \chi\in V.
\end{align*}	 
Then, for $\phi,\psi$ with $\qqq^\dindo_\infty(\phi)\leq 1,\: \qqq^\dindo_\infty(\psi)\leq 1$, Lemma \ref{fhfhfhffhaaaa} applied to $q\equiv\innt_r^\bullet\phi,\:q'\equiv\innt_r^\bullet\psi,\: h\equiv\innt_r^\bullet\phi\in V$, and $g\equiv e$ gives 
	\begin{align*}
		\textstyle\pp\big(\chart\big(\innt_r^\bullet\phi\big)-\chart\big(\innt_r^\bullet\psi\big)\big)\leq
		\textstyle\uu\big(\chart_{\innt_r^\bullet\phi}\big(\innt_r^\bullet\phi\big)-\chart_{\innt_r^\bullet\phi}\big(\innt_r^\bullet\psi\big)\big)= \textstyle(\uu\cp \chart)\big([\innt_r^\bullet\phi]^{-1}[\innt_r^\bullet\psi]\big).
	\end{align*} 
We choose $\uu\leq \mm\in \SEM$ and $\dindu\llleq k$ as in Lemma \ref{jlkfdsjlkfdsjklfdsjklfsdjskl} for $\pp\equiv \uu$ there, define $\dind:=\max(\dindo,\dindu)$; and can additionally assume that $\pp\leq \qq\leq \mm$ holds. 
Then, for $\phi,\psi\in \DIDE_{[r,r']}^k$ with $\mmm^\dind_\infty(\phi)\leq 1,\: \mmm^\dind_\infty(\psi),\: \mmm^\dind_\infty(\psi-\phi)\leq 1$, we have
\begin{align*}
	\textstyle\pp\big(\chart\big(\innt_r^\bullet\phi\big)-\chart\big(\innt_r^\bullet\psi\big)\big)\leq\textstyle(\uu\cp \chart)\big([\innt_r^\bullet\phi]^{-1}[\innt_r^\bullet\psi]\big)\leq \int_r^{\bullet}\mmm(\psi(s)-\phi(s))\:\dd s 
\end{align*}  
by Lemma \ref{jlkfdsjlkfdsjklfdsjklfsdjskl}.
\end{proof}
We finally observe that
\begin{lemma}
\label{fddfxxxxfd}
Suppose that $G$ is \emph{k-continuous};   
and let $\phi\in \DIDE_{[r,r']}$ be fixed. Then, for each open neighbourhood $V\subseteq G$ of $e$,  
there exist $\mm\in \SEM$ and $\dind\llleq k$, such that 
\begin{align*}
	\textstyle\mmm^\dind_\infty(\psi-\phi)\leq 1\quad\:\:\text{for}\quad\:\: \psi\in \DIDE^k_{[r,r']}\qquad\quad\Longrightarrow\qquad\quad \innt_r^\bullet \psi \in \innt_r^\bullet \phi\cdot V.
\end{align*} 
\end{lemma}
\begin{proof}
	We fix $V\lleq \pp\in \SEM$; and choose $\qq\in \SEM$, $\dind\llleq k$ as in Lemma \ref{posdpospodspoaaaa}. Then, Lemma \eqref{opopsopsdopds} provides us with some $\mm\in \SEM$, such that 
	\begin{align*}
		\qqq^\dind_\infty\big(\Ad_{[\innt_r^\bullet \phi]^{-1}}(\chi)\big)\leq \mmm^\dind_\infty(\chi)\qquad\quad\forall\: \chi\in \DIDE^k_{[r,r']}
	\end{align*}
	holds. 
   Then, for each $\psi\in \DIDE_{[r,r']}^k$ with $\mmm^\dind_\infty(\psi-\phi)\leq 1$, we obtain from \ref{kdskdsdkdslkds} and Lemma \ref{posdpospodspoaaaa} that
	\begin{align*}
	\textstyle(\pp\cp\chart)\big(\big[\innt_r^t \phi\big]^{-1} \big[\innt_r^t\psi\big]\big)&\textstyle=(\pp\cp\chart)\big(\innt_r^t\Ad_{[\innt_r^\bullet\phi]^{-1}}(\psi-\phi)\big)\leq 1\qquad\quad\forall\: t\in [r,r'] 
\end{align*}
holds; thus, $[\innt_r^\bullet\phi]^{-1} [\innt_r^\bullet\psi] \in V$, implying $\innt_r^\bullet\psi\in \innt_r^\bullet \phi\cdot V$. 
\end{proof}
\subsection{Piecewise Integrable Curves}
\label{opsdpods}
For $k\in \NN\sqcup\{\lip,\infty\}$ and $[r,r']\in \COMP$, we let $\DP^k([r,r'],\mg)$ denote the set of all $\phi\colon [r,r']\rightarrow \mg$, such that there exist $r=t_0<{\dots}<t_n=r'$ and $\phi[p]\in \DIDE^k_{[t_p,t_{p+1}]}$ with
\begin{align}
\label{opopooppo}
	\phi|_{(t_p,t_{p+1})}=\phi[p]|_{(t_p,t_{p+1})}\qquad\quad\forall\: p=0,\dots,n-1.
\end{align}  
In this situation, we define $\innt_r^r\phi:=e$, as well as 
\begin{align}
\label{defpio}
	\textstyle\innt_r^t\phi&\textstyle:=\innt_{t_{p}}^t \phi[p] \cdot \innt_{t_{p-1}}^{t_p} \phi[p-1]\cdot {\dots} \cdot \innt_{t_0}^{t_1}\phi[0]\qquad\quad \forall\: t\in (t_{p}, t_{p+1}].
\end{align} 
A standard refinement argument in combination with \ref{pogfpogf} then shows that this is well defined; i.e., independent of any choices we have made. 
It is furthermore not hard to see that (cf.\ Appendix \ref{appB}) for $\phi,\psi\in \DP^k([r,r'],\mg)$, we have 
 $\Ad_{[\innt_r^\bullet\phi]^{-1}}(\psi -\phi)\in \DP^k([r,r'],\mg)$ with 
\begin{align}
\label{dfdssfdsfd}
	\textstyle\big[\innt_r^t \phi\big]^{-1}\big[\innt_r^t \psi\big]=\innt_r^t \Ad_{[\innt_r^\bullet\phi]^{-1}}(\psi -\phi)
	\qquad\quad\forall\: t\in [r,r'].
\end{align} 
We now finally will extend Lemma \ref{opdfopfdopfdpoas} for the {\rm 0$\pkt$k}-continuous case to the piecewise integrable setting. 
For this, we fix (a bump function) $\rhon\colon [0,1]\rightarrow [0,2]$ smooth with 
\begin{align}
\label{odsposdposdpopospoadsasa}
	\textstyle\rhon|_{(0,1)}>0,\:\:\int_0^1 \rhon(s)\:\dd s=1\qquad\quad\:\:\text{as well as}\qquad\quad\:\: \rhon^{(k)}(0)=0=\rhon^{(k)}(1)\qquad\forall\: k\in \NN.
\end{align} 
Then, $[r,r']\in \COMP$ and $r=t_0<{\dots}<t_n=r'$ given, we let
\begin{align*}
	\rho_p\colon [t_p,t_{p+1}]\rightarrow [0,2],\qquad t\mapsto \rhon(|t-t_p|\slash|t_{p+1}-t_p|)\qquad\qquad\forall\: p=0,\dots,n-1;
\end{align*}
and define $\rho\colon [r,r']\rightarrow [0,2]$ by 
\begin{align*}
	\rho|_{[t_p,t_{p+1}]}:=\rho_p\qquad\quad \forall\: p=0,\dots,n-1. 
\end{align*}
Then, $\rho$ is smooth with $\rho^{(k)}(t_p)=0$ for each $k\in \NN$, $p=0,\dots,n$; and \eqref{substitRI} shows that
\begin{align*}
	\textstyle\varrho\colon [r,r']\rightarrow [r,r'],\qquad t\mapsto r+\int_r^t \rho(s)\:\dd s 
\end{align*}
holds, with $\varrho(t_p)=t_p$ for $p=0,\dots,n-1$.  
We are ready for
\begin{lemma}
\label{pofdspospods}
	Suppose that $G$ is {\rm 0$\pkt$k}-continuous for $k\in \NN\sqcup \{\lip, \infty\}$; and let $[r,r']\in \COMP$ be fixed. Then, for each $\pp\in \SEM$, there exists some $\pp\leq \mm\in \SEM$, such that 
\begin{align*}
	\textstyle\mmm_\infty(\phi)\leq 1\quad\:\:\text{for}\quad\:\: \phi\in \DP^k([r,r'],\mg)\qquad\quad\Longrightarrow\qquad\quad (\pp\cp\chart)(\innt_r^\bullet\phi)\leq \int_r^\bullet \mmm(\phi(s))\:\dd s.  
\end{align*}
\end{lemma}
\begin{proof}
We let $\pp\leq\qq\in \SEM$ be as in Lemma \ref{opdfopfdopfdpoas} ($s\equiv 0$); and define $\mm:=2\cdot \qq$. 
Then, for $\phi\in \DP^k([r,r'],\mg)$ with $\mmm_\infty(\phi)\leq 1$ given, we choose $r=t_0<{\dots}<t_n=r'$ and $\phi[0],\dots,\phi[n-1]$ as in \eqref{opopooppo},  and fix $\mu[p]\colon  I_p\rightarrow G$ 
of class $C^{k+1}$ ($I_p\subseteq \RR$ an open interval containing $[t_p,t_{p+1}]$) with 
\begin{align*}
	\phi[p]=\wt{\phi}[p]|_{[t_p,t_{p+1}]}\qquad\text{for}\qquad \wt{\phi}[p]\equiv \Der(\mu[p])\qquad\qquad\forall\: 0\leq p\leq n-1.
\end{align*}
We construct $\rho$ and $\varrho$ as described above; and define $\varrho[p]\in C^\infty(I_p,\RR)$ by 
\begin{align*}
	\varrho[p]|_{(-\infty,t_p)\cap I_p}:=\varrho(t_p)\qquad\quad\varrho[p]|_{[t_p,t_{p+1}]} :=\varrho|_{[t_p,t_{p+1}]}\qquad\quad  \varrho[p]|_{(t_{p+1},\infty)\cap I_p}:=\varrho(t_{p+1})
\end{align*}
for $p=0,\dots,n-1$. It follows that (cf.\ Appendix \ref{appAaaaa}) $\psi:=\rho\cdot \rcK{\phi \cp\varrho}\in C^k([r,r'],\mg)$ holds, with 
\begin{align}
\label{zsazusazuzusazuzusazusazusazusa}
	\psi|_{[t_p,t_{p+1}]}=&\:\Der(\mu[p]\cp\varrho[p]|_{[t_p,t_{p+1}]})\in \DIDE^k_{[t_p,t_{p+1}]}\qquad\quad\forall\:p=0,\dots,n-1.
\end{align}
Then, Lemma \ref{sddssdsdsd} shows that $\psi\in \DIDE^k_{[r,r']}$ holds, with
\begin{align}
\label{podspodpodspods}
\begin{split}
	\textstyle\innt_r^t \psi &=\textstyle\innt_{t_{p}}^{t} \rho\cdot \rcK{\phi[p] \cp\varrho}\cdot   \innt_{t_{p-1}}^{t_{p}} \rho\cdot \rcK{\phi[p-1] \cp\varrho}  \cdot {\dots}\cdot \innt_{t_{0}}^{t_1} \rho\cdot \rcK{\phi[0] \cp\varrho}\\
	&\stackrel{\ref{subst}}{=}\textstyle\innt_{t_{p}}^{\varrho(t)} \phi[p] \cdot \innt_{t_{p-1}}^{t_p} \phi[p-1] \cdot {\dots} \cdot \innt_{t_0}^{t_1}\phi[0]\\
	&\textstyle=\innt_r^{\varrho(t)}\phi
\end{split}
\end{align} 
for each $t\in (t_p,t_{p+1}]$ and $0\leq p\leq n-1$. 
Since  
\begin{align*}
	\qqq_\infty(\psi)= 1/2\cdot \mmm_\infty(\psi)\leq \mmm_\infty(\phi)\leq  1
\end{align*}
holds by construction, Lemma \ref{opdfopfdopfdpoas} provides us with 
\begin{align*}
	\textstyle(\pp\cp\chart)(\innt_r^{\varrho(t)} \phi)\stackrel{\eqref{podspodpodspods}}{=}(\pp\cp\chart)(\innt_r^t \psi)\leq\int_r^t \qqq(\psi(s))\: \dd s\leq \int_r^t \mmm(\psi(s))\: \dd s = \int_r^{\varrho(t)} \mmm(\phi(s))\: \dd s;
\end{align*}
whereby the last step is due to \eqref{substitRI} and \eqref{opofdpopfd}.
\end{proof}

\section{Local $\boldsymbol{\mu}$-convexity}
\label{podposddpospodpods}
In this section, we show that 0-continuity can be encoded in a property of the Lie group multiplication. 
More specifically, we will show that
\begin{theorem}
\label{lfdskfddflkfdfd}
$G$ is {\rm 0}-continuous \deff $G$ is locally $\mu$-convex \deff $G$ is $\rm0\pkt\infty$-continuous.
\end{theorem}
\noindent
Here, 
$G$ is said to be {\bf locally $\boldsymbol{\mu}$-convex} \defff for each $\uu\in \SEM$, there exists some $\uu\leq \oo\in \SEM$, such that
\begin{align}
\label{aaajjhguoiuouo}
\begin{split}
	(\uu\cp\chart)(\chart^{-1}(X_1)\cdot {\dots}\cdot \chart^{-1}(X_n))\leq \oo(X_1)+{\dots}+\oo(X_n)
\end{split}
\end{align} 
holds for each $X_1,\dots,X_n\in E$ with $\oo(X_1)+{\dots}+\oo(X_n) \leq 1$. Due to the above Theorem \ref{lfdskfddflkfdfd}, this definition does not depend on the explicit choice of $\chart$.
\vspace{6pt}

\noindent
For instance,
\begin{example}
\label{exxxcon}
\begingroup
\setlength{\leftmargini}{17pt}
\begin{enumerate}
\item[]
\item
\label{exxxcon0}
$E\slash\Gamma$ is locally $\mu$-convex, for each discrete subgroup $\Gamma\subseteq E$, cf.\ Appendix \ref{QuotConstrained}. 
\item
\label{exxxcon1}
Banach-Lie groups are locally $\mu$-convex, cf.\ Proposition 14.6 in \cite{HGGG}, or Appendix \ref{BanachConstrained}.
\item
\label{exxxcon2}
The unit group\footnote{Confer \cite{HGWA} for a proof of the fact that $\MAU$ is a Lie group.} $\MAU$ of a 
continuous inverse algebra $\MA$ fulfilling the condition $(*)$ from \cite{HGIA} is locally $\mu$-convex, cf.\ Appendix \ref{HGConstrained}. We recall that the condition $(*)$ imposed on the algebra multiplication in \cite{HGIA} states that for each $\vv\in \SEM$, there exists some $\vv\leq \ww\in \SEM$ with
\begin{align}
\label{invACond}
	\vv(a_1\cdot {\dots} \cdot a_n)\leq \ww(a_1)\cdot{\dots}\cdot \ww(a_n)\qquad\quad\forall\: a_1,\dots,a_n\in \MA
\end{align}
for each $n\geq 1$.\footnote{In the Theorem proven in \cite{HGIA}, $(\MA,+)$ is additionally assumed to be Mackey complete. We will discuss this condition in Example \ref{opfdopdfpofddfop}.\ref{exxxxcon2} in Sect.\ \ref{COMPAPPR}.}
 \hspace*{\fill}$\ddagger$
\end{enumerate}
\endgroup
\noindent
\end{example}
\noindent
We break up the proof of Theorem \ref{lfdskfddflkfdfd} into the two directions.
\subsection{The Triangle Inequality}
We first show that \eqref{aaajjhguoiuouo} holds if $G$ is 0$\pkt\infty$-continuous. For this, we recall that
	\begin{align}
	\label{dovcxjdsfjgf}
		\Der(\chartinv\cp\gamma)=\dermapdiff(\gamma,\dot\gamma)\qquad\quad\forall\: \gamma\in C^{1}([r,r'],\V),\:\: [r,r']\in \COMP
	\end{align}
	holds, for $\w$ defined by \eqref{opopopop1}; and conclude from \eqref{cpocpoxjdsndscxaaabc} that
\begin{lemma}
\label{poposposaadccx}
	For each $\mm\in \SEM$, there exists some $\V\lleq\oo\in \SEM$ with $\mm\leq \oo$, such that  
\begin{align*}	
	\textstyle\mmm(\Der(\chartinv\cp\gamma))\leq \oo(\dot\gamma)\quad\:\:\text{holds for each}\quad\:\: \gamma\in \bigsqcup_{[r,r']\in \COMP}C^1([r,r'],E)\quad\:\:\text{with}\quad\:\: \im[\gamma]\subseteq\OB_{\oo,1}. 
\end{align*}	
\end{lemma}
\begin{proof}
Up to renaming seminorms, this is clear from \eqref{dovcxjdsfjgf} and \eqref{cpocpoxjdsndscxaaabc}.
\end{proof}
We obtain that 
\begin{lemma}
\label{hgffhhfhg}
$G$ is locally $\mu$-convex if G is $\rm 0\pkt\infty$-continuous.
\end{lemma}
\begin{proof}
For $\uu\equiv \pp\in \SEM$ fixed, we let $\pp\leq \mm$ be as in Lemma \ref{pofdspospods} for $[r,r']\equiv[0,2]$ there; and choose $\mm\leq \oo\in \SEM$ as in Lemma \ref{poposposaadccx}. 
Then, for $X_1,\dots,X_n\in E$ with $\oo(X_1)+{\dots}+\oo(X_n)=: \varepsilon\leq 1$ fixed, we define $Y_p:=X_{n-p}$ for $p=0,\dots,n-1$ as well as $Y_n:=0$. We let $\emptyset\neq \JI\subseteq \{0,\dots,n\}$ denote the set of all indices $0\leq p\leq n$ with $\oo(Y_p)=0$; and denote its cardinality by $\mathrm{d}:=|\JI|\geq 1$. 
We define
$$
\delta_p:= 
\begin{cases}
\:1/\mathrm{d} &\text{for each}\quad\:\:\: p\in \JI\\
\oo(Y_{p})\:\:\: &\text{for each}\quad\:\:\:  p \in \{0,\dots,n\}-\JI;
\end{cases}
$$
and let $t_0:=0$, as well as $t_p:=\delta_0+{\dots}+\delta_{p-1}$ for $p=1,\dots,n+1$. We furthermore consider
\begin{align*}
	 \phi[p]:=\Der(\chartinv\cp \gamma[p])  \qquad\quad \text{with}\qquad\quad \gamma[p]\colon [t_p,t_{p+1}]\colon t\mapsto (t-t_p)\cdot \delta_{p}^{-1}\cdot Y_{p}
\end{align*}
for $p=0,\dots,n$; and define $\phi\in \DP^\infty([0,2],\mg)$ by $\phi|_{[1+\varepsilon,2]}:=0$, as well as
\begin{align*}
	\phi|_{[t_p,t_{p+1})}:=\phi[p]|_{[t_p,t_{p+1})}\qquad\quad\forall\: p=0,\dots,n. 
\end{align*}
Then, $\oo_\infty(\gamma[p])\leq 1$ holds by construction for each $0 \leq p\leq n$; so that  
 Lemma \ref{poposposaadccx} shows  
 $$
\mmm_\infty\big(\phi|_{[t_p,t_{p+1})}\big)\leq \delta_p^{-1}\cdot\oo(Y_p) =  
\begin{cases}
0 &\:\:\text{for}\quad\:\: p\in \JI\\
1 &\:\:\text{for}\quad\:\: p\in \{0,\dots,n\}-\JI.
\end{cases}
$$
We thus have $\mmm_\infty(\phi)\leq 1$, as well as 
\begin{align*}
	\textstyle\int \mmm(\phi(s))\:\dd s= \sum_{p\in \{0,\dots,n\}-\mathrm{J}}\int \mmm(\phi[p](s))\:\dd s \leq \sum_{p\in \{0,\dots,n\}-\mathrm{J}}\delta_p=\oo(X_1)+{\dots}+\oo(X_n)=\varepsilon;
\end{align*} 
so that Lemma \ref{pofdspospods} shows
\begin{align*}
	\varepsilon&\textstyle\geq  (\uu\cp\chart)(\innt\phi)\textstyle\\
	&\textstyle=(\uu\cp\chart)\big(\innt_{1+\varepsilon}^{2}\phi\cdot \innt_{t_{n}}^{t_{n+1}}\phi[n]\cdot {\dots}\cdot \innt_{t_0}^{t_1}\phi[0]\:\big)\\
	&\textstyle=(\uu\cp\chart)\big(\chart^{-1}(Y_n)\cdot {\dots}\cdot \chart^{-1}(Y_0)\big)\\[1pt]
	&\textstyle =(\uu\cp\chart)\big(\chart^{-1}(X_1)\cdot {\dots}\cdot \chart^{-1}(X_n)\big),
\end{align*}
from which the claim is clear. 
\end{proof}
\subsection{Continuity of the Integral}
Let us next show that $G$ is 0-continuous if it is locally $\mu$-convex. 
For this, we recall that 
\begin{align}
\label{pofdpofdpofdpofd}
	\textstyle\dot\gamma=\dermapinvdiff(\gamma,\phi)\qquad\quad\text{holds for}\qquad\quad\gamma:=\chart\cp \innt_r^\bullet\phi,
\end{align}
for each $\phi\in \DIDE_{[r,r']}^k$ with $\innt_r^\bullet\phi\in \U$; and conclude from \eqref{dssdsdsdsdsddsds} that
\begin{lemma}
\label{pofdpofdpofdpofdaaa}
For each $\oo\in \SEM$, there exists some $\V\lleq\ww\in \SEM$ with $\oo\leq \ww$, such that 
\begin{align*}
	\textstyle (\oo\cp\chart)\big(\innt_r^\bullet\phi\big)\leq \int_r^\bullet \www(\phi(s))\:\dd s \quad\:\:\:\text{holds for each}\quad\:\:\:\phi\in \DIDE_{[r,r']}\quad\:\:\:\text{with}\quad\:\:\: \innt_r^\bullet\phi\in\chartinv(\OB_{\ww,1}).
\end{align*}
\end{lemma}
\begin{proof}
We choose $\oo\leq \ww\in \SEM$ as in \eqref{dssdsdsdsdsddsds} for $\vv\equiv \oo$ there. Then, the rest is clear from \eqref{pofdpofdpofdpofd} and Lemma \ref{ofdpofdpopssssaaaasfffff}.
\end{proof}
In addition to that, we observe that
\begin{lemma}
\label{cxcxcxxccxcxxccx}
For each $\ww\in \SEM$ and  $\phi\in \DIDE_{[r,r']}$, there exist $r=t_0<{\dots}<t_n=r'$ with
\begin{align*}
	\textstyle(\ww\cp\chart)\big(\innt_{t_p}^\bullet \phi|_{[t_p,t_{p+1}]}\big)\leq 1\qquad\quad\forall\: p=0,\dots,n-1.
\end{align*}
\end{lemma}
\begin{proof}
	We fix $\mu\colon I\rightarrow G$ ($I\subseteq \RR$ open with $[r,r']\subseteq I$) of class $C^1$ with $\Der(\mu)|_{[r,r']}=\phi$, choose $d>0$ such small that $K_d\equiv [r-d,r'+d]\subseteq I$ holds, and define 
\begin{align*}
	\alpha\colon I\times I\ni(t,s)\mapsto \mu(t)\cdot \mu(s)^{-1}\in G.  
\end{align*}	
Since $[r,r']$ is compact, and since $\alpha$ is continuous with $\alpha(t,t)=e$ for each $t\in [r,r']$, to each open neighbourhood $U$ of $e$, there exists some $0<\delta_U\leq d$, such that 
	\begin{align*}
	\textstyle U\ni \alpha(t+s,t)=\textstyle \innt_t^{t+s}\Der(\mu)\qquad\quad\forall\: t\in [r,r'],\:\:0\leq s\leq \delta_U
	\end{align*}
	holds; from which the claim is clear. 
\end{proof}
We conclude from Lemma \ref{pofdpofdpofdpofdaaa} and Lemma \ref{cxcxcxxccxcxxccx} that 
\begin{proposition}
\label{aaapofdpofdpofdpofd}
Suppose that $G$ is locally $\mu$-convex. Then, for each $\pp\in \SEM$, there exists some $\pp\leq \qq\in \SEM$, such that
\begin{align*}
	\textstyle\int\qqq(\phi(s))\:\dd s \leq 1\quad\:\text{for}\quad\:\phi\in \DP^0([r,r'],\mg)\qquad\:\:\Longrightarrow\qquad\:\:	\textstyle(\pp\cp\chart)\big(\innt_r^\bullet\phi\big)\leq \int_r^\bullet \qqq(\phi(s))\:\dd s,
\end{align*} 
for each $[r,r']\in \COMP$.
\end{proposition}
\begin{proof}
For $\uu\equiv\pp$ fixed, we let $\uu\leq \oo\in \SEM$ be as in \eqref{aaajjhguoiuouo}; and choose $\oo\leq\ww\equiv \qq\in \SEM$ as in Lemma \ref{pofdpofdpofdpofdaaa}. 
Then, since $\phi|_{[\ell,\ell']}\in \DP^0([\ell,\ell'],\mg)$ holds for each $\phi\in \DP^0([r,r'],\mg)$ and $\COMP\ni [\ell,\ell']\subseteq [r,r']\in \COMP$,  
the claim follows if we show that 
\begin{align*}
	\textstyle\int\www(\phi(s))\:\dd s \leq 1\quad\:\text{for}\:\quad\phi\in \bigsqcup_{[r,r']\in \COMP}\DP^0([r,r'],\mg) \qquad\:\:\Longrightarrow\qquad\:\: \textstyle(\uu\cp\chart)\big(\innt\phi\big)\leq \int \www(\phi(s))\:\dd s.
\end{align*}
To verify this, we fix $\phi\in \DP^0([r,r'],\mg)$ with $\int \www(\phi(s))\:\dd s\leq 1$; and let $r=t_0<{\dots}<t_n=r'$ as well as $\phi[0],\dots,\phi[n-1]$ be as in \eqref{opopooppo}. By Lemma \ref{cxcxcxxccxcxxccx}, we can refine this decomposition in such a way  that
\begin{align*} 
	\textstyle\mu[p]\colon [t_p,t_{p+1}]\ni t\mapsto \innt_{t_p}^{t}\phi[p]\in \chartinv(\OB_{\ww,1}) \qquad\quad\forall\: p=0,\dots,n-1
\end{align*}
holds; so that Lemma \ref{pofdpofdpofdpofdaaa} shows
\begin{align*}
	\textstyle(\oo\cp\chart)(\mu[p](t_{p+1}))\leq \int_{t_p}^{t_{p+1}} \www(\phi(s))\:\dd s\qquad\quad\forall\: p=0,\dots, n-1. 
\end{align*}
We define $X_{n-p}:= (\chart\cp\mu[p])(t_{p+1})$ for each $0\leq p\leq n-1$; and obtain 
\begin{align*}
	\textstyle\oo(X_1)+{\dots}+\oo(X_n)\leq\int\www(\phi(s))\:\dd s\leq 1.
\end{align*}
Then, \eqref{aaajjhguoiuouo} provides us with
\begin{align*}
	\textstyle(\uu\cp\chart)(\innt \phi)
	&=(\uu\cp\chart)(\chart^{-1}(X_1)\cdot {\dots}\cdot \chart^{-1}(X_n))\textstyle\leq \oo(X_1)+{\dots}+\oo(X_n)\leq\int\www(\phi(s))\:\dd s;
\end{align*}
which proves the claim.
\end{proof}
We are ready for the 
\begin{proof}[Proof of Theorem \ref{lfdskfddflkfdfd}]
	Clearly, $G$ is $\rm 0\pkt\infty$-continuous if G is 0-continuous. 
	Moreover, if $G$ is $\rm 0\pkt\infty$-continuous, then $G$ is locally $\mu$-convex by Lemma \ref{hgffhhfhg}. Finally, if $G$ is locally $\mu$-convex, then $\EVE_{[0,1]}^0$ is $C^0$-continuous at zero by Proposition \ref{aaapofdpofdpofdpofd}; so that $G$ is $0$-continuous by Lemma \ref{klklllkjlaaa}.
\end{proof} 

\section{Completeness and Approximation}
\label{COMPAPPR}
In this section, we discuss completeness properties  of Lie groups; and prove certain approximation statements for continuous-, and Lipschitz curves. Both will be relevant for our investigation of semiregularity in Sect.\ \ref{sec:confined}. 
\subsection{Completeness Conditions}
\label{jjdsjksdljdsjldsljdsljdsds}
A sequence $\{g_n\}_{n\in \NN}\subseteq G$ is said to be a 
\begingroup
\setlength{\leftmargini}{12pt}
\begin{itemize}
\item
{\bf Cauchy sequence} \defff for each $\pp\in \SEM$ and $\varepsilon>0$, there exists some $p\in \NN$ with
\begin{align}
\label{djdslkjdsdslkdsklds}
	(\pp\cp\chart)(g_m^{-1}\cdot g_n)\leq \varepsilon\qquad\quad\forall\: m,n\geq p.
\end{align}
We then clearly can assume that $g_m^{-1}\cdot g_n\in \U$ holds for all $m,n\in \NN$ in the following.
\item
{\bf Mackey-Cauchy sequence} \defff 
 \begin{align}
\label{psdopsdopds}
 	(\pp\cp\chart)(g^{-1}_m\cdot g_{n})\leq \mackeyconst_\pp\cdot \lambda_{m,n}\qquad\quad\hspace{2.4pt}\forall\: m,n\geq \mackeyindex_\pp,\:\:\pp\in\SEM
 \end{align}
 holds for certain $\{\mackeyconst_\pp\}_{\pp\in \SEM}\subseteq \RR_{\geq 0}$, $\{\mackeyindex_\pp\}_{\pp\in \SEM}\subseteq \NN$, and $\RR_{\geq 0}\supseteq \{\lambda_{m,n}\}_{(m,n)\in \NN\times \NN}\rightarrow 0$.
 
 Clearly, each Mackey-Cauchy sequence is a Cauchy sequence.
\end{itemize}
\endgroup
\begin{remark}
\label{popodspodspodssds}
	It is straightforward from Lemma \ref{kldskldsksdklsdl} and Lemma \ref{ofdpofdpopssssaaaasfffff} (applied to coordinate changes) that these definitions are independent of the explicit choice of $\chart$, cf.\ Appendix \ref{Mackeyind}. 
\end{remark}
\noindent
 We say that $G$ is 
 \begingroup
\setlength{\leftmargini}{12pt}
\begin{itemize}
\item
{\bf sequentially complete} \defff each Cauchy sequence in $G$ converges in $G$.
\item
{\bf Mackey complete} \defff each Mackey-Cauchy sequence in $G$ converges in $G$.
\end{itemize}
\endgroup
\noindent
We say that a locally convex vector space $F$ is sequentially\slash Mackey complete \defff $F$ is sequentially\slash Mackey complete when considered as the Lie group $(F,+)$. Obviously, these definitions coincide with the standard definitions given in the literature.
\begin{remark}
\label{confkjdsjdskjsdkjds}
\begingroup
\setlength{\leftmargini}{17pt}
\begin{enumerate}
\item[]
\item
\label{confev12}
Since each Mackey-Cauchy sequence is a Cauchy sequence, sequentially completeness of $G$ implies Mackey completeness of $G$.
\item
\label{confev22}
	It is straightforward from the definitions that a Cauchy\slash{Mackey-Cauchy} sequence converges \deff one of its subsequences converges.
\item
\label{confev23}
	If $\{g_n\}_{n\in \NN}\subseteq G$ is a Cauchy\slash{Mackey-Cauchy} sequence, then $\{h\cdot g_n\}_{n\in \NN}\subseteq G$ is a Cauchy\slash{Mackey-Cauchy} sequence for each $h\in G$; and (evidently) $\{g_n\}_{n\in \NN}$ converges \deff $\{h\cdot g_n\}_{n\in \NN}$ converges for each $h\in G$. 
\item
\label{confev24}
If $\{g_n\}_{n\in \NN}\subseteq G$ is a Cauchy\slash{Mackey-Cauchy} sequence, and $U\subseteq G$ an open neighbourhood of $e$, then there exists some $q\in \NN$ with $\{g^{-1}_q\cdot g_n\}_{n\geq q}\subseteq U$. 
	  
Thus, in order to show that $G$ is sequentially\slash{Mackey-Cauchy}, by the previous two points, it suffices to verify convergence of each Cauchy\slash{Mackey-Cauchy} sequence that is contained in a fixed open neighbourhood $U$ of $e$.
 \hspace*{\fill}$\ddagger$
\end{enumerate}
\endgroup
\end{remark}
\begin{example}
\label{opfdopdfpofddfop}
\begingroup
\setlength{\leftmargini}{17pt}
\begin{enumerate}
\item[]
\item
\label{exxxxcon0}
Let $\Gamma$ be a discrete subgroup of $(E,+)$. Then, $E$ is sequentially\slash{Mackey} complete \deff $E\slash \Gamma$ is sequentially\slash{Mackey} complete,   
cf.\ Appendix \ref{QuotientMackey}. 
\item
\label{exxxxcon1}
Banach Lie groups are sequentially complete, cf.\ Appendix \ref{BanachMackey}.
\item
\label{exxxxcon2}
The unit group $\MAU$ of a continuous inverse algebra $\MA$ fulfilling the condition $(*)$ from \cite{HGIA} (i.e., condition  \eqref{invACond}) is sequentially/Mackey complete if $(\MA,+)$ is sequentially/Mackey complete,  
cf.\ Appendix \ref{HGMackey}.
 \hspace*{\fill}$\ddagger$
\end{enumerate}
\endgroup
\noindent
\end{example}
\noindent
We now are going to show that
\begin{theorem}
\label{dsjkhjsdjhsd}
$G$ is Mackey complete if $G$ is $C^\infty$-semiregular.
\end{theorem}
\begin{remark}
\label{sequrem}
\begingroup
\setlength{\leftmargini}{17pt}
\begin{enumerate}
\item[]
\item
\label{sequrem1}
The idea of the proof of Theorem \ref{dsjkhjsdjhsd} is to construct some $\phi\in C^\infty([0,1],\mg)$ whose integral $\innt\phi$ is the limit of a (subsequence of a) given Mackey-Cauchy sequence $\{g_n\}_{n\in\NN}\subseteq G$. Roughly speaking,  
we will use the substitution formula \ref{subst} in order to glue together smooth curves whose integrals equal $g_n^{-1}\cdot g_{n-1}$ via suitable bump functions. Here, we will use that, passing to a subsequence if necessary, we can achieve that $(\pp\cp \chart)(g_n^{-1}\cdot g_{n-1})$ decreases suitably fast; namely, (up to a factor $\mackeyconst_\pp$) in the same way for all seminorms -- This ensures that the so-constructed $\phi$ is defined and smooth at $1$ (where all of its derivatives must necessarily be zero).
\item
\label{sequrem2}
An analogous result cannot hold for the $C^0$-semiregular case; i.e., we cannot have that $C^0$-semiregularity implies sequentially completeness. This can be seen immediately by considering the special situation where $G$ equals a Hausdorff locally convex vector space $(E,+)$ as then integrability of all continuous curves -- which is \emph{integral completeness} in the sense of \cite{HGGG} -- is equivalent to the ``metric convex compactness property'' \cite{Weiz} that, in general, is strictly weaker than sequentially completeness \cite{Voigt}, cf.\ proof of Theorem C.(d) in \cite{HGGG}.  

Indeed, the strategy, sketched in \ref{sequrem1} for the $C^\infty$-semiregular case does not work out in the $C^0$-semiregular situation; i.e., given a Cauchy sequence $\{g_n\}_{n\in\NN}\subseteq G$, we cannot apply the same procedure to construct some $\phi\in C^0([0,1],\mg)$ whose integral is the limit of (a subsequence of) $\{g_n\}_{n\in \NN}$. The problem is that by passing to a subsequence, we can only assure that $\lim_{t\rightarrow 1}(\ppp\cp\phi)(t)=0$ holds for finitely many seminorms $\pp\in \SEM$, but not for all of them. \hspace*{\fill}$\ddagger$ 
\end{enumerate}
\endgroup
\noindent
\end{remark}
\noindent 
Now, before we can prove Theorem \ref{dsjkhjsdjhsd}, we first need some preparation:
\begingroup
\setlength{\leftmargini}{12pt}
\begin{itemize}
\item
For $\gamma\colon [r,r']\ni t\rightarrow |t-r|\cdot Y$ with $[r,r']\cdot Y\subseteq\V$, we have
\begin{align*}
	\phi:=\Der(\chartinv\cp\gamma)=
	\dermapdiff(\gamma,Y)=
	\omega[0](\gamma,Y)\in \DIDE_{[r,r']}^\infty.
\end{align*}
\vspace{-18pt}
\item
	For $\varrho\colon [r,r']\rightarrow [r,r']$ smooth, $p\in \NN$, and $\rho\equiv\dot\varrho$, we define
	\begin{align*}
	\textstyle C[\rho,p]:=\max_{0\leq m,n \leq p}(\sup\{ |\rho^{(m)}(t)|^{n+1}\:| \: t\in [r,r']\})\qquad\quad \forall\: p\in \NN;
\end{align*}
	and obtain from \ref{chainrule}, \ref{productrule} that 
\begin{align*}
	\textstyle(\rho\cdot \rcK{\phi\cp\varrho})^{(p)}&\textstyle=\sum_{q,m,n=0}^{p} h_p(q,m,n)\cdot \big(\he\rho^{(m)}\big)^{n+1}\cdot \dermapdiff[q](\gamma\cp\varrho,Y,\dots,Y)\\[1pt]
	&\leq (p+1)^3\cdot C[\rho,p]\cdot  \dermapdiff[q](\gamma\cp\varrho,Y,\dots,Y)
\end{align*}
holds, for a map $h_p\colon (0,\dots,p)^3\rightarrow \{0,1\}$ that is independent of $\varrho,\rho,Y$.
\item
\label{assaassaasa3}
For $\vv\in \SEM$, we choose $\V\lleq \ww\in \SEM$ with $\vv\leq \ww$ as in \eqref{omegakll}; and conclude that
\begin{align*}
	\ww(Y),\: \ww(\gamma\cp\varrho)\equiv \ww(|\varrho-r|\cdot Y)\leq 1\qquad\text{implies}\qquad \vvv\big((\rho\cdot \rcK{\phi\cp\varrho})^{(q)}\big)\leq (p+1)^3\cdot C[\rho,p]\cdot \ww(Y)
\end{align*}
for $0\leq q\leq p$, for each fixed $p\in \NN$.
\end{itemize}
\endgroup
\noindent
Let now $\rhon\colon [0,1]\rightarrow [0,2]$ be as in Sect.\ \ref{opsdpods}, cf.\ \eqref{odsposdposdpopospoadsasa}; and suppose 
that we are given $\{Y_n\}_{n\in \NN}\subseteq \V$, as well as $\{t_n\}_{n\in \NN}\subseteq [0,1]$ strictly increasing with $t_0=0$. 
\vspace{6pt}

\noindent
Then, for each $n\in \NN$,
\begingroup
\setlength{\leftmargini}{12pt}
\begin{itemize}
\item
we let $\delta_n:=t_{n+1}-t_n$, and define 
\begin{align*}
	\kappa_n\colon [t_n,t_{n+1}]\ni t\mapsto \delta_n^{-1}\cdot |t-t_n|\in [0,1]\qquad\quad\text{as well as}\qquad\quad\gamma_n\colon [t_n,t_{n+1}]\ni t\mapsto |t-t_n|\cdot Y_n. 
\end{align*}
\vspace{-19pt}
\item
we let $\phi_n:=\Der(\chartinv\cp\gamma_n)$, and define 
\begin{align*}
	\textstyle\rho_n:=\rhon\cp\kappa_n\qquad\quad\text{as well as}\qquad\quad \varrho_n\colon [t_n,t_{n+1}]\ni t\mapsto  t_n+\int_{t_n}^t\rho_n(s)\:\dd s \in [t_n,t_{n+1}].
\end{align*}
\vspace{-19pt}
\item
we define $\phi\colon [0,1]\rightarrow \mg$ by $\phi(1):=0$, and $\phi|_{[t_n,t_{n+1}]}:=\rho_n\cdot \rcK{\phi_n\cp\varrho_n}$ for each $n\in \NN$.
\vspace{-6pt}
\end{itemize}
\endgroup 
\vspace{6pt}

\noindent
Then, the same arguments as in the proof of Lemma \ref{pofdspospods} show that   
$\phi|_{[0,t_{n}]}\in \DIDE_{[0,t_ {n}]}^\infty$ holds, with 
\begin{align}
\label{podfopdfopffd}
	\textstyle \innt_0^{t_n}\phi=\chartinv(\delta_{n-1}\cdot Y_{n-1})\cdot {\dots}\cdot \chartinv(\delta_0\cdot Y_0)\qquad\quad\forall\: n\geq 1.
\end{align}
Moreover, for $\vv\leq \ww\in \SEM$ as above, $p\in \NN$, and $n\in \NN$ with $\ww(Y_n)\leq 1$ (thus, $\ww(|\varrho_n-t_n|\cdot Y_n)\leq 1$), we have   
\begin{align}
\label{dspoopsdopdsaasas}
	\vvv\big((\rho_n\cdot\rcK{\phi_n\cp\varrho_n})^{(q)}\big)\leq (p+1)^3\cdot C[\rho_n,p]\cdot \ww(Y_n)\leq (p+1)^3\cdot \delta_n^{-(p+1)^2}\cdot C[\rhon,p]\cdot \ww(Y_n)
\end{align}
for $q=0,\dots,p$. 
\vspace{6pt}

\noindent
We are ready for the
\begin{proof}[Proof of Theorem \ref{dsjkhjsdjhsd}]
Let $\{g_n\}_{n\in \NN}\subseteq G$ be a Mackey-Cauchy sequence; and $U\subseteq G$ a symmetric open neighbourhood of $e$ with $U\cdot U\subseteq \U$. 
 By Remark \ref{confkjdsjdskjsdkjds}.\ref{confev22}, we can assume that $\{g_n\}_{n\in \NN}\subseteq U$ holds; and, by Remark \ref{confkjdsjdskjsdkjds}.\ref{confev22}, it suffices to show that a subsequence of $\{g_n\}_{n\in \NN}$ converges. Passing to a subsequence if necessary, we thus can assume that $\lambda_{n,n-1}\leq 2^{-n^2}$ holds for each $n\geq 1$. Then,  
\begingroup
\setlength{\leftmargini}{12pt}
\begin{itemize}
\item
We define $X_0:=0$, as well as $\V\ni X_n:=\chart\big(g_{n}^{-1}\cdot g_{n-1}\big)$ for each $n\geq 1$.
\item
For each $\ww\in \SEM$, we fix some $\wmackeyindex_\ww\in \NN$ with
\begin{align}
\label{usaiusauisaiusa}
	\ww(X_{n})\leq \mackeyconst_\ww\cdot 2^{-n^2}\qquad\quad\forall\:n\geq \wmackeyindex_\ww.
\end{align}
\vspace{-17pt}
\item
We define $t_0:=0$, as well as $t_n:=\sum_{k=1}^n 2^{-k}$ for $n \geq 1$; and obtain 
\begin{align}
\label{asklklsaklsalk}
	1/(1-h)\leq 2^{n+2}\qquad\quad\forall\:h\in [t_{n},t_{n+1}],\:\: n \in\NN,
\end{align}
from $1-h\geq 1-t_{n+1}=1-\sum_{k=1}^{n+1}2^{-k}=\sum_{k=n+2}^\infty 2^{-k}\geq 2^{-(n+2)}$.
\item
We let   
$Y_n:= 2^{n+1}\cdot X_n$ for each $n\in \NN$; and define $\delta_n$, $\gamma_n$, $\phi_n$, as well as $\phi\colon [0,1]\rightarrow\mg$ as described above; i.e., we have $\delta_n\equiv |t_{n+1}-t_n|=2^{-(n+1)}$ for each $n\in \NN$.
\end{itemize}
\endgroup
\noindent
	The claim now follows once we have shown that $\phi$ is smooth, because then  
	\eqref{podfopdfopffd} provides us with
	\begin{align*}
		\textstyle\big(\innt\phi \cdot g^{-1}_{0}\he\big)^{-1} &\textstyle\:\he=\:\he \lim_n \big(\big[\innt_{0}^{t_{n+1}}\phi\big]\cdot g^{-1}_{0}\he\big)^{-1}\\[-2pt]
		&\stackrel{\eqref{podfopdfopffd}}{=}\textstyle \lim_n \big(\chartinv(\delta_{n}\cdot Y_{n})\cdot{\dots}\cdot \chartinv(\delta_0\cdot Y_0)\cdot  g^{-1}_{0}\he\big)^{-1}\\[1pt]
		&\:\he=\:\he\textstyle \lim_n \big(\chartinv(X_{n})\cdot{\dots}\cdot \chartinv(X_1)\cdot  g^{-1}_{0}\he\big)^{-1}\\[3pt]
		&\:\he=\:\he\textstyle \lim_n g_{n}.
	\end{align*}	 
Since $\phi$ is smooth on $[0,1)$, here we only have to verify that 
\begin{align*}
\textstyle\lim_{[0,1)\ni h\rightarrow 1} \rcf{h} \cdot \phi^{(p)}(h)=0\qquad\quad\forall\: p\in \NN 
\end{align*}
holds.\footnote{We then automatically have $\lim_{[0,1)\ni h\rightarrow 1} \phi^{(p)}(h)=0$.} 
To show this, we fix $p\in \NN$ and $\vv\in \SEM$, choose $\ww$ as in \eqref{dspoopsdopdsaasas}; and observe that
\begin{align*}
	\ww(Y_n)=2^{n+1}\cdot \ww(X_n)\leq \mackeyconst_\ww\cdot 2^{-n^2+n+1}\leq 1\qquad\quad\forall\:n\geq  \wmackeyindex'_\ww 
\end{align*}
holds, for some $\wmackeyindex'_\ww\geq \max(2,\wmackeyindex_\ww)$ suitably large.
We conclude from \eqref{dspoopsdopdsaasas} that 
\begin{align*}
	\vvv\big((\rho_n\cdot \rcK{\phi_n\cp\varrho_n})^{(p)}\big)&\leq (p+1)^3\cdot 2^{(n+1)\cdot (p+1)^2} \cdot C[\rhon,p]\cdot  \mackeyconst_\ww \cdot 2^{-n^2+n+1}\\
	&= (p+1)^3\cdot C[\rhon,p]\cdot  \mackeyconst_\ww \cdot 2^{-n^2+(n+1)\cdot {((p+1)^2+1)}}
\end{align*}
holds for each $n\geq \wmackeyindex'_\ww$. 
Then, for $h\in [t_{n},t_{n+1}]$ with $n \geq \wmackeyindex'_\ww$, we obtain from \eqref{asklklsaklsalk} that
\begin{align*} 
	1/(1-h)\cdot \vvv\big(\phi^{(p)}(h)\big)&\leq 2^{n+2}\cdot \vvv\big((\rho_n\cdot\rcK{\phi_n\cp\varrho_n})^{(p)}(h)\big)\\
	&\leq (p+1)^3\cdot C[\rhon,p]\cdot  \mackeyconst_\ww \cdot 2^{1-n^2+(n+1)\cdot {((p+1)^2+2)}}\\
&= (p+1)^3\cdot C[\rhon,p]\cdot  \mackeyconst_\ww \cdot 2^{1-n\cdot [n-\frac{n+1}{n}\cdot ((p+1)^2+2)]}
\end{align*}
holds; 
which clearly tends to zero for $n\rightarrow \infty$.
\end{proof}

\subsection{Approximation}
We now finally provide some approximation statements for curves that will be important for our discussions of particular situations in the context of Theorem \ref{confev} in Sect.\ \ref{sec:confined}.
\vspace{6pt}

\noindent
In analogy to Sect.\ \ref{jjdsjksdljdsjldsljdsljdsds}, we say that  $\{\phi_n\}_{n\in \NN}\subseteq \CP^0([r,r'],\mg)$ is a
\begingroup
\setlength{\leftmargini}{12pt}
\begin{itemize}
\item
Cauchy sequence \defff for each $\pp\in \SEM$ and $\varepsilon>0$, there exists some $p\in \NN$, such that 
\begin{align*}
	\ppp_\infty(\phi_m-\phi_n)\leq \varepsilon \qquad\quad 
	\forall\: m,n\geq p.
\end{align*}
\vspace{-18pt}
\item
Mackey-Cauchy sequence \defff 
\begin{align}
\label{mfgjjghfghhghg}
	\ppp_\infty(\phi_m-\phi_n)\leq \mackeyconst_\pp\cdot \lambda_{m,n} \qquad\quad\hspace{2.4pt}\forall\: m,n\geq \mackeyindex_\pp,\:\:\pp\in \SEM
\end{align}
holds for certain $\{\mackeyconst_\pp\}_{\pp\in \SEM}\subseteq \RR_{\geq 0}$, $\{\mackeyindex_\pp\}_{\pp\in \SEM}\subseteq \NN$, and $\RR_{\geq 0}\supseteq \{\lambda_{m,n}\}_{(m,n)\in \NN\times \NN}\rightarrow 0$.
\end{itemize}
\endgroup
\noindent
We say that $\{\phi_n\}_{n\in \NN}\rightarrow \phi$ (converges) uniformly for $\phi\in C^0([r,r'],\mg)$ \defff 
\begin{align*}
	\textstyle\lim_{n\rightarrow \infty}\ppp_\infty(\phi-\phi_n)=0 \qquad\quad\text{holds for each}\qquad\quad \pp\in \SEM;
\end{align*}
and obtain
\begin{lemma}
\label{jshjsdahjsd}
Let $[r,r']\in \COMP$ be fixed. 
\begingroup
\setlength{\leftmargini}{17pt}
\begin{enumerate}
\item
\label{jshjsdahjsd1}
For each $\phi\in C^0([r,r'],\mg)$, there exists a \emph{Cauchy sequence} $\{\phi_n\}_{n\in \NN}\subseteq \DP^\infty([r,r'],\mg)$ with $\{\phi_n\}_{n\in \NN}\rightarrow \phi$ uniformly.
\item
\label{jshjsdahjsd2}
For each $\phi\in C^\lip([r,r'],\mg)$, there exists a \emph{Mackey-Cauchy sequence} $\{\phi_n\}_{n\in \NN}\subseteq \DP^\infty([r,r'],\mg)$ with $\{\phi_n\}_{n\in \NN}\rightarrow \phi$ uniformly.
\end{enumerate}
\endgroup
\end{lemma}
\begin{proof}
We let $\phi\in C^0([r,r'],\mg)$ be fixed; and, for the case that $\phi\in C^\lip([r,r'],\mg)$ holds, we denote the Lipschitz constants of $\phi$ by $\{L_\pp\}_{\pp\in \SEM}\subseteq \RR_{\geq 0}$. 
\begingroup
\setlength{\leftmargini}{12pt}
\begin{itemize}
\item
We choose 
$\Delta>0$ such small that $[0,\Delta]\cdot \dd_e\chart(\im[\phi])\subseteq \V$ holds; and fix $m\geq 1$ with $|r'-r|/m\leq \Delta$.
\item
We define $\gamma[t']\colon [0,\Delta]\ni t\mapsto t\cdot \dd_e\chart(\phi(t'))$ for each $t'\in [r,r']$; and let 
\begin{align*}
	\Phi(t,t'):= \w(t\cdot \dd_e\chart(\phi(t')),\dd_e\chart(\phi(t')))\equiv \Der(\chartinv\cp \gamma[t'])(t)\qquad\quad\forall\: (t,t')\in [0,\Delta]\times [r,r'].
\end{align*} 
\end{itemize}
\endgroup
\noindent
For each $n\geq m$, we construct $\phi_n\in \DP^\infty([r,r'],\mg)$ as follows: 
\begingroup
\setlength{\leftmargini}{12pt}
\begin{itemize}
\item
	We define $\Delta_n:=|r'-r|/n$; and let $t_{n,p}:=r+ p\cdot \Delta_n$ for $p=0,\dots,n$.
\item
	We define $\phi|_{[t_{n,n-1},t_{n,n}]}:=\Phi(\cdot-t_{n,n-1},t_{n,n-1})$, as well as
\begin{align*}
	\phi_n|_{[t_{n,p},t_{n,p+1})}:=\Phi(\cdot-t_{n,p},t_{n,p})\qquad\quad\forall\: p=0,\dots,n-2.\hspace{30.7pt} 
\end{align*}
\end{itemize}
\endgroup
\noindent
By construction, we have
\begin{align}
\label{opggfogpo}
	\phi_n(t_{n,p})=\Phi(0,t_{n,p})=\phi(t_{n,p})\qquad\quad\forall\:  n\geq m,\:\:p=0,\dots,n.
\end{align}
Let now $\vv\in \SEM$ be fixed. We choose $\V\lleq\ww$ as in \eqref{omegakll} for $p\equiv 1$ there; and let $\mackeyindex_\vv\geq m$ be such large that $\Delta_{\mackeyindex_\vv}\cdot \www_\infty(\phi)\leq 1$ holds, i.e., we have $\ww_\infty(\gamma[t_{n,p}]|_{[0,\Delta_n]})\leq 1$ for each $n\geq \mackeyindex_\vv$ and $p=0,\dots,n-1$. 
\vspace{6pt}

\noindent
Then, for each $n\geq \mackeyindex_\vv$, 
\begingroup
\setlength{\leftmargini}{12pt}
\begin{itemize}
\item
we obtain from \eqref{omegakll} and Lemma \ref{ofdpofdpopssssaaaasfffff} that
\begin{align}
\label{pofddfpopofdfd}
\begin{split}
	\vvv\big(\phi_n(t)- \phi_n(t_{n,p})\big)&=\vvv(\Phi(t-t_{n,p},t_{n,p})-\Phi(0,t_{n,p}))\\[1pt]
	&\textstyle\leq \int_0^{t-t_{n,p}}(\vvv\cp\dermapdiff[1])(\gamma[t_{n,p}](s),\dd_e\chart(\phi(t_{n,p})),\dd_e\chart(\phi(t_{n,p})))\:\dd s\\[1pt]
	&\textstyle\leq  \ww(\dd_e\chart(\phi(t_{n,p})))^2\cdot |t-t_{n,p}|\\
	&\leq  \www_\infty(\phi)^2\cdot |t-t_{n,p}|
	\end{split}
\end{align}
holds, for each 
  	\begin{align}
  	\label{sdpodspopodspodspodsaassaposa}
t\in  
\begin{cases}
[t_{n,p},t_{n,p+1})\quad\:\: \text{for}\quad\:0\leq p\leq n-2,\\
[t_{n,n-1},t_{n,n}]\quad\:\: \text{for}\quad\:\:p=n.
\end{cases}
  	\end{align} 
\item
we obtain from \eqref{opggfogpo} and \eqref{pofddfpopofdfd} that
  	  	  	\begin{align}
  	  	  	\label{AAAAAA}
  	  	  	\begin{split}
  		\vvv(\phi(t)-\phi_n(t))&\leq \vvv(\phi(t)-\phi(t_{n,p}))  + \vvv(\phi(t_{n,p})-\phi_{n}(t_{n,p}))+\vvv(\phi_{n}(t)-\phi_{n}(t_{n,p}))\\
  		&\leq \vvv(\phi(t)-\phi(t_{n,p}))  +\www_\infty(\phi)^2\cdot  |t-t_{n,p}|
  		\end{split}
  	\end{align}	
  	holds, for $t$ as in \eqref{sdpodspopodspodspodsaassaposa}.  
\end{itemize}
\endgroup
\noindent
Clearly, \eqref{AAAAAA} implies that $\{\phi_{n-m}\}_{n\in \NN}$ is a Cauchy sequence with $\{\phi_{n-m}\}_{n\in \NN}\rightarrow  \phi$ uniformly. 
Moreover, for the case that $\phi\in C^\lip([r,r'],\mg)$ holds, we define $\mackeyconst_\vv:=L_\vv + \www_\infty(\phi)^2$, and obtain
\vspace{-6pt} 
  	\begin{align*}
  		\vvv(\phi(t)-\phi_n(t))
  		\stackrel{\eqref{AAAAAA}}{\leq} L_\vv\cdot |t-t_{n,p}|  +  \www_\infty(\phi)^2\cdot |t-t_{n,p}|
  		= \mackeyconst_\vv \cdot|t-t_{n,p}|\qquad\quad\forall\: n\geq \mackeyindex_\vv
  	\end{align*}	 
  	for $t$ as in \eqref{sdpodspopodspodspodsaassaposa};  
  	so that \eqref{mfgjjghfghhghg} is clear from the triangle inequality.
\end{proof}
Obviously, we also have
\begin{lemma}
\label{jshjsdahjsda}
Let $[r,r']\in \COMP$ be fixed. Then,
\begingroup
\setlength{\leftmargini}{17pt}
\begin{enumerate}
\item
For each $\phi\in C^0([r,r'],\mg)$, there exists a \emph{Cauchy sequence} $\{\phi_n\}_{n\in \NN}\subseteq \COP([r,r'],\mg)$ with $\{\phi_n\}_{n\in \NN}\rightarrow \phi$ uniformly.
\item
For each $\phi\in C^\lip([r,r'],\mg)$, there exists a \emph{Mackey-Cauchy sequence} $\{\phi_n\}_{n\in \NN}\subseteq \COP([r,r'],\mg)$ with $\{\phi_n\}_{n\in \NN}\rightarrow \phi$ uniformly.
\end{enumerate}
\endgroup
\end{lemma}
\noindent
This Lemma will be relevant for our discussion of the situation where $G$ admits an exponential map, as then clearly $\COP([r,r'],\mg)\subseteq \DP^\infty([r,r'],\mg)$ holds for each $[r,r']\in \COMP$.

\section{The Confined Condition}
\label{sec:confined}
In this section, we clarify under which circumstance a locally $\mu$-convex Lie group is $C^k$-semiregular for $k\in \NN\sqcup\{\lip,\infty\}$ (partially for $k\equiv 0$). We first provide the basic definitions; and then prove the main result in Sect.\ \ref{Evmain}. In the last part of this section, we will discuss several particular situations. 
\vspace{6pt}

\noindent
A sequence $\{\phi_n\}_{n\in \NN}\subseteq\DP^0([r,r'],\mg)$ is said to be {\bf tame} \defff for each $\vv\in \SEM$, there exists some $\vv\leq \ww\in \SEM$ with
	\begin{align}
	\label{confii}
		\vvv\cp \Ad_{[\innt_r^\bullet \phi_n]^{-1}}\leq \www\qquad\quad\forall\: n \in \NN.
	\end{align} 
We say that $\phi\in \bigsqcup_{[r,r']\in \COMP}C^0([r,r'],\mg)$ is
	 \begingroup
\setlength{\leftmargini}{12pt}
\begin{itemize}
\item
		{\bf $\sequy$-integrable} \defff there exists a tame Cauchy sequence 
	$\{\phi_n\}_{n\in \NN}\subseteq\DP^0(\dom[\phi],\mg)$ with $\{\phi_n\}_{n\in \NN}\rightarrow\phi$ uniformly.
	
	The set of all such $\phi$ will be denoted by $\sequ$ in the following. 
\item
	{\bf $\mackey$-integrable} \defff there exists a tame Mackey-Cauchy sequence 
	$\{\phi_n\}_{n\in \NN}\subseteq\DP^0(\dom[\phi],\mg)$ with $\{\phi_n\}_{n\in \NN}\rightarrow\phi$ uniformly.

	The set of all such $\phi$ will be denoted by $\mack$ in the following. 
\end{itemize}
\endgroup
\noindent
Evidently,
\begin{lemma}
\label{sdsddsd}
We have $\DIDE\subseteq \mack\subseteq \sequ$.
\end{lemma}
\begin{proof}
The second inclusion is evident. For the first inclusion, we fix $\phi\in \DIDE_{[r,r']}$ for $[r,r']\in \COMP$; and define 
 $\{\phi_n\}_{n\in \NN}\subseteq \DP^0([r,r'],\mg)$ by $\phi_n:=\phi$ for each $n\in \NN$. Since $\compact:=\inv(\im[\innt_r^\bullet\phi])$ is compact, the first inclusion is clear from \eqref{askasjkjksaasqwqqwasw}. 
\end{proof}
\noindent
Conversely, we have, cf.\ Sect.\ \ref{Evmain} 
\begin{proposition}
\label{dfopfdpoofdp}
Suppose that $G$ is locally $\mu$-convex. Then, 
\begingroup
\setlength{\leftmargini}{17pt}
\begin{enumerate}
\item
\label{dfopfdpoofdp1}
$\sequ\subseteq \DIDE$ holds if $G$ is sequentially complete.
\item
\label{dfopfdpoofdp2}
$\mack\subseteq \DIDE$ holds if $G$ is Mackey complete.
\end{enumerate}
\endgroup
\end{proposition}
\noindent
We say that $G$ is {\bf k-confined} 
\begingroup
\setlength{\leftmargini}{12pt}
\begin{itemize}
\item
	For $k\equiv 0$:\hspace{63.6pt}\qquad \defff $C^0([0,1],\mg)\subseteq \sequ$ holds.
\item
	For $k\in \NN_{\geq 1}\sqcup\{\lip,\infty\}$:\qquad \defff $C^k([0,1],\mg)\subseteq\mack$ holds.
\end{itemize}
\endgroup
\noindent
We conclude from Lemma \ref{sdsddsd} and Proposition \ref{dfopfdpoofdp} that:
\begin{theorem}
\label{confev}
Suppose that $G$ is locally $\mu$-convex. Then,
\begingroup
\setlength{\leftmargini}{17pt}
\begin{enumerate}
\item
\label{confev1}
$G$ is $C^0$-semiregular if $G$ is sequentially complete and {\rm 0}-confined.
\item
\label{confev2}
$G$ is $C^k$-semiregular for $k\in \NN_{\geq 1}\sqcup\{\lip,\infty\}$ \deff $G$ is Mackey complete and {\rm k}-confined.
\end{enumerate}
\endgroup
\end{theorem}
\begin{proof}
If $G$ is $C^k$-semiregular for $k\in \NN_{\geq 1}\sqcup \{\infty\}$, then $G$ is Mackey complete by Theorem \ref{dsjkhjsdjhsd}, as well as {\rm k}-confined by Lemma \ref{sdsddsd}. Moreover, 
\begingroup
\setlength{\leftmargini}{12pt}
\begin{itemize}
\item
	If $G$ is sequentially complete and 0-confined, then $C^0([0,1],\mg)\subseteq  \sequ\subseteq \DIDE$ holds by Proposition \ref{dfopfdpoofdp}.\ref{dfopfdpoofdp1}; so that $G$ is $C^0$-semiregular.
\item
	If $G$ is Mackey complete and k-confined for $k\in \NN_{\geq 1}\sqcup \{\infty\}$, then $C^k([0,1],\mg)\subseteq  \mack\subseteq \DIDE$ holds by Proposition \ref{dfopfdpoofdp}.\ref{dfopfdpoofdp2}; so that $G$ is $C^k$-semiregular.
\end{itemize}
\endgroup
\noindent
This proves the theorem.
\end{proof}

\subsection{Semiregularity}
\label{Evmain}
We now provide the 
\begin{proof}[Proof of Proposition \ref{dfopfdpoofdp}]
We fix $\phi\in \sequ\slash\mack$, and choose a tame Cauchy/Mackey-Cauchy sequence $\{\phi_n\}_{n\in \NN}\subseteq \DP^0(\dom[\phi],\mg)$ that converges uniformly to $\phi$; i.e.,
\begingroup
\setlength{\leftmargini}{12pt}
\begin{itemize}
\item
if $\phi\in \sequ$ holds, then for each $\pp\in \SEM$ and $\varepsilon>0$, there exists some $p\in \NN$, such that 
\begin{align*}
	\ppp_\infty(\phi_m-\phi_n)\leq \varepsilon \qquad\quad 
	\forall\: m,n\geq p.
\end{align*}
\item
\vspace{-10pt}
if $\phi\in \mack$ holds, then we have
\begin{align*}
	\ppp_\infty(\phi_m-\phi_n)\leq \mackeyconst_\pp\cdot \lambda_{m,n} \qquad\quad\forall\: m,n\geq \mackeyindex_\pp,\:\:\pp\in \SEM
\end{align*}
for sequences $\{\mackeyconst_\pp\}_{\pp\in \SEM}\subseteq \RR_{\geq 0}$, $\{\mackeyindex_\pp\}_{\pp\in \SEM}\subseteq \NN$, and $\RR_{\geq 0}\supseteq \{\lambda_{m,n}\}_{(m,n)\in \NN\times \NN}\rightarrow 0$. 
\end{itemize}
\endgroup
\noindent
We let $[r,r']\equiv \dom[\phi]$, define $\mu_n:=\innt_r^\bullet \phi_n$ for each $n\in \NN$; and fix an open neighbourhood $\OO\subseteq G$ of $e$ with $\ovl{\OO}\subseteq \U$.\footnote{Recall that topological groups are $T_3$ spaces.} 
Since 
\begin{align*}
	\textstyle\bound:= \im[\phi]\cup\bigcup_{n\in \NN}\im[\phi_n]
\end{align*}
is bounded, decomposing $[r,r']$ if necessary, we can assume that $\im[\mu_n]\subseteq \ovl{\OO}\subseteq \U$ holds for each $n\in \NN$: This is just clear from Lemma \ref{sddssdsdsd} and Lemma \ref{pofdspospods}. 
We now will show in three steps that $\mu=\lim_n\mu_n$ exists, is of class $C^1$, and fulfills $\Der(\mu)=\phi$ with $\mu(0)=e$. 
\vspace{6pt}

\noindent
{\bf Existence of the Limit:}
\vspace{3pt}

\noindent
For $\pp\in \SEM$, we choose $\qq\in \SEM$ as in Proposition \ref{aaapofdpofdpofdpofd}; and let $\qq\leq \ww\in \SEM$ be as in \eqref{confii} for $\vv\equiv \qq$ there. We choose $p\in \NN$ such large that 
	$|r'-r|\cdot \www_\infty(\phi_m-\phi_n)\leq 1$ holds for each $m,n\geq p$,    
and obtain from \eqref{confii} that 
\begin{align*}
	\textstyle\int \qqq\big( \Ad_{[\innt_r^s \phi_m]^{-1}}(\phi_n(s)-\phi_m(s)) \big) \: \dd s\leq |r'-r|\cdot \www_\infty(\phi_m-\phi_n)\leq 1\qquad\quad \forall\: m,n\geq p.
\end{align*}
Then, \ref{kdskdsdkdslkds} in combination with Proposition \ref{aaapofdpofdpofdpofd} gives
\begin{align}
\label{flkfdlklkfdfd}
\begin{split}
	\textstyle(\pp\cp\chart)\big(\mu_m^{-1}(t)\cdot \mu_n(t)\big)&\textstyle=(\pp\cp\chart)\big(\innt_r^t \Ad_{[\innt_r^\bullet \phi_m]^{-1}}(\phi_n-\phi_m)\big)\\
	&\leq \textstyle\int_r^t \qqq\big( \Ad_{[\innt_r^s \phi_m]^{-1}}(\phi_n(s)-\phi_m(s)) \big)\:\dd s \\
	&\leq\textstyle |r'-r|\cdot \www_\infty(\phi_n-\phi_m)
	\end{split}
\end{align}
for each $m,n\geq p$, and each $t\in [r,r']$. Now, 
\begingroup
\setlength{\leftmargini}{12pt}
\begin{itemize}
\item
This implies that $\{\mu_n(t)\}_{n\in \NN}$ is a Cauchy sequence for each $t\in [r,r']$; thus, converges to some $\mu(t)\in \ovl{\OO}\cap G\subseteq \U$ with $\mu(r)=e$, provided that $G$ is sequentially complete.
\item
	If $\{\phi_n\}_{n\in \NN}$ is a Mackey-Cauchy sequence (i.e., we have $\phi\in \mack$), we replace $\mackeyindex_\pp$ by $\max(\mackeyindex_\pp,p)$ as well as $\mackeyconst_\pp$ by $|r'-r|\cdot \max(\mackeyconst_\pp,\mackeyconst_\ww)$ for each $\pp\in \SEM$. Then, $\{\mu_n(t)\}_{n\in \NN}$ is a Mackey-Cauchy sequence for each $t\in [r,r']$; thus, converges to some $\mu(t)\in \ovl{\OO}\cap G\subseteq \U$ with $\mu(r)=e$, provided that $G$ is Mackey complete. 	
\end{itemize}
\endgroup
\noindent
The rest of the proof is the same for both situations, as we will only use the fact that $\{\phi_n\}_{n\in \NN}$ is a Cauchy sequence in the following. 
We now first have to show that $\mu\colon [r,r']\ni t\mapsto \mu(t)\in G$ is continuous.
\vspace{6pt}

\noindent
{\bf Continuity of the Limit:}
\vspace{3pt}

\noindent
We fix $\pp\in \SEM$, $t\in [r,r']$, $1\geq \varepsilon>0$, and define $J_\delta:=[[r,r']-t\he]\cap (-\delta,\delta)$ for each $\delta>0$. We now have to show that for $\delta>0$ suitably small, we have
\begin{align}
\label{fdfddfdffd}
	\pp\big(\chart(\mu(t))-\chart(\mu(t+\tau))\big)\leq \varepsilon\qquad\quad \forall\: \tau\in J_\delta.
\end{align}
We choose $\pp\leq \uu\in \SEM$ as in Lemma \ref{fhfhfhffhaaaa} for $\compact\equiv\{\mu(t)\}$ there; and obtain
\begin{align*}
	\pp\big(\chart(\mu(t))-\chart(\mu(t+\tau))\big)\leq \uu\big(\chart_{\mu(t)}(\mu(t))-\chart_{\mu(t)}(\mu(t+\tau))\big)=(\uu\cp\chart)\big(\mu^{-1}(t)\cdot\mu(t+\tau)\big)
\end{align*}
provided that $(\uu\cp\chart)\big(\mu^{-1}(t)\cdot\mu(t+\tau)\big)\leq 1$ holds. 
Thus, in order to prove \eqref{fdfddfdffd}, it suffices to show that there exist $p\in \NN$ and $\delta>0$, such that
\begin{align}
\label{saposaopsa}
\begin{split}
	\varepsilon\geq (\uu\cp\chart)\big(\mu^{-1}(t)\cdot\mu(t+\tau)\big)= (\uu\cp\chart)\big(&(\chartinv\cp\chart)\big(\mu^{-1}(t)\cdot \mu_p(t)\big)\he\cdot\\
	& (\chartinv\cp\chart)\big(\mu^{-1}_p(t)\cdot \mu_p(t+\tau)\big)\he\cdot\\
	& (\chartinv \cp\chart)\big(\mu^{-1}_p(t+\tau)\cdot \mu(t+\tau)\big)\big)
\end{split}
\end{align}
holds for each $\tau\in J_\delta$. 
\vspace{6pt}

\noindent
For this, we let $\uu\leq \oo\in \SEM$ be as in \eqref{aaajjhguoiuouo}; and will now show that there exist $p\in \NN$, $\delta>0$, such that
\begin{align*}
	(\oo\cp\chart)\big(\mu^{-1}(t)\cdot \mu_p(t)\big),\:\:(\oo\cp\chart)\big(\mu^{-1}_p(t)\cdot \mu_p(t+\tau)\big),\:\: (\oo\cp\chart)\big(\mu^{-1}_p(t+\tau)\cdot \mu(t+\tau)\big)\:\:\leq\:\: \varepsilon/3 
\end{align*} 
holds for all $\tau\in J_\delta$: Then,  \eqref{saposaopsa} is clear from \eqref{aaajjhguoiuouo}. 

Now, 
\begingroup
\setlength{\leftmargini}{12pt}
\begin{itemize}
\item
In order to estimate the second term, 
\begingroup 
\setlength{\leftmarginii}{12pt}
\begin{itemize}
\item[$\circ$]
We choose $\oo\leq \nn$ as in \eqref{invrel}, for $\mm\equiv\oo$ there. 
\item[$\circ$]
We choose $\nn\leq\qq\in \SEM$ as in Proposition \ref{aaapofdpofdpofdpofd}, for $\pp\equiv\nn$ there.
\item[$\circ$]
\label{ccc}
We choose $\qq\leq \ww\in \SEM$ as in \eqref{confii} for $\vv\equiv \qq$ there; and fix 
\begin{align*}
	1\geq \delta:=\varepsilon/3 \cdot\max(1,\sup\{X\in \bound\:|\: \ww(X)\})^{-1}. 
\end{align*}
\end{itemize}
\endgroup
\noindent
We then have to discuss the cases $\tau\geq 0$ and $\tau<0$ separately. 
\begingroup
\setlength{\leftmarginii}{11pt}
\begin{itemize}
\item[$\triangleright$]
Let $\tau\in J_\delta$ with $\tau\geq 0$. Then, for each $p\in \NN$, we have 
\begin{align*}
	\textstyle (\oo\cp\chart)\big(\mu_p^{-1}(t)\cdot \mu_p(t+\tau)\big)&\leq  (\nn\cp\chart)\big(\mu_p^{-1}(t)\cdot \mu_p(t+\tau)\big)\\
	&\textstyle=(\nn\cp\chart)\big(\mu_p^{-1}(t)\cdot \big[\innt_t^{t+\tau}\phi\big]\cdot \mu_p(t)\big)\\
	&\textstyle=(\nn\cp\chart)\big(\innt_t^{t+\tau} \Ad_{\mu^{-1}_p(t)}(\phi_p)\big)\\
	&\textstyle \leq \int_t^{t+\tau} \qqq\big(\Ad_{\mu^{-1}_p(t)}(\phi_p(s))\big)\:\dd s\\&\textstyle\leq  \int_t^{t+\tau} \www(\phi_p(s))\:\dd s\leq \varepsilon/3.
\end{align*}
In the second step, we have used \ref{pogfpogf}; and in the third step, we have applied \ref{homtausch} to $\Psi\equiv\conj_{\mu^{-1}_p(t)}$. 
\item[$\triangleright$]
Let $\tau\in J_\delta$ with $\tau< 0$. Then, for each $p\in \NN$, we have
\begin{align*}
	\textstyle(\oo\cp\chart)\big(\mu^{-1}_p(t)\cdot \mu_p(t-|\tau|)\big)&\textstyle=(\oo\cp\chart\cp\inv)\big(\mu^{-1}_p(t-|\tau|)\cdot \mu_p(t)\big)\\
	&\leq \textstyle(\nn\cp\chart)\big(\mu^{-1}_p(t-|\tau|)\cdot \mu_p(t)\big)\\ 
	&\textstyle= (\nn\cp\chart)\big(\mu^{-1}_p(t-|\tau|)\cdot \big[\innt_{t-|\tau|}^t\phi_p\big]\cdot \mu_p(t-|\tau|)\big)\\
	&\textstyle= (\nn\cp\chart)\big(\innt_{t-|\tau|}^t \Ad_{\mu^{-1}_p(t-|\tau|)}(\phi_p)\big)\\
	&\textstyle\leq \int_{t-|\tau|}^t \qqq\big(\Ad_{\mu^{-1}_p(t-|\tau|)}(\phi_p(s))\big)\:\dd s\\
	&\textstyle \leq \int_{t-|\tau|}^t \www(\phi_p(s))\:\dd s\textstyle\leq \varepsilon/3.
\end{align*}  
\end{itemize}
\endgroup
\noindent
\item
In order to estimate the first-, and the third term, we let $\oo\leq\ff\in \SEM$ be as in \eqref{aaajjhguoiuouo} for $\uu\equiv\oo$ and $\oo\equiv\ff$ there. We choose $\iota\colon \NN\rightarrow\NN$ strictly increasing with (use \eqref{flkfdlklkfdfd})
\begin{align*}
	\textstyle \sum_{n=0}^\infty (\ff\cp\chart)\big(\mu_{\iota(n+1)}^{-1}\cdot \mu_{\iota(n)}\big)\leq \varepsilon/3\qquad\quad\:\:\text{and}\qquad\quad\:\: \sum_{n=0}^\infty (\ff\cp\chart)\big(\mu_{\iota(n)}^{-1}\cdot \mu_{\iota(n+1)}\big)\leq \varepsilon/3;
\end{align*}
and observe that 
\begin{align*}
\begin{split}
	\textstyle\mu^{-1}&\cdot \mu_{\iota(0)}\textstyle=\lim_n \big((\mu^{-1}_{\iota(n)}\cdot \mu_{\iota(n-1)})\cdot (\mu_{\iota(n-1)}^{-1}\cdot \mu_{\iota(n-2)})\cdot {\dots}\cdot(\mu_{\iota(1)}^{-1}\cdot \mu_{\iota(0)})\big)\\
	\textstyle\mu^{-1}_{\iota(0)}&\cdot \mu\hspace{14.5pt} \textstyle=\lim_n \big((\mu_{\iota(0)}^{-1}\cdot \mu_{\iota(1)})\cdot {\dots}\cdot (\mu_{\iota(n-2)}^{-1}\cdot \mu_{\iota(n-1)})\cdot(\mu_{\iota(n-1)}^{-1}\cdot \mu_{\iota(n)})\big)
\end{split}
\end{align*}
holds. It is thus clear from \eqref{aaajjhguoiuouo} that 
\begin{align*}
	\textstyle(\oo\cp\chart)\big(\mu^{-1}(t)\cdot \mu_{p}(t)\big)\leq \varepsilon/3\qquad\quad\text{and}\qquad\quad (\oo\cp\chart)\big(\mu^{-1}_{p}(t)\cdot \mu(t)\big)\leq \varepsilon/3
\end{align*} 
holds for each $t\in [r,r']$, for $p:=\iota(0)$. From this, the claim is clear. 
\end{itemize}
\endgroup
\vspace{6pt}

\noindent
{\bf Uniform Convergence:}
\vspace{3pt}

\noindent
We define $\gamma:=\chart\cp\mu$, as well as $\gamma_n:=\chart\cp\mu_n$ for each $n\in \NN$; and now show that  
 $\{\gamma_n\}_{n\in \NN}$ converges uniformly to $\gamma$. 
For this, we let $\pp\in \SEM$, and $1\geq \varepsilon>0$ be fixed; and observe that $\compact\equiv\im[\mu]$ is compact, because $\mu$ is continuous. By Lemma \ref{fhfhfhffhaaaa}, there thus exists some  
$\pp\leq \uu\in \SEM$, such that (let $\compact\equiv \im[\mu]$,  $g\equiv g(t):= \mu(t)$, $q\equiv q(t):=\mu(t)$, $q'\equiv q'(t):=\mu_m(t)$, $h\equiv e$ there)
\begin{align*}
	(\uu\cp\chart)\big(\mu^{-1}\cdot \mu_m\big)\leq 1\quad\:\:\text{for}\quad\:\: m\in \NN\qquad\quad\:\Longrightarrow\qquad\quad\: \pp(\gamma-\gamma_m)\leq (\uu\cp\chart)\big(\mu^{-1}\cdot \mu_m\big).
\end{align*}
We choose $\uu\leq \oo$ as in \eqref{aaajjhguoiuouo}; and let $\iota\colon \NN\rightarrow\NN$ be strictly increasing with (use \eqref{flkfdlklkfdfd})
\begin{align*}
	\textstyle  (\oo\cp\chart)\big(\mu_m^{-1}\cdot \mu_{\iota(0)}\big)\leq \varepsilon/2\qquad\forall\: m\geq \iota(0)\qquad\quad\:\:\text{and}\quad\qquad\:\: \sum_{n=0}^\infty (\oo\cp\chart)\big(\mu_{\iota(n+1)}^{-1}\cdot \mu_{\iota(n)}\big)\leq \varepsilon/2.
\end{align*}
Then, \eqref{aaajjhguoiuouo} shows  
\begin{align*}
	(\uu\cp\chart)\big(\mu^{-1}\cdot\mu_{m}\big)&=\textstyle (\uu\cp\chart)\big(\big(\mu^{-1}\cdot \mu_{\iota(0)}\big) \cdot \big(\mu_{\iota(0)}^{-1} \cdot\mu_{m}\big)\big)\\
	&\textstyle=\lim_n (\uu\cp\chart)\big(\big(\mu^{-1}_{\iota(n)}\cdot \mu_{\iota(n-1)}\big)\cdot \big(\mu_{\iota(n-1)}^{-1}\cdot \mu_{\iota(n-2)}\big)
	\cdot {\dots}\cdot\big(\mu_{\iota(1)}^{-1}\cdot \mu_{\iota(0)}\big)\\
	&\hspace{144.8pt}\cdot \big(\mu_{\iota(0)}^{-1} \cdot\mu_{m}\big)\big)\\
	&\textstyle\leq \varepsilon
\end{align*}
for each $m\geq \iota(0)$, which proves the claim. 
 
We are ready to show 
\vspace{6pt}

\noindent
{\bf The solution property:}
\vspace{3pt}

\noindent
Let $\dermapinvdiff$ be as in \eqref{opopopop2}. 
Then, it is straightforward from the definitions that
\begin{align}
\label{asiosaioaois}
	\textstyle\gamma_n=\int_r^\bullet\dermapinvdiff(\gamma_n(s),\phi_n(s))\:\dd s\qquad\quad\forall\: n\in \NN
\end{align}
holds, cf.\ Appendix \ref{assaasssaaggggs}. Moreover, since $\dermapinvdiff$ is continuous, since $\im[\gamma]\times \im[\phi]$ is compact, and since $\{\gamma_n\}_{n\in \NN}$ and $\{\phi_n\}_{n\in \NN}$ converge uniformly to $\gamma$ and $\phi$, respectively,  we additionally  obtain
\begin{align*}
	\textstyle\lim_n\int_r^\bullet\dermapinvdiff(\gamma_n(s),\phi_n(s))\:\dd s=\int_r^\bullet\dermapinvdiff(\gamma(s),\phi(s))\:\dd s\in \comp{E}.
\end{align*}
Together with  \eqref{asiosaioaois}, this shows
\begin{align*}
	\textstyle\gamma=\lim_n\gamma_n=\lim_n\int_r^\bullet\dermapinvdiff(\gamma_n(s),\phi_n(s))\: \dd s=\int_r^\bullet\dermapinvdiff(\gamma(s),\phi(s))\:\dd s;
\end{align*}
i.e., that $\gamma$ is of class $C^1$ with $\dot\gamma=\dermapinvdiff(\gamma,\phi)\in E$. We obtain
\begin{align*}
\Der(\mu)&=\hspace{5pt}\dd_\mu\RT_{\mu^{-1}}\big(\dd_\gamma\chartinv(\dot\gamma)\big)\\
&=\hspace{5pt}\dd_\mu\RT_{\mu^{-1}}\big(\dd_\gamma\chartinv(\dermapinvdiff(\gamma,\phi))\big)\\
&=\big(\dd_\mu\RT_{\mu^{-1}}\cp \dd_\gamma\chartinv \cp \dd_{\chartinv(\gamma)}\chart\cp\dd_e\RT_{\chartinv(\gamma)}\big)(\phi)\\
&=\big(\dd_\mu\RT_{\mu^{-1}}\cp \dd_e\RT_\mu\big)(\phi)=\phi,
\end{align*}
which proves the claim. 
\end{proof}

\subsection{Particular Cases}
\label{patcases}
In this subsection, we discuss several situations in which $G$ is automatically k-confined for each $k\in \NN\sqcup\{\lip,\infty\}$. We start with
\subsubsection{Reliable Lie Groups}
We say that $G$ is {\bf reliable} \defff for each $\vv\in \SEM$, there exists a symmetric neighbourhood $V\subseteq G$ of $e$ as well as a sequence $\{\ww_n\}_{n\in \NN_{\geq 1}}\subseteq \SEM$ with
\begin{align}
\label{sdpodspods}
	\vvv\cp \Ad_{g_1}\cp{\dots}\cp\Ad_{g_n}\leq \www_{n}\qquad\quad\forall\: g_1,\dots,g_n\in V,\:\:n\geq 1.
\end{align}
For instance, $G$ is reliable: 
\begingroup
\setlength{\leftmargini}{20pt}
{
\renewcommand{\theenumi}{{\Alph{enumi}})} 
\renewcommand{\labelenumi}{\theenumi}
\begin{enumerate}
\item
\label{ExA}
If $G$ is abelian.
\item
\label{ExB}
	If for each $\vv\in \SEM$, there exist $\ww\in \SEM$, $C\geq 0$, and $V$ open with $e\in V$, such that 
	\begin{align*}
		\vvv\cp \Ad_{g_1}\cp{\dots}\cp\Ad_{g_n}\leq C^n\cdot \www\qquad\quad\forall\: g_1,\dots,g_n\in V,\:\:n\geq 1.
	\end{align*}	
	In particular, this is the case for the unit group $\MAU$ of a  continuous inverse algebra $\MA$ in the sense of \cite{HGIA}, just by \eqref{invACond}.
\item
\label{ExC}
If for each $\vv\in \SEM$, there exist $\vv\leq \ww\in \SEM$, $C\geq 0$, and $V$ open with $e\in V$, such that
\begin{align*}
	\www\cp\Ad_g\leq C\cdot \www\qquad\quad\forall\: g\in V.
\end{align*}
	In particular, this is the case 
		if $(\mg,\bl\cdot,\cdot\br)$ is submultiplicative, cf.\ Proposition \ref{sopsopdsop}; so that Banach Lie groups are reliable (of course, this can also be directly seen from \eqref{odspospodpof}).
\end{enumerate}}
\endgroup
\noindent
Then, 
\begin{lemma}
\label{opdopdspo}
Suppose that $G$ is locally $\mu$-convex and reliable, let $\bound\subseteq \mg$ be bounded, and $[r,r']\in \COMP$ be fixed. Then, for each $\vv\in \SEM$, there exists some $\vv\leq \ww\in \SEM$, such that
\begin{align*}
	\vvv\cp\Ad_{[\innt_r^\bullet\phi]^{-1}}\leq \www
\end{align*}
holds for each $\phi\in \DP^0([r,r'],\mg)$ with $\im[\phi]\subseteq \bound$.
\end{lemma}
\begin{proof}
	We choose $\{\ww_n\}_{n\in \NN_{\geq 1}}\subseteq \SEM$ and $V$ as in \eqref{sdpodspods}; and can assume that $\vv\leq \ww_1\leq \ww_2\leq {\dots}$ holds, just by replacing $\ww_n\rightarrow \ww_1+{\dots}+\ww_n$ for each $n\geq 1$ if necessary. Then, 
	\begingroup
\setlength{\leftmargini}{12pt}
\begin{itemize}
\item
By Proposition \ref{aaapofdpofdpofdpofd}, there exists some $\qq\in \SEM$, such that $\innt_\ell^\bullet \psi\in V$ holds for each $\psi\in \DP^0([\ell,\ell'],\mg)$,  $[\ell,\ell']\in \COMP$, with $\int \qq(\psi(s))\: \dd s\leq 1$. 
\item
 We define $\lambda:=\sup\{\qq(X)\:|\:X\in \bound\}$; and choose $n\geq 1$ such large that $\lambda\cdot |r'-r|/n \leq 1$ holds.
\item 
  We define $t_p:= r+ p\cdot |r'-r|/n$ for $p=0,\dots,n$; and obtain $\textstyle\big[\innt_{t_p}^{t}\phi\big]^{-1} \in V$ for each $t\in [t_p,t_{p+1}]$,
 for $p=0,\dots,n-1$.	
\end{itemize}
\endgroup
\noindent
We define $\ww:=\ww_n$, and obtain 
	 \begin{align*}
	 	\textstyle\vvv\cp\Ad_{[\innt_r^{t}\phi]^{-1}}\stackrel{\ref{pogfpogf}}{=}\vvv\cp \Ad_{[\innt_{t_0}^{t_1} \phi]^{-1}\cdot {\dots}\cdot [\innt_{t_p}^t\phi]^{-1}}= \vvv\cp \Ad_{[\innt_{t_0}^{t_1} \phi]^{-1}}\cp {\dots} \cp\Ad_{[\innt_{t_p}^t\phi]^{-1}} \leq \www_{p+1}\leq\www
	 \end{align*}
	for each $t\in [t_p,t_{p+1}]$, for $p=0,\dots,n-1$.
\end{proof}
We obtain
\begin{lemma}
\label{opsdopds}
Suppose that $G$ is locally $\mu$-convex and reliable. Then, $G$ is {\rm k}-confined for each $k\in \NN\sqcup\{\lip,\infty\}$. 
\end{lemma}
\begin{proof}
This is just clear from Lemma \ref{jshjsdahjsd} and Lemma \ref{opdopdspo}.
\end{proof}
We thus have
\begin{proposition}
\label{sddsdsds}
Suppose that $G$ is locally $\mu$-convex and reliable. Then, 
\begingroup
\setlength{\leftmargini}{17pt}
\begin{enumerate}
\item
\label{confevv1}
$G$ is $C^0$-semiregular if $G$ is sequentially complete.
\item
\label{confevv2}
$G$ is $C^\lip$-semiregular \deff $G$ is Mackey complete \deff $G$ is $C^\infty$-semiregular.
\end{enumerate}
\endgroup
\end{proposition}
\begin{proof}
By Lemma \ref{opsdopds}, $G$ is k-confined for each $k\in \NN\sqcup\{\lip,\infty\}$. Thus, 
\begingroup
\setlength{\leftmargini}{12pt}
\begin{itemize}
\item
	If $G$ is sequentially complete, then $G$ is $C^0$-semiregular by Theorem \ref{confev}.\ref{confev1}.
\item
	If $G$ is Mackey complete, then $G$ is $C^\lip$-semiregular by Theorem \ref{confev}.\ref{confev2}.
\end{itemize}
\endgroup
\noindent
The converse direction in \ref{confevv2} is clear from Theorem \ref{dsjkhjsdjhsd}. 
\end{proof}
For instance,
\begin{corollary}
\label{sddsdsdsx}
Suppose that $G$ is abelian and locally $\mu$-convex. Then, 
\begingroup
\setlength{\leftmargini}{17pt}
\begin{enumerate}
\item
\label{confevvx1}
$G$ is $C^0$-semiregular if $G$ is sequentially complete.
\item
\label{confevvx2}
$G$ is $C^\lip$-semiregular \deff $G$ is Mackey complete \deff $G$ is $C^\infty$-semiregular.
\end{enumerate}
\endgroup
\end{corollary}
In particular, we recover the well known fact that\footnote{Clearly, Corollary \ref{sddsdsdsx}.\ref{confevv1} proves the obvious fact that each $\phi\in \bigsqcup_{[r,r']\in \COMP}C^0([r,r'],E)$ is Riemann integrable if $E$ is sequentially complete. } 
\begin{corollary}
\label{Mackeycor}
	$E$ is Mackey complete \deff the Riemann integral $\int \phi(s) \:\dd s\in E$ exists for each $\phi\in C^\infty([0,1],E)$ \deff the Riemann integral $\int \phi(s) \:\dd s\in E$ exists for each $\phi\in \bigsqcup_{[r,r']\in \COMP}C^\lip([r,r'],E)$. 
\end{corollary}

\subsubsection{Constricted Lie Groups}
We say that $G$ is {\bf constricted} \defff for each bounded subset $\bound\subseteq \mg$, and each $\vv\in \SEM$, there exist $C\geq 0$ and $\vv\leq \ww\in \SEM$, such that 
\begin{align}
\label{assssasaasas}
	\vvv\cp \com{X_1}\cp {\dots}\cp \com{X_n}\leq  C^n\cdot \www\qquad\quad\forall\: X_1,\dots,X_n\in \bound,\:\: n\geq 1
\end{align}
holds, with $\com{X} \colon \mg\ni Y\mapsto [X,Y]\in \mg$ for each $X\in \mg$. 
We define $\com{X}^0:=\id_\mg$ as well as inductively 
\begin{align*}
	(\com{X})^n:=\com{X}\cp(\com{X})^{n-1}\qquad\quad\forall\: n\geq 1.
\end{align*}  
Clearly, $G$ is constricted: 
\begingroup
\setlength{\leftmargini}{12pt}
\begin{itemize}
\item
If $(\mg,\bl\cdot,\cdot\br)$ is asymptotic estimate in the sense of \cite{BOSECK}.
\item
If $\bl\cdot,\cdot\br$ is submultiplicative; i.e., \defff for each $\vv\in \SEM$, there exists some $\vv\leq \ww\in \SEM$, such that 
\begin{align}
\label{nbxcnxnbcbxc}
 	\www(\bl X,Y\br)\leq \www(X)\cdot \www(Y)\qquad\quad\forall\: X,Y\in \mg.
\end{align}
\vspace{-18pt}
\item
If $\bl\cdot,\cdot \br$ is nilpotent in the sense that there exists some $n\geq 2$, such that
\begin{align*}
	\com{X_1}\cp {\dots}\cp \com{X_n}=0\qquad\quad\forall\: X_1,\dots,X_n\in \mg.
\end{align*}
\end{itemize}
\endgroup
\noindent
We will now show step by step that
\begin{proposition}
\label{opopsdods}
Suppose that $G$ is constricted, and admits an exponential map; and that $\mg$ is sequentially complete. Then, $G$ is {\rm k}-confined for each $k\in \NN\sqcup\{\lip,\infty\}$.
\end{proposition}
Let us first recall that
\begin{lemma}
\label{Potreih}
	Suppose that 
	 $\sum_{n=0}^\infty r^n\cdot a_n\in \mgc$ converges for some $r\in \RR_{\neq 0}$, and $\{a_n\}_{n\in \NN}\subseteq \mgc$.  
Then, 
	$\textstyle\alpha\colon I\rightarrow \mgc,\quad
		t\mapsto \sum_{n=0}^\infty t^n\cdot a_n$ 
	is of class $C^1$ (smooth) for each open interval $I\subseteq [-r,r]$, with
	\begin{align*}
		\textstyle\dot\alpha=\sum_{n=1}^\infty n\cdot t^{n-1}\cdot a_n\qquad\quad\text{as well as}\qquad\quad \int_0^t\alpha(s)\:\dd s=\sum_{n=0}^\infty \frac{t^{n+1}}{n+1}\cdot a_n.
	\end{align*}	
\end{lemma}
\begin{proof}
	This just follows as in the case where $\mg=\mgc=\CC$ holds.
\end{proof}
We obtain
\begin{lemma}
\label{opqwopqwowpwopqwq}
Suppose that $G$ is constricted, and that $\mg$ is sequentially complete. Then,
\begin{align*}
	\textstyle\alpha_{X,Y}\colon \RR\ni t\mapsto \sum_{n=0}^\infty \frac{t^n}{n!}\cdot (\com{X})^n(Y)\in \mg\qquad\quad\forall\: X,Y\in \mg
\end{align*}
is of class $C^1$ with $\dot\alpha_{X,Y}=\bl X,\alpha_{X,Y}\br$; thus, smooth by Corollary \ref{bhsbsshshdkksjdhjsd}.
\end{lemma}
\begin{proof}
	It is straightforward from the definitions that $\{\sum_{k=0}^n \frac{t^k}{k!}\cdot (\com{X})^k(Y)\}_{n\in \NN}\subseteq \mg$ is a Cauchy sequence for each $t\in \RR$, and $X,Y\in \mg$; thus, converges to some $\alpha_{\bl X,Y\br}(t)\in\mg$. By Lemma \ref{Potreih}, $\alpha_{X,Y}\colon \RR\rightarrow  \mg\subseteq \mgc$ is of class $C^1$ with $\dot\alpha_{X,Y}=\bl X,\alpha_{X,Y}\br$; which implies $\im[\dot\alpha_{X,Y}]\subseteq\mg$.	
\end{proof}
Let now $[r,r']\in \COMP$ be fixed; and recall that \cite{Omori}
\begin{lemma}[Omori]
\label{omori}
Let $\phi\in \DIDE_{[r,r']}$, $Y\in \mg$, and $\alpha\in C^1([r,r'],\mg)$ be fixed. Then, we have
\begin{align*}
		\textstyle \alpha=\Ad_{\mu}(Y)\quad\:\:\text{for}\quad\:\: \mu:=\innt_r^\bullet \phi \qquad\quad\:\: \Longleftrightarrow\qquad\quad\:\:
		\dot\alpha=\bl \phi,\alpha\br\quad\:\:\text{holds with}\quad\:\: \alpha(r)=Y.
	\end{align*} 
\end{lemma}
\begin{proof} 
	The proof is elementary, and can be found in Appendix \ref{assaauuuuuuu}. 
\end{proof}
We conclude that
\begin{corollary}
\label{saklasklklsa}
	Suppose that $G$ is constricted, and admits an exponential map; and that $\mg$ is sequentially complete. Then, we have $\Ad_{\exp(-t\cdot X)}(Y)=\alpha_{-X,Y}(t)$ for all $t\geq 0$, and $X,Y\in \mg$.
\end{corollary}
\begin{proof}
By \eqref{lkdsklsdkdl}, we have   
	$\alpha(t):=\Ad_{\exp(-t\cdot X)}(Y)= \Ad_{\innt_0^t \phi_{-X}}(Y)$, i.e.,  
 $\dot\alpha=\bl \phi_{-X},\alpha\br\equiv \bl-X,\alpha\br$ by Lemma \ref{omori}. The claim is thus clear from Lemma \ref{opqwopqwowpwopqwq} and Lemma \ref{omori}. 
\end{proof}
For the rest of this section, we let $\exxp$ denote the exponential function on $\RR$.

We obtain
\begin{lemma}
\label{klsdklsdkldsklds}
	Suppose that $G$ is constricted, and admits an exponential map; and that $\mg$ is sequentially complete. Then, for each $\pp\in \SEM$, and each bounded subset $\bound\subseteq \mg$, there exists some $\qq\in \SEM$, such that 
	$\ppp\cp \Ad_{[\innt_r^\bullet\phi]^{-1}}\leq \qqq$ holds for each $\phi\in \COP([r,r'],\mg)$ with $\im[\phi]\subseteq \bound$. 
\end{lemma}
\begin{proof}
	We let $\ww$ be as in \eqref{assssasaasas}, for $\vv\equiv\pp$ there; and choose $r=t_0<{\dots}<t_n=r'$ as well as $X_0,\dots,X_{n-1}\in \mg$, with 
	$\phi|_{(t_p,t_{p+1})}=X_p$ for all $p=0,\dots,n-1$.  
Then, for $0\leq p\leq n-1$ and $t\in (t_p,t_{p+1}]$, we have 
\begin{align*}
	\Ad_{[\innt_r^t\phi]^{-1}}=\Ad_{\exp(-|t_1-t_0|\cdot X_0)}\cp {\dots}\cp \Ad_{\exp(-|t_p-t_{p-1}|\cdot X_{p-1})}\cp \Ad_{\exp(-|t-t_p|\cdot X_{p})};
\end{align*}
so that Corollary \ref{saklasklklsa} together with \eqref{assssasaasas} shows
\begin{align*}
	\big(\ppp\cp \Ad_{[\innt_r^t\phi]^{-1}}\big)(Y)\leq \exxp(|t-r|\cdot C)\cdot \www(Y)\qquad\quad\forall\: Y\in \mg,\:\: t\in [r,r'].
\end{align*}
The claim thus holds for $\qq:= \exxp(|r'-r|\cdot C)\cdot \ww$.
\end{proof}
We are ready for the
\begin{proof}[Proof of Proposition \ref{opopsdods}]
The claim is clear from Lemma \ref{jshjsdahjsda} and Lemma \ref{klsdklsdkldsklds}.
\end{proof}

\subsubsection{Submultiplicative Lie Algebras}
\label{pofdofdpofdpofdpofdfdofdpodf}
We finally want to discuss the situation where $(\mg,\bl\cdot,\cdot\br)$ is submultiplicative. Clearly, $G$ is constricted in this case; but, as we are going to show now, there exists a sharper version of Proposition \ref{opopsdods} neither presuming the existence of the exponential map nor sequentially completeness of $\mg$. More specifically, we will show that 
\begin{proposition}
\label{sopsopdsop}
If $(\mg,\bl\cdot,\cdot\br)$ is submultiplicative, then $G$ is 
{\rm k}-confined for each $k\in \NN\sqcup\{\lip,\infty\}$.  
Moreover, $G$ is reliable as it fulfills the condition introduced in \ref{ExC}.  
\end{proposition}
For this, let $[r,r']\in\COMP$ be fixed; and recall that 
\begin{lemma}[Gr\"onwall]
\label{apoapoapoapo}
	Let $\alpha,\beta\colon [r,r']\rightarrow \RR_{\geq 0}$ be of class $C^1$, and $C\geq 0$. Then,
	\begin{align*}
		\textstyle\alpha\leq C+\int_r^\bullet (\alpha\cdot \beta)(s)\:\dd s\qquad\quad\Longrightarrow\qquad\quad \alpha\leq C\cdot \exxp\big(\int_r^\bullet\beta(s)\:\dd s\big).
	\end{align*}
\end{lemma}
We obtain that
\begin{lemma}
\label{Groenwall}
	Suppose that $(\mg,\bl\cdot,\cdot\br)$ is submultiplicative, and let $\vv\leq \ww\in \SEM$ be as in \eqref{nbxcnxnbcbxc}. 
		Then, for each $\phi\in C^0([r,r'],\mg)$, $Y\in \mg$, and  $\alpha\in C^1([r,r'],\mg)$ with 
		$\dot\alpha=\bl\phi,\alpha\br$ and $\alpha(r)=Y$, 
	we have 
\begin{align*} 
	\textstyle\vvv(\alpha)\leq  \www(Y)\cdot \exxp\big(\int_r^\bullet \www(\phi(s))\:\dd s\big).
\end{align*} 
\end{lemma}
\begin{proof}
	 We conclude from Lemma \ref{ofdpofdpopssssaaaasfffff} that 
	 \begin{align*}
	 	\textstyle\www(\alpha)\leq \www(Y)+\int_r^\bullet\www(\dot\alpha(s))\: \dd s\leq\www(Y)+\int_r^\bullet \www(\alpha(s))\cdot \www(\phi(s))\: \dd s
	 \end{align*}
	 holds; so that the claim is clear from Lemma \ref{apoapoapoapo}.
\end{proof}
\begin{corollary}
\label{sajksakjsa}
Suppose that $(\mg,\bl\cdot,\cdot\br)$ is submultiplicative, and let $\vv\leq \ww\in \SEM$ be as in 
\eqref{nbxcnxnbcbxc}. 
Then, 
\begin{align*}
	\textstyle\vvv\big(\Ad_{\innt_r^\bullet\phi}(Y)\big)\leq \exxp(\int_r^\bullet \www(\phi(s))\:\dd s)\cdot \www(Y)
\end{align*}
holds for each $Y\in \mg$ and $\phi\in \DIDE_{[r,r']}$; thus, 
\begin{align*}
	\textstyle\www\big(\Ad_{[\innt_r^\bullet\phi]^\pm}(Y)\big)\leq \exxp(|r'-r|\cdot \www_\infty(\phi))\cdot \www(Y)\qquad\quad\forall\: Y\in \mg,\:\: \phi\in \DIDE_{[r,r']}.
\end{align*}
\end{corollary}
\begin{proof}
The first statement is clear from  Lemma \ref{omori}, and Lemma \ref{Groenwall}. Then, the second statement is immediate from Example \ref{fdpofdopdpof}.
	\end{proof}
We are ready for the
\begin{proof}[Proof of Proposition \ref{sopsopdsop}]
The first statement is clear from Corollary \ref{sajksakjsa} and Lemma \ref{jshjsdahjsd}. 
For the second statement, we let $\vv\leq\ww\in \SEM$ be as in \eqref{nbxcnxnbcbxc}, choose $\ww\leq \oo\in \SEM$ as in Lemma \ref{poposposaadccx} for $\mm\equiv \ww$ there; and define $V:=\chartinv(\B_{\oo,1})$. We furthermore define
\begin{align*}
	\gamma_x\colon [0,1]\ni t\mapsto t\cdot x\in V\qquad\:\:\text{as well as}\qquad\:\: \phi_x:= \Der(\chartinv\cp\gamma_x)\qquad\:\:\text{for each}\qquad\:\: x\in V.
\end{align*} 
Then, Lemma \ref{poposposaadccx} shows that $\www_\infty(\phi_x)\leq \oo_\infty(\dot\gamma_x)\leq 1$ holds for each $x\in V$; so that Corollary \ref{sajksakjsa} gives
\begin{align*}
	 \www(\Ad_{\chartinv(x)}(Y))=\www(\Ad_{(\chartinv\cp\gamma_x)(1)}(Y))=\www\big(\Ad_{\innt_0^1\phi_x}(Y)\big)\leq \exxp(1)\cdot \www(Y),
\end{align*}
for each $Y\in \mg$; which shows the claim.
\end{proof}

\section{Differentiation Under the Integral}
\label{asopsopdsopsdpoosdp}
In this section, we clarify under which circumstances $\EVE_{[r,r']}^k$ for $k\in \NN\sqcup\{\lip,\infty\}$ and $[r,r']\in \COMP$ is differentiable w.r.t.\ the (standard) the $C^k$-topology. In particular, we will show that\footnote{A $C^1$-version will also be proven for the Lipschitz case, cf.\ Corollary \ref{sddsdsdsds}.}
\begin{theorem}
\label{kdfklfklfdkjljljjllk} 
\noindent
\begingroup
\setlength{\leftmargini}{17pt}
{
\renewcommand{\theenumi}{{\arabic{enumi}})} 
\renewcommand{\labelenumi}{\theenumi}
\begin{enumerate}
\item
\label{aaaa2}
		If $G$ is {\rm 0}-continuous and $C^0$-semiregular, then  $\EVE^0_{[r,r']}$ is smooth for each $[r,r']\in \COMP$ \deff $\mg$ is integral complete \deff $\EVE^0_{[0,1]}$ is differentiable at zero.
\item
\label{aaaa1}
	If $G$ is {\rm k}-continuous and $C^k$-semiregular for $k\in \NN_{\geq 1}\sqcup\{\infty\}$, then $\EVE_{[r,r']}^k$ is smooth for each $[r,r']\in \COMP$ \deff $\mg$ is Mackey complete \deff $\EVE_{[0,1]}^k$ is differentiable at zero.
\end{enumerate}}
\endgroup
\noindent
Here, for $k=0$ in the first-, and $k\in \NN_{\geq 1} \sqcup\{\infty\}$ in the second case, we have 
\begin{align*}
	\textstyle\big(\dd_\phi\:\EVE_{[r,r']}^k\big)(\psi)\textstyle 
	=\dd_e\LT_{\innt \phi}\big(\int \Ad_{[\innt_r^s\phi]^{-1}}(\psi(s))\:\dd s\big)\qquad\quad\forall\: \phi,\psi\in C^k([r,r'],\mg),\:\: [r,r']\in \COMP.
\end{align*}
\end{theorem}
\begin{proof}
Confer Sect.\ \ref{osposdopopsd}. 
\end{proof}
\noindent
We recall \cite{HGGG} that a $\mg$ is said to be {\bf integral complete} \defff $\int\phi(s)\:\dd s\in \mg$ exists for each $\phi\in C^0([0,1],\mg)$.\footnote{Clearly, this is equivalent to require that $\int\phi(s)\:\dd s\in \mg$ exists for each $\phi\in C^0([r,r'],\mg)$, for each $[r,r']\in \COMP$.}  
Theorem \ref{kdfklfklfdkjljljjllk} will be a consequence of the more general Theorem \ref{kckjckjs}, being concerned with differentiation of parameter dependent integrals.  
The key point of the whole discussion is that
	if $G$ is k-continuous and $C^k$-semiregular for $k\in \NN\sqcup\{\lip,\infty\}$, then the directional derivative of $\EVE_{[r,r']}^k$ at zero along some $\phi\in C^k([r,r'],\mg)$ always exists; namely, in the completion of $\mgc$ of $\mg$ -- as explicitly given by 
\begin{align}
\label{dsdsdsdsds}
	\textstyle\frac{\dd}{\dd h}\big|_{h=0}\he \innt h\cdot \phi=\int \phi(s)\:\dd s\in \mgc.
\end{align}  
We thus have to clarify this elementary issues first.
\subsection{Differentiation at Zero}
We fix $[r,r']\in \COMP$ in the following; and let $\comp{\mg}$ and $\comp{E}$ denote the completions of $\mg$ and $E$, respectively. By Lemma \ref{pofsdisfdodjjxcycxj}, then $\dd_e \chart\colon \mg\rightarrow E$ extends uniquely to a continuous isomorphism $\comp{\dd_e\chart}\colon \comp{\mg}\rightarrow \comp{E}$. 
In order to prove \eqref{dsdsdsdsds}, we now first need to show that $\phi\in C^k([r,r'],\mg)$, and $\dind\llleq k\in \NN\sqcup \{\infty\}$ given, there exists a sequence $\{\phi_n\}_{n\in \NN}\subseteq C^\infty([r,r'],\mg)$ with
\begin{align*}
	\textstyle\lim_{n\rightarrow \infty} \ppp_\infty^\dind(\phi-\phi_n)= 0 \qquad\quad \text{as well as}\qquad\quad \int_r^\bullet \phi_n(s) \:\dd s\in \mg\qquad\forall\: n\in \NN.
\end{align*}
Such a sequence can be obtained, e.g., by approximating $\phi^{(\dind)}$ by polygonal curves, smoothening them by convolution, and then integrating them $\dind$-times.  
Basically, then \eqref{dsdsdsdsds} follows from the triangle inequality and Proposition \ref{hghghggh}.
\subsubsection*{Polygons and Convolution:} 
We let $\finite$ denote the set of all finite dimensional linear subspaces $F\subseteq\mg$ of $\mg$; and define 
\begin{align*}
	\textstyle C^k(D,\finite):=\bigsqcup_{F\in \finite } C^k(D,F)\qquad\quad\forall\: D\in \INT,\:\: k\in\NN\sqcup\{\infty\}.
\end{align*}
Moreover, for each $n\geq 1 $, we fix $\rho_n\colon (-1/n,1/n)\rightarrow \RR_{\geq 0}$ smooth and compactly supported with $\int \rho_n(s)\: \dd s=1$. Then, for $\chi\in C^0(I,\finite)$ with $[r,r']\subseteq I$ ($I\subseteq \RR$ an open interval) given, we choose $m \geq 1$ such large that $[r,r']+(-1/m,1/m)\subseteq I$ holds; and define (convolution)
\begin{align*}
	\textstyle C^\infty([r,r'],\finite)\ni \chi*\rho_n\colon [r,r']\ni t\mapsto \int_{t-1/n}^{t+1/n} \rho_n(t-s)\cdot \chi(s)\:\dd s \qquad\quad\forall\: n\geq m.
\end{align*} 
Clearly, $\{\chi*\rho_n\}_{n\geq m}\rightarrow \chi|_{[r,r']}$ converges uniformly (w.r.t.\ the seminorms $\{\ppp_\infty\}_{\pp\in \SEM}$) with
\begin{align*}
	(\chi*\rho_n)^{(p)}=\chi*\rho_n^{(p)}\qquad\quad\forall\: p\in \NN,\:\: n\geq m.
\end{align*}
Let now $\Poly([r,r'],\mg)\subseteq C^0([r,r'],\finite)$ denote the set of all maps $\chi\colon [r,r']\rightarrow \mg$, such that there exist $r=t_0<{\dots}<t_n=r'$ and $X_0,\dots,X_{n-1}\in \mg$ with
\begin{align*}
	\chi(t_p+\tau)=\chi(t_p)+\tau\cdot X_p\qquad\quad\forall\: p=0,\dots,n-1,\:\: \tau\leq t_{p+1}-t_p.
\end{align*}
Clearly, for each $\psi\in C^0([r,r'],\mg)$, there exists a sequence $\{\chi_n\}_{n\in \NN}\subseteq \Poly([r,r'],\mg)$ with $\{\chi_n\}_{n\in \NN}\rightarrow \psi$ uniformly; and, combining this with the statements made above, we easily obtain that
\begin{lemma}
\label{podpodspodspodspo}
	For each $\psi\in C^0([r,r'],\mg)$, there exists a sequence $\{\psi_n\}_{n\in \NN}\subseteq C^\infty([r,r'],\finite)$ with $\{\psi_n\}_{n\in \NN}\rightarrow \psi$ uniformly.
\end{lemma}

\subsubsection*{Iterated Integration:}
We define $\INTE[p]\colon \mg^p\times C^0([r,r'],\mgc)\rightarrow C^{p}([r,r'],\mgc)$ for $p\geq 1$, inductively by 
\begin{align*}
	\textstyle\INTE[1]\colon \mg\times C^0([r,r'],\mgc)\rightarrow C^1([r,r'],\mgc),\qquad (X,\phi)\mapsto X+\int_r^\bullet \phi(s)\:\dd s
\end{align*}
as well as  
\begin{align*}
	\INTE[p](X_1,\dots,X_p,\phi):=\INTE[1](X_p, \INTE[p-1](X_{p-1},\dots,X_{1},\phi))
\end{align*}
for all $X_1,\dots,X_p$ and $\phi\in C^0([r,r'],\mgc)$, for $p\geq 2$. 
Evidently, 
\begingroup
\setlength{\leftmargini}{12pt}
\begin{itemize}
\item
	for all $\phi\in C^0([r,r'],\finite)$ and $X_1,\dots,X_p\in \mg$, we have $\INTE[p](X_1,\dots,X_p,\phi) \in C^p([r,r'],\finite)$. 
\item
	for all $\phi,\psi\in C^0([r,r'],\mgc)$ and $X_1,\dots,X_p\in \mg$, we have 
	\begin{align*}
		\INTE[p](X_1,\dots,X_p,\phi)-\INTE[p](X_1,\dots,X_p,\psi)=\INTE[p](0,\dots,0,\phi-\psi).
\end{align*}
\vspace{-18pt}		
\item
for all $\phi\in C^p([r,r'],\mgc)$, we have $\phi=\INTE[p]\big(\phi^{(p-1)}(r),\dots,\phi^{(0)}(r),\phi^{(p)}\big)$.
\item
for all $\phi\in C^0([r,r'],\mgc)$, $p\geq 1$, and $X_1,\dots,X_p$, we have $\INTE[p](X_1,\dots,X_p,\phi)^{(p)}=\phi$ as well as
\begin{align*}
	\INTE[p](X_1,\dots,X_p,\phi)^{(\dind)}=\INTE[p-\dind](X_1,\dots,X_{p-\dind},\phi)\qquad\quad\forall\: 1\leq \dind\leq p-1.
\end{align*}
\vspace{-18pt}
\item
for all $\phi\in C^0([r,r'],\mgc)$, $p\geq 1$, and $\qq\in \SEM$, we have (apply Lemma \ref{ofdpofdpopssssaaaasfffff} successively)
\begin{align*}
	\textstyle_\cdot\ovl{\qq}_\infty(\INTE[p](0,\dots,0,\phi))\leq |r'-r|^p\cdot \ovl{\qq}_\infty(\phi).
\end{align*}
The previous (and the first) point thus shows that for $0 \leq \dindu\leq \dind\in \NN$, we have
\begin{align*}
	\qqq_\infty^\dindu(\INTE[\dind](0,\dots,0,\phi))\leq \max(1,|r'-r|)^{\dind}\cdot \qqq_\infty(\phi)\qquad\quad\forall\: \qq\in \SEM,\:\: \phi\in C^0([r,r'],\finite).
\end{align*}
\end{itemize}
\endgroup
\subsubsection*{Specific Estimates:}
Let $\dind\llleq  k\in \NN\sqcup\{\lip,\infty\}$, $\mm\in \SEM$, and $\phi\in C^k([r,r'],\mg)$ be given. 
\begingroup
\setlength{\leftmargini}{12pt}
\begin{itemize}
\item
We choose $\{\psi_n\}_{n\in \NN}\subseteq C^\infty([r,r'],\finite)$ with $\{\psi_n\}_{n\in \NN}\rightarrow \phi^{(\dind)}$ uniformly 
 (Lemma \ref{podpodspodspodspo}); and define
\begin{align*}
	\phi_n:=\INTE[\dind]\big(\phi^{(\dind-1)}(r),\dots,\phi^{(0)}(r),\psi_n\big)\in C^\infty([r,r'],\finite)\qquad\quad\forall\: n\in \NN.
\end{align*}
\vspace{-18pt}
\item
We conclude from the third-, second-, and the last point in the previous part that
\begin{align}
\label{pofdpofdpofdpopofdfdfd}
\begin{split}
	\qqq_\infty^{\dind}(\phi-\phi_n)&=  \qqq_\infty^{\dind}\big(\INTE[\dind]\big(\phi^{(\dind-1)}(r),\dots,\phi^{(0)}(r),\phi^{(\dind)}\big)-\INTE[\dind]\big(\phi^{(\dind-1)}(r),\dots,\phi^{(0)}(r),\psi_n\big)\big)\\[1pt]
	&=  \qqq_\infty^{\dind}\big(\INTE[\dind]\big(0,\dots,0,\phi^{(\dind)}-\psi_n\big)\big)\\
	&\leq  \max(1,|r'-r|)^{\dind}\cdot \qqq_\infty\big(\phi^{(\dind)}-\psi_n\big)
	\end{split}
\end{align}
holds for each $\qq\in \SEM$; thus, 
\begin{align}
\label{pofdpofdpofdpopofdfd}
	\qqq_\infty^{\dind}(\phi_n)&\leq \qqq_\infty^{\dind}(\phi)+ \max(1,|r'-r|)^{\dind}\cdot \qqq_\infty\big(\phi^{(\dind)}-\psi_n\big)\qquad\quad\forall\: n\in \NN.
\end{align}
\vspace{-18pt}
\item
For each $h\in \RR$ and $n\in \NN$, we define 
\begin{align*}
	\textstyle\gamma_{h,n}:= h\cdot \dd_e\chart(\int_r^\bullet \phi_n(s)\:\dd s)\equiv h\cdot \dd_e\chart\cp\INTE[\dind+1]\big(\phi^{(\dind-1)}(r),\dots,\phi^{(0)}(r),0,\psi_n\big)\in C^{\infty}([r,r'],E);
\end{align*}
and obtain from \eqref{pofdpofdpofdpopofdfd} that 
\begin{align}
\label{iufdiufdiuiufdiufd}
\begin{split}
	\ww^\dind_\infty(\gamma_{h,n})&\textstyle= |h|\cdot \www^\dind_\infty\big(\int_r^\bullet \phi_n(s)\: \dd s\big)\\
	&\leq |h|\cdot |r'-r|\cdot \www_\infty^\dind(\phi_n)\\
	&\textstyle\leq |h|\cdot |r'-r|\cdot \big(\www_\infty^{\dind}(\phi)+ \max(1,|r'-r|)^{\dind}\cdot \www_\infty\big(\phi^{(\dind)}-\psi_n\big)\big)
	\end{split}
\end{align}
holds for each $\ww\in \SEM$.
\item
Since $\bound:=\im[\phi^{(\dind)}]\cup \bigcup_{n \in \NN}\im[\psi_n]$ is bounded, there exists some $\delta>0$, such that 
\begin{align*}
	\mu_{h,n}:=\chartinv\cp\gamma_{h,n}\qquad\quad\text{and}\qquad\quad \phi_{h,n}:=\Der(\mu_{h,n})=\dermapdiff(\gamma_{h,n},\dot\gamma_{h,n})=h\cdot\dermapdiff(\gamma_{h,n},\dd_e\chart(\phi_n)) 
\end{align*}
 are defined, for each $|h|\leq \delta$ and $n\in \NN$. 
\item
 Then, for $\mm\in \SEM$ fixed, Lemma \ref{oopxcxopcoxpopcx}.\ref{oopxcxopcoxpopcx1} applied to $\Omega\equiv \dermapdiff(\cdot,\dd_e\chart(\cdot))$, $\gamma\equiv\gamma_{h,n}$, $\psi\equiv \phi_n$, $\pp\equiv \mm$, provides us with certain seminorms $\qq,\ww\in \SEM$, such that
\begin{align*}
	\www_\infty^\dind(\gamma_{h,n})\leq 1\qquad\quad\Longrightarrow\qquad\quad 
	\mmm_\infty^\dind(\phi_{h,n})=|h|\cdot \mm_\infty^\dind(\Omega(\gamma_{h,n},\phi_{n}))\leq |h|\cdot \qqq_\infty^\dind(\phi_n).
\end{align*}
Thus, shrinking $\delta$ if necessary, by \eqref{pofdpofdpofdpopofdfd} and boundedness of $\bound$, we can achieve that
\begin{align}
\label{nasmnasnmvbvbvbvsaansnmasnmsa}
	\mmm_\infty^\dind(h\cdot \phi)\leq 1,\:\:\mmm_\infty^\dind(\phi_{h,n})\leq 1,\:\: \mmm_\infty^\dind(h\cdot \phi-\phi_{h,n})\:\:\leq\:\: 1\qquad\quad\forall\: |h|\leq \delta,\:\: n\in \NN.
\end{align}
\end{itemize}
\endgroup
\noindent
We now have everything we need to prove
\begin{proposition}
\label{rererererr}
Suppose that $G$ is \emph{k-continuous} for $k\in \NN\sqcup \{\lip,\infty\}$; and that $(-\delta,\delta)\cdot \phi\subseteq \DIDE^k_{[r,r']}$ holds for some $\phi\in C^k([r,r'],\mg)$ and $\delta>0$.
Then, we have
\begin{align*}
	\textstyle\frac{\dd}{\dd h}\big|_{h=0} \innt h\cdot \phi=\int \phi(s)\:\dd s\in \comp{\mg}.
\end{align*}
\end{proposition}
\begin{proof}
We fix $\pp\in \SEM$, and have 
to show that\footnote{For $|h|\leq \delta$ suitably small, this is defined by Lemma \ref{opdfopfdopfdpoas}.}
\begin{align*}
	\Delta_\phi(h)&:=\textstyle\rcf{|h|} \cdot \textstyle\cpp\big(\chart(\innt h\cdot \phi)-h\cdot  \comp{\dd_e\chart}(\int \phi(s) \:\dd s)\big)
\end{align*}
tends to zero if $h$ tends to zero. For this, we choose $\pp\leq \mm$ and $\dind\llleq k$ as in Proposition \ref{hghghggh}; and let $\{\phi_{h,n}\}_{n\in \NN}\subseteq C^\infty([r,r'],\mg)$, $\{\gamma_{h,n}\}_{n\in \NN}\subseteq C^{\infty}([r,r'],\mg)$, $\{\mu_{h,n}\}_{n\in \NN}\subseteq C^{\infty}([r,r'],G)$, $\delta>0$ be as above. 
Then, Proposition \ref{hghghggh} and \eqref{nasmnasnmvbvbvbvsaansnmasnmsa} (fourth step) show that for $|h|\leq \delta$, we have
\begin{align*}
	\Delta_\phi(h)&\textstyle\leq \rcf{|h|}\cdot \pp\big(\chart(\innt h\cdot \phi)-h\cdot  \dd_e\chart(\int \phi_n(s) \:\dd s)\big) \hspace{13.5pt} + \:\:\: \cpp\big(\ovl{\dd_e\chart}(\int \phi(s) \:\dd s)-\dd_e\chart(\int \phi_n(s) \:\dd s)\big)\\
&\leq \textstyle\rcf{|h|} \cdot\textstyle\pp\big(\chart(\innt h\cdot \phi)-\gamma_{h,n}(r')\big)\hspace{67.5pt}+ \:\:\: \int (\cpp\cp\comp{\dd_e\chart})(\phi(s)-\phi_n(s)) \:\dd s\\
	&=\textstyle\rcf{|h|} \cdot\textstyle\pp\big(\chart(\innt h\cdot \phi)-\chart(\mu_{h,n}(r'))\big)\hspace{50.8pt}+ \:\:\: \int \ppp(\phi(s)-\phi_n(s)) \:\dd s\\
	&\leq \textstyle\rcf{|h|}\cdot\textstyle \int \mmm\big(h\cdot\phi(s)-\phi_{h,n}(s)\big)\:\dd s
	\hspace{51.9pt}+ \:\:\: \int \ppp(\phi(s)-\phi_n(s)) \:\dd s	
	\\
	&\textstyle = 
	\rcfsp 
	\int \mmm\big(\phi(s)-\w(\gamma_{h,n}(s),\dd_e\chart(\phi_n(s)))\big)\:\dd s
	\hspace{7.8pt}+ \:\:\: \int \ppp(\phi(s)-\phi_n(s)) \:\dd s.
\end{align*}
Let now $\varepsilon>0$ be fixed. By \eqref{pofdpofdpofdpopofdfdfd}, there exists some $n_\varepsilon\in \NN$, such that the second summand is bounded by $\varepsilon/3$ for each $n\geq n_\varepsilon$. 
Moreover, since $\phi=\dermapdiff(0,\dd_e\chart(\phi))$ holds, we can estimate the first summand by
\begin{align}
\label{ksdkldsklsksdlklsdlk}
\begin{split}
	\textstyle\int \mmm\big(\phi(s)-\w(\gamma_{h,n}(s),\dd_e\chart(\phi_n(s)))\big)\:\dd s	
	&\textstyle=\int \mmm\big(\w(0,\dd_e\chart(\phi(s))))-\w(\gamma_{h,n}(s),\dd_e\chart(\phi_n(s)))\big)\:\dd s
	\\[3pt]
	&\textstyle\leq \int \mmm\big(\w(0,\dd_e\chart(\phi(s)))-\w(\gamma_{h,n}(s),\dd_e\chart(\phi(s)))\big)\:\dd s\textstyle\\[3pt]
	&\textstyle\quad\he +\int \mmm\big(\w(\gamma_{h,n}(s),\dd_e\chart(\phi(s)-\phi_n(s)))\big)\:\dd s.
\end{split}
\end{align}
\vspace{-10pt}

\noindent
Then, 
\begingroup
\setlength{\leftmargini}{12pt}
\begin{itemize}
\item
Since $\im[\phi]$ is compact, we can achieve that the second line in \eqref{ksdkldsklsksdlklsdlk} is bounded by $\varepsilon/3$ for each $n\in \NN$, just by shrinking $\delta$ if necessary.
\item
In order to estimate the third line in \eqref{ksdkldsklsksdlklsdlk}, we choose $\mm\leq \ww$ as in \eqref{cpocpoxjdsndscxaaabc} for $\vv\equiv \mm$ there. Then, by \eqref{iufdiufdiuiufdiufd}, we can achieve that $\ww_\infty(\gamma_{h,n})\leq \ww^\dind_\infty(\gamma_{h,n})\leq 1$ holds for each $|h|\leq \delta$, for $\delta>0$ suitably small; and obtain
\vspace{-8pt}
\begin{align*}
	\mmm_\infty\big(\w(\gamma_{h,n},\dd_e\chart(\phi-\phi_n))\big) \stackrel{\eqref{cpocpoxjdsndscxaaabc}}{\leq} \www_\infty(\phi-\phi_n)
	\qquad\quad\forall\:  |h|\leq \delta.
\end{align*}  
It is then clear from \eqref{pofdpofdpofdpopofdfdfd} that for $n_\varepsilon'\geq n_\varepsilon$ suitably large, the third line in \eqref{ksdkldsklsksdlklsdlk} is bonded by $\varepsilon/3$ for each $n\geq n'_\varepsilon$ and $|h|\leq \delta$.  
\end{itemize}
\endgroup
\noindent
We thus have $\Delta_\phi(h)\leq \varepsilon$  for each $|h|\leq \delta$; and conclude that $\lim_{h\rightarrow 0}\Delta_\phi(h)=0$ holds.
\end{proof}	 
We immediately obtain
\begin{corollary}
\label{MC}
\noindent
\begingroup
\setlength{\leftmargini}{17pt}
\begin{enumerate}
\item
Suppose that $G$ is {\rm 0}-continuous and $C^0$-semiregular. Then,   
 $\EVE^0_{[0,1]}$ is differentiable at zero \deff $\mg$ is \emph{integral complete}.
\item
Suppose that $G$ is {\rm k}-continuous and $C^\infty$-semiregular for $k\in \NN\sqcup\{\lip,\infty\}$. Then,   
 $\EVE^k_{[0,1]}\big|_{C^\infty([0,1],\mg)}$ is differentiable at zero \deff $\mg$ is Mackey complete.
\end{enumerate}
\endgroup
\end{corollary}
\begin{proof}
This is clear from Proposition \ref{rererererr} and Corollary \ref{Mackeycor}.
\end{proof}

\subsection{Integrals with Parameters}
\label{popodspodsdspodspodspaaaa}
Given an open interval $J\subseteq \RR$ as well as $x\in J$, in the following, we denote 
\begin{align*}
	J[x]:=\{h\in \RR_{\neq 0}\:|\: x+h\in J\}.
\end{align*} 
We now will discuss the differentiation of parameter-dependent integrals. For this, we let $[r,r']\in \COMP$ be fixed; and observe that 
\begin{corollary}
\label{asghasgasghashsa}
Let $G$ be $C^k$-semiregular for $k\in \NN\sqcup\{\lip,\infty\}$. Then, for $\phi,\psi,\chi \in C^k([r,r'],\mg)$, we have
\begin{align*}
	\textstyle\innt (\phi+\psi+\chi)=\alpha\cdot \beta\cdot \gamma\qquad\quad\text{for}\qquad\quad\alpha:=\innt \phi,\quad \beta:=\innt\Ad_{\alpha^{-1}}(\psi),\quad \gamma:=\innt \Ad_{(\alpha\cdot\beta)^{-1}}(\chi).
\end{align*}
\end{corollary}
\begin{proof}
Applying \ref{kdskdsdkdslkds} twice, we obtain 
	\begin{align*}
		\textstyle\beta^{-1}\cdot\alpha^{-1}\cdot [\innt \phi+\psi+\chi]=\beta^{-1}\cdot\innt_r^\bullet \Ad_{\alpha^{-1}}(\psi+\chi)=\innt_r^\bullet \Ad_{(\alpha\cdot\beta)^{-1}}(\chi),
	\end{align*} 
	because $\Ad_{\alpha^{-1}}(\psi+\chi),\: \Ad_{(\alpha\cdot\beta)^{-1}}(\chi)\in C^k([r,r'],\mg)$ holds by Lemma \ref{Adlip}.
\end{proof}
Moreover, let $\delta>0$, and suppose that $\mu,\nu\colon [0,\delta]\rightarrow G$ are maps with
\begin{align*}
	\textstyle\lim_{h\rightarrow 0} \mu(h)=\mu(0)=e&\textstyle\qquad\quad\text{and}\qquad\quad \lim_{h\rightarrow 0} \rcf{h}\cdot (\chart\cp \mu)(h)=X\in \ovl{E}\\
	\textstyle\lim_{h\rightarrow 0} \nu(h)=\nu(0)=e&\textstyle\qquad\quad\text{and}\qquad\quad \lim_{h\rightarrow 0} \rcf{h}\cdot (\chart\cp \nu)(h)=Y\in \ovl{E}.
\end{align*}
Then, we obtain from Lemma \ref{sdsdds} that, cf.\ Appendix \ref{dsdsdsds} 
\begin{align}
	\label{powqopqwopwasasaassaq}
	\textstyle\lim_{h\rightarrow 0} \rcf{h}\cdot \chart(\mu(h)\cdot \nu(h))=X+Y\in \ovl{E}
 \end{align}
 holds; and are ready for
\begin{theorem}
\label{kckjckjs}
Suppose that $G$ is {\rm k}-continuous and $C^k$-semiregular for some $k\in \NN\sqcup\{\lip,\infty\}$; and let $\Phi\colon I\times [r,r']\rightarrow \mg$ ($I\subseteq \RR$ open) be fixed with $\Phi(z,\cdot)\in C^k([r,r'],\mg)$ for each $z\in I$.  
Then, 
\begin{align*}
	\textstyle\frac{\dd}{\dd h}\big|_{h=0} \big([\innt \Phi(x,\cdot)]^{-1}[\innt\Phi(x+h,\cdot)]\big)=\textstyle\int \Ad_{[\innt_r^s\Phi(x,\cdot)]^{-1}}(\partial_z\Phi(x,s))\:\dd s\hspace{1pt}\in \comp{\mg}
\end{align*}
holds for $x\in I$, provided that
\begingroup
\setlength{\leftmargini}{17pt}{
\renewcommand{\theenumi}{{\alph{enumi}})} 
\renewcommand{\labelenumi}{\theenumi}
\begin{enumerate}
\item
\label{saasaassasasa2}
We have $(\partial_z \Phi)(x,\cdot)\in C^k([r,r'],\mg)$.\footnote{More specifically, this means that for each $t\in [r,r']$ the map $I\ni z\mapsto \Phi(z,t)$ is differentiable at $z=x$ with derivative $(\partial_z\Phi)(x,t)$, such that $(\partial_z\Phi)(x,\cdot)\in C^k([r,r'],\mg)$ holds. In particular, the latter condition ensures that $\ppp^\dind_\infty((\partial_z\Phi)(x,\cdot))< \infty$ holds For each $\pp\in \SEM$ and $\dind\llleq k$, cf.\ \ref{aaaaaw2}.}
\item
\label{saasaassasasa1}
For each $\pp\in \SEM$ and $\dind\llleq k$, there exists $L_{\pp,\dind}\geq 0$, as well as $I_{\pp,\dind}\subseteq I$ open with $x\in I_{\pp,\dind}$, such that
\begin{align*}
	1/|h|\cdot\ppp^\dind_\infty(\Phi(x+h,\cdot)-\Phi(x,\cdot))\leq L_{\pp,\dind}\qquad\quad \forall\: h\in I_{\pp,\dind}[x].
\end{align*}
\vspace{-22pt}
\end{enumerate}}
\endgroup
\end{theorem}
\begin{proof}
For $x+h\in I$, we have 
\begin{align*}
	\Phi(x+h,t)=\Phi(x,t)+h\cdot \partial_z\Phi(x,t)+ h\cdot \varepsilon(x+h,t)\qquad\quad\forall\: t\in [r,r'],
\end{align*}
for some $\varepsilon\colon I\times [r,r']\rightarrow\mg$ with
\begingroup
\setlength{\leftmargini}{17pt}{
\renewcommand{\theenumi}{{\bf \roman{enumi}})} 
\renewcommand{\labelenumi}{\theenumi}
\begin{enumerate}
\item
\label{aaaaaw1}
	$\lim_{h\rightarrow 0}\varepsilon(x+h,t)=\varepsilon(x,t)=0$\hspace{128.2pt}$\forall\: t\in [r,r']$,
\item
\label{aaaaaw2}
	$\ppp^\dind_\infty(\varepsilon(x+h,\cdot))\leq L_{\pp,\dind} + \ppp_{\infty}^\dind((\partial_z\Phi)(x,\cdot))=:C_{\pp,\dind}<\infty$\hspace{19pt}$\forall\:h\in I_{\pp,\dind}[x]$\quad for all\quad $\pp\in \SEM$,\: $\dind\llleq k$.
\end{enumerate}}
\endgroup
\noindent
Then, {\it \ref{saasaassasasa2}} together with Corollary \ref{asghasgasghashsa} shows that  
	$\textstyle\innt \Phi(x +h,\cdot) =\textstyle \alpha(1)\cdot \beta(h,1) \cdot \gamma(h,1)$ holds, with    
\begin{align*}
\textstyle	\alpha(t)&\textstyle:=\innt_r^t \Phi(x,\cdot)\\
\beta(h,t)&\textstyle:=\innt_r^t h\cdot \Ad_{\alpha^{-1}}(\partial_z\Phi(x,\cdot))\\[0.5pt]
\gamma(h,t)&\textstyle:= \innt_r^t h\cdot \Ad_{(\alpha\cdot\beta(h,\cdot))^{-1}}(\varepsilon(x+h,\cdot))
\end{align*}
for each $t\in [r,r']$; thus,
\begin{align*}
	\textstyle\frac{\dd}{\dd h}\big|_{h=0}\: \chart\big([\innt \Phi(x,\cdot)]^{-1}[\innt\Phi(x+h,\cdot)]\big)=\frac{\dd}{\dd h}\big|_{h=0}\: \chart(\beta(h,1)\cdot \gamma(h,1))
\end{align*}
provided that the right side exists. Now, since $\Ad_{\alpha^{-1}}(\partial_z\Phi(x,\cdot))$ is of class $C^k$ by Lemma \ref{Adlip}, Proposition \ref{rererererr} shows that
\begin{align*}
	\textstyle\frac{\dd}{\dd h}\big|_{h=0}\:\beta(h,1)=\int \Ad_{\alpha^{-1}(s)}(\partial_z\Phi(x,s))\:\dd s = \int \Ad_{[\innt_r^s \Phi(x,\cdot)]^{-1}}(\partial_z\Phi(x,s))\:\dd s
\end{align*}
holds; so that the claim follows from \eqref{powqopqwopwasasaassaq} once we have verified that
\begin{align}
\label{dhjskdkjsd}
	\textstyle\lim_{h\rightarrow 0} 1/|h| \cdot(\pp\cp\chart)(\gamma(h,1))=0\qquad\quad\forall\: \pp\in \SEM.
\end{align}
To show this, we fix $\pp\in \SEM$, and let 
\begingroup
\setlength{\leftmargini}{12pt}
\begin{itemize}
\item
 $\pp\leq \qq\in \SEM$, $\dindu\llleq k$ be as in Lemma \ref{opdfopfdopfdpoas} for $\dind\equiv \dindu$ (and $p\equiv k$) there; i.e.,
 \begin{align}
 \label{opsdopdsopopaosaas}
	\textstyle\qqq^\dindu_\infty(\phi)\leq 1\quad\:\:\text{for}\quad\:\: \phi\in \DIDE^k_{[r,r']}\qquad\quad\Longrightarrow\qquad\quad (\pp\cp\chart)\big(\innt_r^\bullet\phi\big)\leq \int_r^\bullet \qqq(\phi(s))\:\dd s. 
\end{align} 
\item
 $\qq\leq \mm\in \SEM$, $\dind\llleq k$ be as in Lemma \ref{jlkfdsjasasslkfdsjklfdsjklfsdjkl} for $\pp\equiv \qq$ there; i.e., we have
 \begin{align}
 \label{ofdpoifdpofdfdfdfaaaa}
 	\qqq^\dindp(\Ad_{\beta^{-1}(h,\cdot)}(\psi))\leq \mmm^\dindp(\psi)\qquad\quad\forall\:\psi\in C^k([r,r'],\mg),\:\: 0\leq \dindp\leq\dindu,
 \end{align}
provided that $\mmm^\dind_\infty(h\cdot  \Ad_{\alpha^{-1}}(\partial_z\Phi(x,\cdot))\leq 1$ holds.
\item
 $\mm\leq \nn\in \SEM$ be as in Lemma \ref{opopsopsdopds} for $\pp\equiv\mm$, $\qq\equiv \nn$, $\dind\equiv\dindo:=\max(\dind,\dindu)$, and $\phi \equiv \Phi(x,\cdot)$ there; i.e., we have
\begin{align}
\label{spoaspopaossa}
\mmm^\dind(h\cdot  \Ad_{\alpha^{-1}}(\partial_z\Phi(x,\cdot))&\leq |h|\cdot \nnn^\dind(\partial_z\Phi(x,\cdot))\\
\label{spoaspopaossas}
	\mmm^\dindp(\Ad_{\alpha^{-1}}(\varepsilon(x+h,\cdot)))\hspace{3pt}&\leq\hspace{20.3pt}  \nnn^\dindp(\varepsilon(x+h,\cdot)).
\end{align} 
for each $0\leq\dindp\leq \dindo$ and $h\in I[x]$. 
\end{itemize}
\endgroup
\noindent
We choose $\delta>0$ such small that $(-\delta,0)\cup(0,\delta)\subseteq I_{\nn,\dindo}[x]$ holds, with $|h|\cdot \nnn_\infty^\dind(\partial_z\Phi(x,\cdot))\leq 1$ for each $0<|h|\leq \delta$.  
Then,  
\eqref{spoaspopaossa}, \eqref{ofdpoifdpofdfdfdfaaaa}, \eqref{spoaspopaossas}, and \ref{aaaaaw2} show that
\begin{align}
\label{podsopsdpod}
\begin{split}
	\qqq^\dindp\big(h\cdot \Ad_{(\alpha\cdot\beta(h,\cdot))^{-1}}(\varepsilon(x+h,\cdot))\big)&=|h|\cdot \qqq^\dindp\big( \Ad_{\beta^{-1}(h,\cdot)}\cp\Ad_{\alpha^{-1}}(\varepsilon(x+h,\cdot))\big)\\
	&\leq |h|\cdot\mmm^\dindp( \Ad_{\alpha^{-1}}(\varepsilon(x+h,\cdot)))\\
	&\leq |h|\cdot \nnn^\dindp(\varepsilon(x+h,\cdot))\\
	&\leq |h|\cdot \nnn_\infty^\dindo(\varepsilon(x+h,\cdot))\\
	&\leq  |h|\cdot C_{\nn,\dindo} 
\end{split} 
\end{align}
holds for each $0<|h|\leq \delta$, and each $0\leq \dindp\leq\dindu$. 
Thus, shrinking $\delta$ if necessary, we obtain from \eqref{opsdopdsopopaosaas}, as well as \eqref{podsopsdpod} for $\dindp\equiv\dindu$ there, that
\begin{align*}
	\textstyle 1/|h| \cdot(\pp\cp\chart)(\gamma(h,t))&\textstyle\leq \int_r^t \qqq(\Ad_{(\alpha(s)\cdot\beta(h,s))^{-1}}(\varepsilon(x+h,s)))\:\dd s\qquad\quad\forall\: 0<|h|\leq \delta
\end{align*}
holds. 
Then, \eqref{podsopsdpod}, for $\dindp\equiv0$ there, gives
\begin{align*}
	\textstyle 1/|h| \cdot(\pp\cp\chart)(\gamma(h,t))\leq \int \nnn(\varepsilon(x+h,s))\:\dd s\qquad\quad\forall\: 0<|h|\leq \delta.
\end{align*}
Since the integrand is measurable, and bounded by \ref{aaaaaw2}; and since $\lim_{h\rightarrow 0}\nn(\varepsilon(x+h,\cdot))=0$ converges pointwise by \ref{aaaaaw1}, the dominated convergence theorem shows \eqref{dhjskdkjsd}.
\end{proof}
\begin{remark}
\label{sddssdsd}
In the situation of Theorem \ref{kckjckjs}, we obtain from Lemma \ref{sdsdds} (and \ref{saasaassasasa1}) that
\begin{align}
\label{sopsodpdops1}
	\textstyle\frac{\dd}{\dd h}\big|_{h=0} \innt\Phi(x+h,\cdot)=\textstyle\dd_e\LT_{\innt \Phi(x,\cdot)}\big(\int \Ad_{[\innt_r^s\Phi(x,\cdot)]^{-1}}(\partial_z\Phi(x,s))\:\dd s\big)\in\mg
\end{align}
holds, provided that
\vspace{-1pt}
\begingroup
\setlength{\leftmargini}{12pt}
\begin{itemize}
\item
	$\mg$ is integral complete.
\item
	$\mg$ is Mackey complete with $\partial_z\Phi(x,\cdot)\in C^\lip([r,r'],\mg)$.
\end{itemize}
\endgroup
\vspace{-1pt}
\noindent
Here, the first criterion is obvious, and the second one is clear from Lemma \ref{Adlip}.\hspace*{\fill}$\ddagger$ 
\end{remark}

\subsection{Duhamel's Formula}
\label{kjdkjdjkdkjd}
Suppose that $G$ is $\infty$-continuous and $C^\infty$-semiregular, and that $\mg$ is Mackey complete. We fix $\MX\colon I\rightarrow \mg$ of class $C^1$, and define
\begin{align*}
	\Phi\colon I\times [0,1]\rightarrow \mg,\qquad (z,t)\mapsto \MX(z).	 
\end{align*}
Then, $\Phi$ fulfills the presumptions of Theorem \ref{kckjckjs} for $[r,r']\equiv[0,1]$ there, namely, for each $x\in I$; so that we have
\begin{corollary}
\label{sasassasasa}
Suppose that $G$ is $\infty$-continuous and $C^\infty$-semiregular, and that $\mg$ is Mackey complete. Then, for each $\MX\colon I\rightarrow \mg$ of class $C^1$, we have
\begin{align*}
	\textstyle\partial_z \exp(\MX(x))=\dd_e\LT_{\exp(\MX(x))}\big(\int_0^1 \Ad_{\exp(-s\cdot \MX(x))}(\partial_z\MX(x)) \:\dd s \big)\qquad\quad\forall\: x\in I.
\end{align*}
\end{corollary}
\begin{proof}
Clear. 
\end{proof}
We want to provide a further version of this statement:
 
Referring to Lemma \ref{opqwopqwowpwopqwq}, we say that $G$ is {\bf quasi constricted} \defff 
\begin{align*}
	\textstyle\alpha_{X,Y}\colon \RR\ni t\mapsto \sum_{n=0}^\infty \frac{t^n}{n!}\cdot (\com{X})^n(Y)\in \mg\qquad\quad\forall\: X,Y\in \mg
\end{align*}

is defined and of class $C^1$ with $\dot\alpha_{X,Y}=\bl X,\alpha\br$; thus, of class $C^\infty$ by Corollary \ref{bhsbsshshdkksjdhjsd}.
Then, 

by Lemma \ref{Potreih}, we have
\begin{align}
\label{odsoidoioisdoidsoids}
	\textstyle\frac{\id_\mg-\exp(-\com{\MX(x)})}{\com{\MX(x)}}(Y):=\int_0^1 \alpha_{-X,Y}(s)\: \dd s=\sum_{n=0}^\infty \frac{1}{(n+1)!}\cdot (\com{-\MX(x)})^n(Y)\in\mgc\qquad\forall\: Y\in \mg;
\end{align}  

and, in analogy to Corollary \ref{saklasklklsa}, we obtain
\begin{corollary}
\label{saklasklklsaaa}
	Suppose that $G$ is quasi constricted, and admits an exponential map. Then,   
	\begin{align*}
		\Ad_{\exp(-t\cdot X)}(Y)=\alpha_{-X,Y}(t)\qquad\quad\forall\: t\in \RR,\:\: X,Y\in \mg.
	\end{align*}
\end{corollary}
\begin{proof}
The proof is the same as for Corollary \ref{saklasklklsa}, whereby the statement in Lemma \ref{opqwopqwowpwopqwq} now holds by definition. 
\end{proof}
We obtain
\begin{proposition}[{\bf Duhamel's formula}]
\label{sahjhjsahjsahjsahjsaqqwppowqpowq}
Suppose that $G$ is $\rm \infty$-continuous, $C^\infty$-semiregular, and quasi constricted; and that $\mg$ is Mackey complete.    
Then, for each $\MX\colon I\rightarrow \mg$ of class $C^1$, we have
\begin{align*}
	\textstyle\partial_z \exp(\MX(x))\textstyle=\dd_e\LT_{\exp(\MX(x))}\Big(\frac{\id_\mg-\exp(-\com{\MX(x)})}{\com{\MX(x)}}(\partial_z\MX(x))\Big)\qquad\quad\forall\: x\in I.
\end{align*}
\end{proposition}
\begin{proof}
By Corollary \ref{sasassasasa}, we have
\begin{align*}
	\textstyle\partial_z \exp(\MX(x))=\dd_e\LT_{\exp(\MX(x))}\big(\int \Ad_{\exp(-s\cdot \MX(x))}(\partial_z\MX(x)) \:\dd s\big)\qquad\quad\forall\: x\in I.
\end{align*}
We obtain from Corollary \ref{saklasklklsaaa} and Lemma \ref{Potreih} that
\begin{align*}
	\textstyle\int \Ad_{\exp(-s\cdot \MX(x))}(\partial_z\MX(x)) \:\dd s&\textstyle=\int \sum_{n=0}^\infty \frac{s^n}{n!}\cdot (\com{-\MX(x)})^n(\partial_z\MX(x))\:\dd s\\
	&\textstyle=\sum_{n=0}^\infty \frac{1}{(n+1)!}\cdot \com{(-\MX(x)})^n(\partial_z\MX(x))
\end{align*}
holds for each $x\in I$; which is necessarily in $\mg$.
\end{proof}

\subsection{Smoothness of the Integral}
\label{osposdopopsd}
We now are going to prove Theorem \ref{kdfklfklfdkjljljjllk}. 
For this, we first observe that 
\begin{lemma}
\label{fdfdfdf}
Let $\Gamma\colon G\times \mg\rightarrow \mg$ be continuous, and $k\in \NN\sqcup\{\lip,\infty\}$ be fixed. Suppose furthermore that   
$G$ is {\rm k}-continuous and $C^k$-semiregular.
Then,
\begin{align*}
	\textstyle\wh{\Gamma}\colon C^k([r,r'],\mg)\times C^k([r,r'],\mg)\rightarrow \mgc,\qquad (\phi,\psi)\mapsto\int \Gamma\big(\innt_r^{s}\phi,\psi(s)\big)\:\dd s
\end{align*} 
is continuous for each $[r,r']\in \COMP$.
\end{lemma}
\begin{proof}
This follows by standard arguments from Lemma \ref{fddfxxxxfd}, cf.\ Appendix \ref{appLipasasasasas}.
\end{proof}
\noindent
Let now $k\in \NN\sqcup\{\lip,\infty\}$ be fixed; and suppose that $G$ is k-continuous and $C^k$-semiregular, i.e, locally $\mu$-convex  and $C^\lip$-semiregular for $k\equiv \lip$. Suppose furthermore that
\vspace{-2pt}
\begingroup
\setlength{\leftmargini}{12pt}
\begin{itemize}
\item
	$\mg$ is integral complete if $k\equiv 0$ holds.
\item
	$\mg$ is Mackey complete if $k\in \NN_{\geq 1}\sqcup\{\lip, \infty\}$ holds.
\end{itemize}
\endgroup
\vspace{-2pt}
\noindent
Clearly,  
\vspace{-6pt}
\begin{align*}
	\Phi[\phi,\psi]\colon(-1,1)\times [r,r']\rightarrow \mg,\qquad (h,t)\mapsto \phi(t)+ h\cdot \psi(t)
\end{align*}
fulfills the presumptions of Theorem \ref{kckjckjs} for each $\phi,\psi\in C^k([r,r'],\mg)$; i.e., we have, cf.\ Remark \ref{sddssdsd}
\begin{align}
\label{poasopsaosappoaspos}
	\textstyle\big(\dd_\phi\he\EVE_{[r,r']}^k\big)(\psi)\textstyle 
	=\dd_e\LT_{\innt \phi}\big(\int \Ad_{[\innt_r^s\phi]^{-1}}(\psi(s))\:\dd s\big)\qquad\quad\forall\: \phi,\psi\in C^k([r,r'],\mg),\:\: [r,r']\in \COMP.
\end{align}
This can be written as, cf.\ \eqref{LGPR}
\begin{align*}
	\textstyle \dd_\phi\he \EVE_{[r,r']}^k(\psi)=\dd_{(\innt\phi,e)}\mult\big(0,\wh{\Gamma}(\phi,\psi)\big)\qquad\quad\text{for}\qquad\quad  \Gamma\equiv \Ad(\inv(\cdot),\cdot);
\end{align*}
so that $\EVE_{[r,r']}^k$ is of class $C^1$ by Lemma \ref{fdfdfdf}. We thus have
\begin{corollary}
\label{sddsdsdsds}
Let $G$ be {\rm k}-continuous and $C^k$-semiregular for $k\in \NN\sqcup \{\lip,\infty\}$. Suppose furthermore that $\mg$ is 
\begingroup
\setlength{\leftmargini}{12pt}
\begin{itemize}
\item
	integral complete for $k\equiv 0$.
\item
	Mackey complete for $k\in \NN_{\geq 1}\sqcup\{\lip,\infty\}$.
\end{itemize}
\endgroup
\noindent
Then, $\EVE_{[r,r']}^k$ is of class $C^1$ with  
\begin{align*}
	\textstyle\dd_\phi\he \EVE_{[r,r']}^k(\psi)\textstyle 
	=\dd_e\LT_{\innt\phi}\big(\int \Ad_{[\innt_r^s\phi]^{-1}}(\psi(s))\:\dd s\big)\qquad\quad\forall\: \phi,\psi\in C^k([r,r'],\mg) 
\end{align*}
for each $[r,r']\in \COMP$. 
\end{corollary}
\begin{proof}
Clear.
\end{proof}

We are ready for the
\begin{proof}[Proof of Theorem \ref{kdfklfklfdkjljljjllk}]
By Corollary \ref{MC}, it remains to show that $\EVE_{[r,r']}^k$ is smooth for each $[r,r']\in \COMP$ 
\begingroup
\setlength{\leftmargini}{12pt}
\begin{itemize}
\item
	if $\mg$ is integral complete for $k\equiv 0$.
\item
	if $\mg$ is Mackey complete for $k\in \NN_{\geq 1}\sqcup\{\infty\}$.
\end{itemize}
\endgroup
\noindent
Now, since Corollary \ref{sddsdsdsds} shows that $\EVE_{[0,1]}^k$ is of class $C^1$, Theorem E in \cite{HGGG} shows that $\innt_{[0,1]}^k$ is smooth. Then, for $[r,r']\in \COMP$ fixed, we define 
\begin{align*}
	\varrho\colon [0,1]\rightarrow [r,r'],\qquad t\mapsto r+ t\cdot |r'-r|;
\end{align*}
 and recall that (cf.\ proof of Lemma \ref{klklllkjlaaa}) $\EVE_{[r,r']}^k=\EVE_{[0,1]}^k\cp \:\eta$ holds, for the k-continuous, linear map 
 	\begin{align*}
		\eta\colon C^k([r,r'],\mg)\rightarrow C^k([0,1],\mg),\qquad \phi\mapsto \dot\varrho\cdot \rcK{\phi\cp\varrho}\equiv |r'-r|\cdot \rcK{\phi\cp\varrho}.
	\end{align*}
	Since $\eta$ is smooth by \ref{linear}, the claim follows. 
\end{proof}
\begin{remark}
	It is to be expected that Theorem \ref{kdfklfklfdkjljljjllk},\ref{aaaa1} also holds for $k\equiv\lip$; i.e., that we have:
\begingroup
\setlength{\leftmargini}{20pt}
\begin{enumerate}
\item[2')]
If $G$ is $\lip$-continuous and $C^\lip$-semiregular, then $\EVE^\lip_{[r,r']}$ is smooth for each $[r,r']\in \COMP$ \deff $\mg$ is Mackey complete \deff $\EVE^\lip_{[0,1]}$ is differentiable at zero.	
\end{enumerate}
\endgroup
\noindent
Indeed, by Corollary \ref{MC}, it only remains to show that $\EVE^\lip_{[r,r']}$ smooth; whereby (due to the explicit formula \eqref{poasopsaosappoaspos}) Corollary \ref{sddsdsdsds} already shows that $\EVE^\lip_{[r,r']}$ is of class $C^1$. Using similar arguments as in Lemma \ref{fdfdfdf}, it should follow inductively from \eqref{poasopsaosappoaspos} that $\EVE^\lip_{[r,r']}$ is of class $C^\infty$. The details, however, seem to be quite elaborate and technical; so that we leave this issue to a another paper.	\hspace*{\fill}$\ddagger$
\end{remark}

\section*{Acknowledgements}
This work has been supported in part by the Alexander von Humboldt Foundation of Germany and NSF Grant PHY-1505490. The author thanks Stefan Waldmann for general remarks on a draft of the present article.

\addtocontents{toc}{\protect\setcounter{tocdepth}{0}}
\appendix

\section*{APPENDIX}

\section{Appendix to Sect.\ \ref{prelim}}

\subsection{}
\label{appA1}
\begin{proof}[Proof of Lemma \ref{kldskldsksdklsdl}]
Since $\Phi$ is continuous with $\Phi(x,0,\dots,0)=0$, there exist $\qq_1\in \SEMM_1,\dots,\qq_n\in \SEMM_n$ as well as $V\subseteq X$ open with $x\in V$, such that
\begin{align}
\label{fdfddffdfdd}
	(\pp\cp\Phi)(y,Y_1,\dots,Y_n) \leq 1\qquad\quad\forall\: y\in V
\end{align} 
holds for all $Y_1\in \OB_{\qq_1,1},\dots,Y_n\in \OB_{\qq_n,1}$. Let now $X_1\in F_1,\dots,X_n\in F_n$ be fixed; and define 
$$
Y_k:= 
\begin{cases}
\hspace{52pt} X_k\quad\:\: \text{for}\quad\:\: \qq_k(X_k)=0,\\
1/\qq_k(X_k)\cdot X_k\quad\:\: \text{for}\quad\:\:\qq_k(X_k)>0,
\end{cases}
$$
for $k=1,\dots,n$. Then,
\begingroup
\setlength{\leftmargini}{12pt}
\begin{itemize}
\item
if $\qq_1(X_1),\dots,\qq_n(X_n)>0$ holds, we obtain  
\begin{align*}
	(\pp\cp\Phi)(y,X_1,\dots,X_n)\stackrel{\eqref{fdfddffdfdd}}{\leq} \qq_1(X_1)\cdot{\dots}\cdot \qq_n(X_n)\qquad\quad\forall\: y\in V.
\end{align*}
\item
if $\qq_k(X_k)=0$ holds for some $1\leq k\leq n$, we have $\qq_k(n\cdot Y_k)=0$ for each $n\geq 1$; thus,
\begin{align*}
	(\pp\cp\Phi)(y,Y_1,\dots,Y_n)\stackrel{\eqref{fdfddffdfdd}}{\leq} 1/n\qquad\forall\:n\geq 1\qquad\quad
	&\Longrightarrow\qquad\quad (\pp\cp\Phi)(y,Y_1,\dots,Y_n)\hspace{5.3pt}=0\\
	&\Longrightarrow\qquad\quad (\pp\cp\Phi)(y,X_1,\dots,X_n)=0
\end{align*}
for each $y\in V$.
\end{itemize}
\endgroup
\noindent
From this, the claim is clear.
\end{proof}

\subsection{}
\label{appA2}
\begin{proof}[Proof of Corollary \ref{ofdpopfdofdp}]
Since $\compacto$ is compact, by Lemma \ref{kldskldsksdklsdl}, there exist seminorms $\qq[p]_1\in \SEMM_1,\dots,\qq[p]_n\in \SEMM_n$ for $p=1,\dots,m$, as well as $V_1,\dots,V_m\subseteq X$ open with $\compacto\subseteq V_1\cup{\dots}\cup V_m=:O$, such that  
\begin{align*}
		(\pp\cp\Phi)(y,X_1,\dots,X_n) \leq \qq[p]_1(X_1)\cdot {\dots}\cdot \qq[p]_n(X_n)\qquad\quad\forall\: y\in V_p,\:\: p=1,\dots,m 
	\end{align*} 
	holds for all $X_1\in F_1,\dots,X_n\in F_n$. Evidently, then \eqref{oaosapop} holds for any $\qq_1\in \SEMM_1,\dots,\qq_n\in \SEMM_n$ with $\qq[1]_k,\dots,\qq[m]_k\leq \qq_k$ for $k=1,\dots,n$.
\end{proof}

\subsection{}
\label{appdiffssddsssdds}

\begin{proof}[Proof of Lemma \ref{aaasasdswewe}]
It is clear that $\gamma^{(k)}$ is of class $C^1$ if $\gamma$ is of class $C^{k+1}$; and the other direction is clear if $D\equiv I$ is open.  
Thus, suppose that $D$ is not open; and that $\gamma$ is of class $C^k$ with $\gamma^{(k)}$ of class $C^1$. We define $r:= \inf\{D\}$ and $r':=\sup\{D\}$; and proceed as follows:
\begingroup
\setlength{\leftmargini}{12pt}
\begin{itemize}
\item
If $r\notin D$ holds, we let $D' := D$ and $\gamma':=\gamma$.
\item
If $r\in D$ holds, we let $D':=(r-\varepsilon,r)\sqcup D$ for some $\varepsilon>0$; and define $\gamma'\colon D'\rightarrow E$ by 
\begin{align*}
	\textstyle\gamma'|_{(r-\varepsilon,r)}:= (\cdot-r)^{k+1}/(k+1)\he !\cdot \big(\gamma^{(k)}\big)^{(1)}(r)+ \sum_{p=0}^{k} (\cdot-r)^p/p\he!\cdot \gamma^{(p)}(r)
\end{align*}
and $\gamma'|_D:=\gamma$.
\end{itemize}
\endgroup
\noindent
Then,
\begingroup
\setlength{\leftmargini}{12pt}
\begin{itemize}
\item
If $r'\notin D$ holds, we let $I:= D'$ and $\gamma'':=\gamma'$.
\item
If $r'\in D$ holds, we let $I:=D'\sqcup (r',r'+\varepsilon')$ for some $\varepsilon'>0$; and define $\gamma''\colon I\rightarrow E$ by 
\begin{align*}
	\textstyle \gamma''|_{(r',r'+\varepsilon')}:= (\cdot -r')^{k+1}/(k+1)\he!\cdot \big(\gamma^{(k)}\big)^{(1)}(r')+\sum_{p=0}^{k} (\cdot-r')^p/p\he!\cdot \gamma^{(p)}(r')
\end{align*}
and $\gamma''|_{D'}:=\gamma'$.
\end{itemize}
\endgroup
\noindent
By construction, $I$ is open; and we have $\gamma=\gamma''|_D$, for $\gamma''$ of class $C^{k+1}$.
\end{proof}

\subsection{}
\label{appdiff}
\begin{proof}[Proof of Lemma \ref{sddsdssd}]
Passing to $C^k$-extensions of $\gamma_i$ for $i=1,2$, we can assume that $D\equiv I$ is open. 
	Then, the first claim is clear from \ref{speccombo}, \ref{chainrule}. Moreover, for $\alpha$ as in \eqref{lkdslkdslkdslkdsds} and $t\in I$, we have
\begin{align*}	
	\textstyle\dot\alpha(t)&=\dd_t(\Psi\cp\beta)(1)\stackrel{\ref{chainrule}}{=}\dd\Psi(\beta(t),\dd_t\beta(1))\\
	&\stackrel{\ref{speccombo}}{=}\dd\Psi\big(\beta(t),\gamma^{(z_1+1)}_{i_1}(t)\times {\dots}\times \gamma_{i_m}^{(z_m+1)}(t)\big)\\&\textstyle\stackrel{\ref{productrule}}{=}\sum_{u=1}^m \partial_u\Psi\big(\beta(t),\gamma^{(z_u+1)}_{i_u}(t)\big),
\end{align*}
for $\beta\equiv\gamma^{(z_1)}_{i_1}\times {\dots}\times \gamma_{i_m}^{(z_m)}$; as well as 
\begin{align*}
	\partial_u\Psi=\dd \Psi|_{V\times\{0\}^{u-1} \times F_{i_u}\times \{0\}^{m-u}}\qquad\quad \forall\: u=1,\dots,m
\end{align*}   
smooth by \ref{iterated}. 
The second claim thus follows inductively, as it clearly holds for $p=0$.
\end{proof}

\subsection{}
\label{appdifff}
\begin{proof}[Proof of Lemma \ref{oopxcxopcoxpopcx}]
By Corollary \ref{fddfd}, for $0\leq \dindp\leq \dindu$, and each $\gamma\in C^\dindu([r,r'],W_1)$, $\psi\in C^\dindu([r,r'],F_2)$, we have $\Omega(\gamma,\psi)^{(\dindp)}=\sum_{i=1}^{d_\dindp} \alpha_{\dindp,i}(\gamma,\psi)$, with 
	\begin{align*}
		\alpha_{\dindp,i}\colon (\gamma,\psi)\mapsto ([\partial_1]^{m[\dindp,i]}\Omega)\big(\gamma,\gamma^{(z[\dindp,i]_{1})},\dots,\gamma^{(z[\dindp,i]_{m[\dindp,i]})},\psi^{(q[\dindp,i])}\big)
	\end{align*}  
	for certain $z[\dindp,i]_1,\dots,z[\dindp,i]_{m[\dindp,i]}, q[\dindp,i]\leq \dindp$ and $m[\dindp,i]\geq 1$. Then, 
	\begingroup
\setlength{\leftmargini}{15pt}
	{
\renewcommand{\theenumi}{{\arabic{enumi}})} 
\renewcommand{\labelenumi}{\theenumi}
\begin{enumerate}
\item
For $\pp\in \SEM$ fixed, Lemma \ref{kldskldsksdklsdl} provides us with $\qq_1\in \SEMM_1$, $\qq_2\in \SEMM_2$, and an open neighbourhood $V\subseteq F_1$ of $0$, such that 
	\begin{align*}
		\pp(\alpha_{\dindp,i}(\gamma,\psi))&\leq \qq_1\big(\gamma^{(z[\dindp,i]_{1})}\big)\cdot {\dots}\cdot \qq_1\big(\gamma^{(z[\dindp,i]_{m[\dindp,i]})}\big)\cdot \qq_2\big(\psi^{(q[\dindp,i])}\big)\qquad\forall\: i=1,\dots,d_\dindp,\:\: \dindp=0,\dots,\dindu
	\end{align*}
	holds, provided that we have $\im[\gamma]\subseteq V$. The claim thus holds for $\qq:= \max(d_0,\dots,d_\dindp)\cdot \qq_2$, and each $V\lleq \mm\in \SEMM_1$ with $\qq_1\leq \mm$.
\item
For $\pp\in \SEM$, $\gamma\in C^\dindu([r,r'],W_1)$ fixed, Corollary \ref{ofdpopfdofdp} provides us with $\qq_1\in \SEMM_1$, $\qq_2\in \SEMM_2$, such that 
	\begin{align*}
		\pp(\alpha_{\dindp,i}(\gamma,\psi))&\leq \qq_1\big(\gamma^{(z[\dindp,i]_{1})}\big)\cdot {\dots}\cdot \qq_1\big(\gamma^{(z[\dindp,i]_{m[\dindp,i]})}\big)\cdot \qq_2\big(\psi^{(q[\dindp,i])}\big)\qquad\forall\: i=1,\dots,d_\dindp,\:\: \dindp=0,\dots,\dindu.
	\end{align*}
	Since we have ${\qq_1}^\dindu_\infty(\gamma)<\infty$, the claim holds for $\qq= C\cdot \qq_2$, for $C\geq 0$ suitably large.
\end{enumerate}}
\endgroup
\noindent
This proves the claim.
\end{proof}

\subsection{}
\label{dkldksldkslsdlkklsdkldskl}
\begin{proof}[Proof of Lemma \ref{sdsdds}]
Recall that $\ovl{\dd_{\gamma(t)}f}$ is defined, linear, and continuous by Lemma \ref{pofsdisfdodjjxcycxj}. We choose $\delta>0$ such small that for each $h\in M:= (D-t)\cap ((-\delta,0) \cup (0,\delta))$, we have
\begin{align*}
	\gamma(t) + [0,1]\cdot \Delta_h\subseteq U\qquad\quad\text{for}\qquad\quad \Delta_h:=\gamma(t+h)-\gamma(t).
\end{align*}
We obtain from \eqref{Taylor} that
\begin{align}
	\textstyle\rcf{h}\:\cdot (f(\gamma(t+h))-f(\gamma(t)))
	&\textstyle=\rcf{h}\cdot \big(\dd_{\gamma(t)}f(\Delta_h) + \int_0^1 (1-s)\cdot \dd^2_{\gamma(t)+s\cdot \Delta_h}f(\Delta_h,\Delta_h) \:\dd s\big)\nonumber\\
	\label{pofdpofdpofdpofdasbsbsa}
	&\textstyle= \ovl{\dd_{\gamma(t)} f}\:(\rcf{h}\cdot\Delta_h) \hspace{2.5pt} +  \int_0^1 (1-s)\cdot \dd^2_{\gamma(t)+s\cdot \Delta_h}f(\rcf{h}\cdot \Delta_h,\Delta_h) \:\dd s
\end{align} 
holds for each $h\in M$. Since $\ovl{\dd_{\gamma(t)}f}$ is continuous, we have
\begin{align*}
	\textstyle\lim_{h\rightarrow 0} \ovl{\dd_{\gamma(t)} f}\:(\rcf{h}\cdot\Delta_h)=\ovl{\dd_{\gamma(t)}f}\he(X).
\end{align*} 
The claim thus follows once we have shown that the second summand in \eqref{pofdpofdpofdpofdasbsbsa} tends to zero if $h$ tends to zero. 
For this, we fix $\pp\in \SEM$; and choose $\qq_1,\qq_2\in \SEMM$ as well as $V\subseteq U$ open with $\gamma(t)\in V$ as in Lemma \ref{kldskldsksdklsdl},  
for $\Phi\equiv \dd^2f\colon U\times F\times F\rightarrow E$ and $x\equiv \gamma(t)$ there. Since $\lim_{h\rightarrow 0}\Delta_h=0$ holds by continuity of $\gamma$ at $t\in D$, we obtain
\vspace{-6pt}
\begin{align*}
	\textstyle\lim_{h\rightarrow 0}\pp\big(\int_0^1 (1-s)\cdot \dd^2_{\gamma(t)+s\cdot \Delta_h}f(\rcf{h}\cdot \Delta_h,\Delta_h) \:\dd s\big)&\textstyle\stackrel{\eqref{absch2}}{\leq}\lim_{h\rightarrow 0} \pp\big(\dd^2_{\gamma(t)+s\cdot \Delta_h}f(\rcf{h}\cdot \Delta_h,\Delta_h)\big)\\[-8pt]
	&\stackrel{\phantom{\eqref{absch2}}}{\leq}\textstyle\lim_{h\rightarrow 0} \qq_1(\rcf{h}\cdot \Delta_h)\cdot \qq_2(\Delta_h)\\[-4pt]
	&\stackrel{\phantom{\eqref{absch2}}}{=}\textstyle\lim_{h\rightarrow 0} \ovl{\qq}_1(\rcf{h}\cdot \Delta_h)\cdot \qq_2(\Delta_h)\\[-6pt]
	&\stackrel{\phantom{\eqref{absch2}}}{=}0;
\end{align*}
which shows the claim.
\end{proof}

\subsection{}
\label{ceinseig}

\begin{proof}[Proof of Lemma \ref{sddssdsdsd}]
By Lemma \ref{evk}.\ref{aaaaaaaaaaaaaaaab} and \ref{pogfpogf}, it suffices to show that there exists some $\mu\in C^1(I,G)$, for $I\subseteq \RR$ an open interval containing $[r,r']$, such that $\Der(\mu|_{[r,r']})=\phi$ holds:
\vspace{6pt}

\noindent
By assumption, for $p=0,\dots,n-1$, we have 
\begin{align}
\label{opaopopaopsopas}
	\phi|_{[t_p,t_{p+1}]} =\Der(\mu[p]|_{[t_p,t_{p+1}]})\qquad\quad\text{for some}\qquad\quad \mu[p]\in C^{k+1}(I_p,G)
\end{align}
with $I_p\subseteq \RR$ an open interval containing $[t_p,t_{p+1}]$; and, due to the first identity in \eqref{fgfggf}, we can assume that 
\begin{align*}
		\mu[p](t_{p+1})=\mu[p+1](t_{p+1})\qquad\quad\forall\: p=0,\dots,n-2
\end{align*}
holds. 
We write $I_0\equiv (\iota, \ell')$, $I_{n-1}\equiv (\ell,\iota')$, let $I\equiv (\iota,\iota')$, and define
\begingroup
\setlength{\leftmargini}{12pt}
\begin{itemize}
\item
$\psi\in C^0(I,\mg)$\:\:\hspace{2.2pt} by\:\: $\psi|_{(\iota,r)}:=\Der(\mu[0]|_{(\iota,r)})$,\:\: $\psi|_{[r,r']}:=\phi$,\:\: $\psi|_{(r',\iota')}:=\Der(\mu[0]|_{(r',\iota')})$.
\item
$\mu\in C^0(I,G)$\:\: by
\begin{align*}
	\mu|_{(\iota,r]}&:=\mu[0]|_{(\iota,r]},\\
	\mu|_{(t_p,t_{p+1}]}&:=\mu[p]|_{(t_p,t_{p+1}]} \qquad\qquad\forall\: p=0,\dots, n-1,\\
	\mu|_{(r',\iota')}&:=\mu[n-1]|_{(r',\iota')}.
\end{align*}  
\end{itemize}
\endgroup
\noindent
We obtain from \eqref{opaopopaopsopas} that
\begin{align}
\label{poapoaopaopsayxx}
\textstyle	\lim_{h\rightarrow 0}\rcf{h}\cdot \chart(\mu(t+h)\cdot \mu(t)^{-1})=\dd_e\chart(\psi(t))\qquad\quad\forall\: t\in I
\end{align}
holds; and now will conclude from Lemma \ref{sdsdds} that $\mu$ is of class $C^1$. 
\vspace{6pt}

\noindent
For this, we let $\tau \in I$ be fixed; and choose a chart $\chart'\colon G\supseteq \U'\rightarrow \V'\subseteq E$ with $\mu(\tau)\in \U'$. Moreover, 
we choose $V\subseteq \V$ open with $0\in V$, as well as $J\subseteq I$ open with $\tau\in J$, such that $\chartinv(V)\cdot \mu(J)\subseteq \U'$ holds. 
Then, shrinking $J$ if necessary, we can assume that $(\chart\cp\mult)(\mu(J), (\inv\cp \mu)(J))\subseteq V$ holds. 
For each $s\in J$, we define  
$\gamma_s\colon J\ni t\mapsto (\chart\cp\mult)(\mu(t),(\inv\cp \mu)(s))\in V$, as well as 
\begin{align*}
	f_s\colon V\rightarrow \V',\qquad x\mapsto (\chart'\cp\mult)(\chartinv(x), \mu(s)).  
\end{align*} 
Then, Lemma \ref{sdsdds} (second step) shows that
\begin{align*}
	\textstyle\lim_{h\rightarrow 0} \rcf{h}\cdot ((\chart'\cp\mu|_J)(s+h)- (\chart'\cp\mu|_J)(s)) &\textstyle\stackrel{\phantom{\eqref{poapoaopaopsayxx}}}{=}\lim_{h\rightarrow 0} \rcf{h}\cdot (f_s(\gamma_s(s+h))-f_s(\gamma_s(s)))\\[-2pt]
	&\textstyle \stackrel{\phantom{\eqref{poapoaopaopsayxx}}}{=}\dd_{\gamma_s(s)} f_s(\lim_{h\rightarrow 0 } \rcf{h}\cdot (\gamma_s(s+h) -\gamma_s(s)) )\\
	&\textstyle \stackrel{\eqref{poapoaopaopsayxx}}{=}(\dd_{\mu(s)} \chart' \cp \dd_{e}\RT_{\mu(s)})(\psi(s))
\end{align*}
holds, which shows that $\mu$ is of class $C^1$.
\end{proof}

\subsection{}
\label{appLip}
\begin{proof}[Proof of the Lipschitz Case in Lemma \ref{Adlip}]
We have to show that $\Ad_\mu(\phi)\in C^\lip([r,r'],\mg)$ holds for each $\mu\in C^1([r,r'],\mg)$, and each $\phi\in C^\lip([r,r'],\mg)$ with Lipschitz constants $\{L_\pp\}_{\pp\in \SEM}\subseteq \RR_{\geq 0}$. For this, we fix $\pp\in \SEM$; and observe that
\begin{align*}
	\ppp\big(\Ad_{\mu(t)}(\phi(t))-\Ad_{\mu(t')}(\phi(t'))\big)
	&\leq \ppp\big(\Ad_{\mu(t)}(\phi(t)-\phi(t'))\big) + \ppp\big(\big(\Ad_{\mu(t)}-\Ad_{\mu(t')}\big)(\phi(t'))\big)
\end{align*} 
holds. Then, 
\begingroup
\setlength{\leftmargini}{12pt}
\begin{itemize}
\item
	We let $\compact:=\im[\mu]$, choose $\pp\leq \mm\in \SEM$ as in \eqref{askasjkjksaasqwqqwasw} for $\nn\equiv \pp$ there; and obtain
\begin{align*}
	\ppp\big(\Ad_{\mu(t)}(\phi(t)-\phi(t'))\big)\leq  \mmm(\phi(t)-\phi(t'))\leq L_\mm\cdot |t'-t| 
\end{align*}
for $r\leq t<t'\leq r'$. 
\item
Since $\alpha\colon [r,r']\times \im[\phi]\ni(s,X)\rightarrow \partial_s\Ad_{\mu(s)}(X)$ is defined and continuous, Lemma \ref{ofdpofdpopssssaaaasfffff} shows that 
\begin{align*}
	\textstyle\ppp\big(\big(\Ad_{\mu(t)}-\Ad_{\mu(t')}\big)(\phi(t'))\big)\leq \int_t^{t'}  \ppp\big(\partial_s \Ad_{\mu(s)}(\phi(t'))\:\dd s\big)\leq C\cdot |t'-t|
\end{align*}
holds, for $C:=\sup\{\ppp(\alpha(s,X)) \:|\: (s,X)\in [r,r']\times \im[\phi] \}<\infty$.
\end{itemize}
\endgroup
\noindent
From this, the claim is clear.
\end{proof}

\subsection{}
\label{ceinseigaaaa}

\begin{proof}[Proof of the statement made in Remark \ref{exponentialmap}.\ref{exponentialmap1}]
We obtain from \eqref{lkdsklsdkdl}, \ref{pogfpogf}, \ref{subst} that
\begin{align}
\label{oopop}
\textstyle\exp(r\cdot X)\cdot \exp(s\cdot X)=\exp((r+s)\cdot X)=\exp(s\cdot X)\cdot \exp(r\cdot X)\qquad\quad\forall\: s,t\geq 0
\end{align}
holds. Then, \eqref{lkdsklsdkdl} shows $\Ad_{\exp(t\cdot X)}(X)=X$ for each $t\geq0$; thus,
\begin{align*}
	\textstyle\exp(t\cdot X)^{-1}\equiv [\innt_0^t \phi_X]^{-1} \stackrel{\ref{pogfpogfaaa}}{=} \innt_0^t -\phi_X\equiv \exp(-t\cdot X)\qquad\quad\forall\: t\geq 0.
\end{align*}
It follows that \eqref{oopop} even holds for all $s,r\in \RR$; i.e., that  $\beta \colon \RR\ni t\mapsto \exp(t\cdot X)\in G$ is a  group homomorphisms. 
Then, smoothness of $\beta$ is clear from \eqref{fgfggf}, \eqref{lkdsklsdkdl}, and Lemma \ref{evk}.\ref{aaaaaaaaaaaaaaaab}.
\end{proof}

\subsection{}
\label{ceinseigaaaaa}

\begin{proof}[Proof of the statement made in Remark \ref{exponentialmap}.\ref{exponentialmap3}]
We define $\psi\in C^0([r-2,r'+2],\mg)$ by 
\begin{align*}
	\psi|_{[r-2,r)}:=\phi(r),\qquad\quad\:\:\psi|_{[r,r']}:=\phi,\qquad\quad\:\:\psi|_{(r',r'+2]}:=\phi(r'),
\end{align*}
as well as $\beta\in C^1((r-1,r'+1),\mg)$ by 
\begin{align*}
	\textstyle\beta\colon (r-1,r'+1)\ni t\mapsto \int_{r-1}^t\psi(s)\:\dd s.
\end{align*} 
For $t\in (r-1,r'+1)$ fixed, and $0<h\leq 1$, we let 
\begin{align*}
	\textstyle Y:=\int_{r-1}^{t+h}\psi(s) \: \dd s,\qquad\quad\:\: X:=\int_{r-1}^{t}\psi(s) \: \dd s,\qquad\quad\:\: X_h:=\int_t^{t+h}\psi(s) \: \dd s;
\end{align*}
 and obtain  
\begin{align*}
	\textstyle(\exp\cp\beta)(t+h)&\textstyle\equiv\innt_0^1\phi_Y=\innt_0^1 \phi_X+\phi_{X_h}
	\textstyle\stackrel{\ref{kdsasaasassaas}}{=}\innt_0^1 \phi_{X_h} \cdot \innt_0^1 \phi_X\\
	&\textstyle= \exp\big(\int_t^{t+h}\psi(s) \: \dd s\big)\cdot \exp\big(\int_{r-1}^{t}\psi(s) \: \dd s\big).
\end{align*}
Since $\exp$ is of class $C^1$, we obtain from \eqref{lkdsklsdkdl} and \ref{chainrule} that $\Der(\exp\cp\beta)|_{[r,r']}=\phi$ holds; which shows the claim.
\end{proof}

\section{Appendix to Sect.\ \ref{dsjshdhkjshkhjsd}}

\subsection{}
\label{appB}
\begin{proof}[Proof of Equation \eqref{dfdssfdsfd}]
Applying a standard refinement argument, we obtain $r=t_0<{\dots}<t_n=r'$ as well as $\phi[p],\psi[p]\in \DIDE^k_{[t_p,t_{p+1}]}$ for $p=0,\dots,n-1$ with
\begin{align*}
	\phi|_{(t_p,t_{p+1})}=\phi[p]|_{(t_p,t_{p+1})},\: \psi|_{(t_p,t_{p+1})}=\psi[p]|_{(t_p,t_{p+1})}\qquad\quad\forall\: p=0,\dots,n-1.
\end{align*}
We let $\alpha:=[\innt_r^\bullet \phi]^{-1}[\innt_r^\bullet \psi],\:\mu:= \innt_r^\bullet\phi,\:\nu:= \innt_r^\bullet\psi$; and define
\begin{align}
\label{oopopoo}
	\textstyle\alpha_p&\textstyle:=\alpha|_{[t_p,t_{p+1}]}\stackrel{\eqref{defpio}}{=} \mu(t_p)^{-1}\big[\innt_{t_p}^\bullet \phi[p]\big]^{-1}\big[\innt_{t_p}^\bullet \psi[p]\big]\cdot \nu(t_p) \\
\label{oopopooasass}
	\mu_p&\textstyle:= \mu|_{[t_p,t_{p+1}]}\in C^{k+1}([t_p,t_{p+1}],\mg)
\end{align} 
for $p=0,\dots,n-1$. 
We obtain from \ref{kdskdsdkdslkds} that
\begin{align}
\label{hjdfhjfdd}
	\Der(\alpha_p)|_{(t_p,t_{p+1})}=\Ad_{\mu_p^{-1}}(\psi[p]-\phi[p])|_{(t_p,t_{p+1})}\qquad\quad\forall\: p=0,\dots,n-1
\end{align}
holds; so that 
Lemma \ref{Adlip} and \eqref{oopopooasass} show $\Ad_{\mu^{-1}}(\psi-\phi)\in \DP^k([r,r'],\mg)$. Then, for $t\in (t_{p},t_{p+1}]$ with $0\leq p\leq n-1$, we have
\begin{align*}
	\textstyle\innt_r^t\Ad_{\mu^{-1}}(\psi-\phi)&\stackrel{\eqref{defpio}, \eqref{hjdfhjfdd}}{=}\big[\alpha_p(t)\cdot \alpha_{p}(t_{p})^{-1}\big]\cdot \big[\alpha_{p-1}(t_{p})\cdot \alpha_{p-1}(t_{p-1})^{-1}\big]\cdot {\dots}\: \cdot \textstyle \big[\alpha_{0}(t_{1})\cdot \alpha_{0}(t_{0})^{-1}\big]\\
&\hspace{8.8pt}\stackrel{\eqref{oopopoo}}{=}\hspace{10.4pt}\alpha(t)
\end{align*}
which proves the claim.
\end{proof}	
	
\subsection{}
\label{appAaaaa}
\begin{proof}[Proof of Equation \eqref{zsazusazuzusazuzusazusazusazusa} and the $C^k$-statement made in the proof of Lemma \ref{pofdspospods}] It is straightforward from the triangle inequality, the properties of $\rho$, and smoothness of $\varrho$ that  $\psi\in C^\lip([r,r'],\mg)$ holds for $k\equiv \lip$. Thus, in order to prove that $\psi$ is of class $C^k$, and to verify Equation \eqref{zsazusazuzusazuzusazusazusazusa}, we can assume that $k\in \NN\sqcup \{\infty\}$ holds in the following. 
\vspace{6pt}

\noindent
Now, to prove the $C^k$-statement, we have to show that $\psi=\psi'|_{[r,r']}$ holds for some $\psi'\in C^k(I,\mg)$ with $I\subseteq \RR$ open containing $[r,r']$. For this, we define   
$\varrho'\in C^\infty(\RR,\RR)$ by 
\begin{align*}
	\varrho'|_{(-\infty,r)}:=\varrho(r)\qquad\qquad\varrho'|_{[r,r']} :=\varrho\qquad\qquad  \varrho'|_{(r',\infty)}:=\varrho(r');
\end{align*}
and let $\psi':=\dot\varrho'\cdot \rcK{\phi\cp\varrho'}\colon \RR\rightarrow \mg$. Then,
\begingroup
\setlength{\leftmargini}{12pt}
\begin{itemize}
\item
	we have $\psi'^{(m)}|_{\RR-[r,r']}=0$ for $0\leq m\leq k$, as well as
\begin{align}
\label{popdfpoodfpodpofdf}
	(\psi'|_{(t_p,t_{p+1})})^{(m)}=((\dot\varrho[p]\cdot\rcK{\wt{\phi}[p]\cp\varrho[p]})|_{(t_p,t_{p+1})})^{(m)}\qquad\quad\forall\: 0\leq m\leq k,\:\: 0\leq p\leq n-1.
\end{align}
\item
we obtain from \ref{chainrule} and \ref{productrule} that
\begin{align}
\label{popdfpoodfpodpofdfllll}
	(\dot\varrho[p]\cdot\rcK{\wt{\phi}[p]\cp\varrho[p]})^{(m)}(t_p)=0=(\dot\varrho[p]\cdot\rcK{\wt{\phi}[p]\cp\varrho[p]})^{(m)}(t_{p+1})
\end{align}
holds, for $0\leq m\leq k$ and $p=0,\dots,n-1$.
\end{itemize}
\endgroup
\noindent
Now, since $\psi'$ is of class $C^0$, we can assume that it is of class $C^q$ for some $0\leq q\leq k-1$. Then,  \eqref{popdfpoodfpodpofdf} (for $m\equiv q$ there) shows that 
\begin{align}
\label{nmdsnmsnmsdnmdsnmdsnmdsnmdnmds}
\psi'^{(q)}|_{[t_p,t_{p+1}]}=(\dot\varrho[p]\cdot\rcK{\wt{\phi}[p]\cp\varrho[p]})^{(q)}|_{[t_p,t_{p+1}]}\qquad\quad\forall\: p=0,\dots,n-1
\end{align}
holds; with $\psi'^{(q)}|_{\RR-[(t_0,t_1)\:\cup\:{\dots}\:\cup\: (t_{n-1},t_n)]}=0$ by the first-, and by the second point (for $m\equiv q$ there). 
Together with \eqref{popdfpoodfpodpofdfllll} (for $m\equiv q$ there), this implies that  
$\psi'^{(q)}$ is differentiable with 
\begin{align*}
	\psi'^{(q+1)}|_{\RR-[(t_0,t_1)\:\cup\:{\dots}\:\cup\: (t_{n-1},t_n)]}=0;
\end{align*}
 so that \eqref{popdfpoodfpodpofdf} and \eqref{popdfpoodfpodpofdfllll} (for $m\equiv q+1$ there) show that  $\psi'^{(q+1)}$ is continuous. It thus follows inductively that $\psi'$ is of class $C^k$.
\vspace{6pt} 
 
\noindent
In particular, \eqref{zsazusazuzusazuzusazusazusazusa} is now clear from 
\begin{align*}
	\psi|_{[t_p,t_{p+1}]}=\psi'|_{[t_p,t_{p+1}]}\stackrel{\eqref{nmdsnmsnmsdnmdsnmdsnmdsnmdnmds}}{=}(\dot\varrho[p]\cdot\rcK{\wt{\phi}[p]\cp\varrho[p]})|_{[t_p,t_{p+1}]}=\Der(\mu[p]\cp\varrho[p]|_{[t_p,t_{p+1}]})
\end{align*}
for $p=0,\dots,n-1$.
\end{proof}

\section{Appendix to Sect.\ \ref{podposddpospodpods}}

\subsection{}
\label{QuotConstrained}
\begin{proof}[Proof of the statement made in Example \ref{exxxcon}.\ref{exxxcon0}]
We let $\pi\colon E\rightarrow G\equiv E\slash \Gamma,\:\:X\mapsto[X]$ denote the canonical projection, define $e:=[0]$, and fix an   open neighbourhood $\OO\subseteq E$ of $0$, such that $\OO\cap [\OO+ [\Gamma-\{e\}]]=\emptyset$ holds.\footnote{Confer, e.g., Theorem 1.10 in \cite{Rudin} for the existence of such $\OO$.} 
Then, a chart of $G$ that is centered at $e\equiv [0]\in G$, is given by
\begin{align*}
	\chart\colon \U\equiv \pi(\OO)\rightarrow \V\equiv \OO,\qquad [X]\mapsto \pi^{-1}(X)\cap \OO\subseteq E.
\end{align*}
Then, for $\V\lleq \pp\in \SEM$ and $X_1,\dots,X_n\in E$ with $\pp(X_1)+{\dots}+\pp(X_n)\leq 1$, we have 
\begin{align*}
	\pp(X_1+{\dots}+X_n)\leq 1\qquad\quad\text{implying}\qquad\quad [X_1+{\dots}+X_n]\in \U;
\end{align*}
and obtain
\begin{align*}
	(\pp\cp\chart)(\chartinv(X_1)\cdot{\dots}\cdot\chartinv(X_n))&=(\pp\cp\chart)([X_1]\cdot{\dots}\cdot[X_n])\\
	&\equiv(\pp\cp\chart)([X_1+{\dots}+X_n])\\
	&=\pp(X_1+{\dots}+X_n)\\
	&\leq \pp(X_1)+{\dots}+\pp(X_n),
\end{align*} 
which shows that $G$ is locally $\mu$-convex. 
\end{proof}

\subsection{}
\label{BanachConstrained}
\begin{proof}[Proof of the statement made in Example \ref{exxxcon}.\ref{exxxcon1}]
We let $\uu\equiv\|\cdot\|$ denote the Banach norm on $E$; and can assume that $\V\lleq \uu$ holds, just by rescaling $\uu$ if necessary. We fix $0<\rad\leq 1$ with, cf.\ Remark \ref{banachchchch}
\begin{align}
\label{fdpofdpodfpo}
	\big\|\dd_{g\cdot\chartinv(x)\cdot q}\chart\big(\dd_{\chartinv(x)\cdot q}\LT_{g}\cp\dd_{\chartinv(x)}\RT_{q}\cp \dd_x\chartinv\big)\big\|_\op\leq \rad^{-1}
\end{align}
for all $g,q\in \U$ and $x\in \V$ with 
	$(\uu\cp\chart)(g),(\uu\cp\chart)(q),\uu(x)\leq \rad$; 
 and define $\oo:=\rad^{-2}\cdot \uu$. 
 Then, since we have $\rad\leq 1$, it suffices to show that 
\begin{align}
\label{opdfoopfd}
	(\uu\cp\chart)(\chart^{-1}(X_1)\cdot {\dots}\cdot \chart^{-1}(X_n))\leq \rad\cdot \varepsilon
\end{align}
holds for all $X_1,\dots,X_n\in E$ with $\oo(X_1)+{\dots}+\oo(X_n)=: \varepsilon\leq 1$. 
\vspace{6pt}

\noindent
Now, \eqref{opdfoopfd} is clear for $n=1$, as 
\begin{align*}
	\oo(X)\leq \varepsilon\quad\:\:\text{for}\quad\:\: X\in E\qquad\quad\Longrightarrow\qquad\quad (\uu\cp\chart)(\chartinv(X))=\rad^2\cdot \oo(X)\leq\rad\cdot \varepsilon. 
\end{align*}
We thus can assume that \eqref{opdfoopfd} holds for all $1\leq q\leq n$ for some $n\geq 1$, fix $X_1,\dots,X_{n+1}\in E$ with $\oo(X_1)+{\dots}+\oo(X_{n+1})=: \varepsilon\leq 1$, and define
\begin{align*}
	\rho\colon [0,1]\ni t\mapsto \chart(\chartinv(t\cdot  X_1)\cdot {\dots}\cdot \chartinv(t\cdot X_{n+1})). 
\end{align*} 
Then, applying the induction hypotheses, Lemma \ref{ofdpofdpopssssaaaasfffff} together with \eqref{LGPR} and \eqref{fdpofdpodfpo} gives
\begin{align*}
	\textstyle\uu(\rho(1))
	&\textstyle\leq  \sup_{t\in[0,1]} \uu(\dot\rho(t))\leq \rad^{-1}\cdot (\uu(X_1)+{\dots}+\uu(X_{n+1}))=\rad\cdot (\oo(X_1)+{\dots}+\oo(X_{n+1}))\leq \rad\cdot \varepsilon.
\end{align*}
Equation \eqref{opdfoopfd} thus follows inductively for each $n\geq 1$.
\end{proof}

\subsection{}
\label{HGConstrained}
\begin{proof}[Proof of the statement made in Example \ref{exxxcon}.\ref{exxxcon2}]
Let us first observe that  
\begin{align}
\label{opsopdsopsd}
	\textstyle(1+\varepsilon_{1})\cdot {\dots}\cdot (1+\varepsilon_n)-1\leq 2\cdot \sum_{k=1}^n \varepsilon_k
\end{align} 
holds, for $\varepsilon_1,\dots,\varepsilon_n>0$ with $\sum_{k=1}^n \varepsilon_k\leq 1/2$.  
This is clear for $n=1$; and follows inductively for each $n\geq 1$. In fact, suppose that \eqref{opsopdsopsd} holds for $n \geq 1$, and let $\varepsilon_1,\dots,\varepsilon_{n+1}>0$ with $\sum_{k=1}^{n+1} \varepsilon_k\leq 1/2$. Then, we obtain from \eqref{opsopdsopsd} that
\begin{align*}
	\textstyle(1+\varepsilon_{n+1}) \cdot (1+\varepsilon_{1})\cdot {\dots}\cdot (1+\varepsilon_n)-1\leq \big(2\cdot \sum_{k=1}^n \varepsilon_k\big) + \big(\varepsilon_{n+1}\cdot (1+2\cdot \sum_{k=1}^n \varepsilon_k)\leq 2\cdot \sum_{k=1}^{n+1} \varepsilon_k\big).
\end{align*}
Let now $\uu\in  \SEM$ be fixed. We choose $\uu\leq \ww\in \SEM$ as in \eqref{invACond} for $\vv\equiv\uu$ there, let $\oo:=2\cdot \ww$; and consider the chart 
\begin{align*}
	\chart\colon \U\equiv \MAU\rightarrow \V\equiv\MAU-\EINS,\qquad a\mapsto a-\EINS
\end{align*}
 for $\EINS\equiv e$; and obtain from \eqref{invACond} that
\begin{align}
\label{sakjaskjsakjasjkas}
\begin{split}
	(\uu\cp\chart)(\chart^{-1}(X_1)\cdot {\dots}\cdot \chart^{-1}(X_n))&=\uu((\EINS+X_1)\cdot {\dots}\cdot (\EINS+X_n)-\EINS)\\
	&\leq (1+\ww(X_1))\cdot {\dots}\cdot (1+\ww(X_n)) -1
\end{split}
\end{align}
holds for all $X_1,\dots, X_n\in \V$ with $n\geq 1$. 
Then, 
\begin{align*}
	\textstyle\oo(X_1)+{\dots}+\oo(X_n)=: \varepsilon\leq 1\qquad\quad\Longrightarrow\qquad\quad \sum_{k=1}^n \ww(X_k)\leq 1/2;
\end{align*}
and we conclude from \eqref{opsopdsopsd} and \eqref{sakjaskjsakjasjkas} that
\begin{align*}
	(\uu\cp\chart)\big(\chart^{-1}(X_1)\cdot {\dots}\cdot \chart^{-1}(X_n)\big)
	&\textstyle\leq 2\cdot \sum_{k=1}^n \ww(X_k)=\sum_{k=1}^n \oo(X_k)= \varepsilon
\end{align*} 
holds, which shows the claim.
\end{proof}

\section{Appendix to Sect.\ \ref{COMPAPPR}}

\subsection{}
\label{Mackeyind}
\begin{proof}[Proof of the statement made in Remark \ref{popodspodspodssds}] 
Let $\chart'\colon G\supseteq \U'\rightarrow \V'\subseteq E$ be a further chart of $G$ with $e\in \U'$ and $\chart'(e)=0$. Then, shrinking $\V$ if necessary, we can assume that 
\begin{align*}
	\Phi\equiv\dd\he (\chart'^{-1}\cp\chart) \colon \V\times E\rightarrow E 
\end{align*}
is defined. 
Let now $\pp\in \SEM$ be fixed. We choose $\qq\equiv \qq_1\in \SEM$ and $V\subseteq \V$ as in Lemma \ref{kldskldsksdklsdl}, additionally convex; and define $\gamma_x\colon [0,1]\ni t\mapsto (\chart'^{-1}\cp\chart)(t\cdot x)\in V$ for each $x\in V$. Then, Lemma \ref{ofdpofdpopssssaaaasfffff} shows
\begin{align*}
	\textstyle\pp(\chart'^{-1}\cp\chart)(x)&=\textstyle\pp(\gamma_x(1)- \gamma_x(0))\textstyle\leq  \int \pp(\dot\gamma_x(s)) \:\dd s\textstyle= \int (\pp\cp\Phi)(\gamma_x(s),\dot\gamma_x(s)) \:\dd s\leq \int \qq(x) \:\dd s=\qq(x)
\end{align*}
for each $x\in V$, from which the claim is clear.
\end{proof}

\subsection{}
\label{QuotientMackey}
\begin{proof}[Proof of the statement made in Example \ref{opfdopdfpofddfop}.\ref{exxxxcon0}] We let $\chart\colon \U\ni [X]\rightarrow X\in \V$   
be defined as in Appendix \ref{QuotConstrained}; and fix $V\subseteq \V$ symmetric open with $\ovl{V}\subseteq \V$ and $\ovl{V}+\ovl{V}\subseteq \V$. Then, for $X,Y\in \ovl{V}$ (or, alternatively, $[X],[Y]\in \chartinv(\ovl{V})$), we have 
\begin{align*}
	\pp(-X + Y)=(\pp\cp\chart)([-X+Y])=(\pp\cp\chart)([X]^{-1}\cdot [Y])\qquad\quad\forall\: \pp\in \SEM.
\end{align*}
The claim now follows easily from Remark \ref{confkjdsjdskjsdkjds}.\ref{confev24}, when applied to $U\equiv V$ as well as $U\equiv \chartinv(V)$ there. 
\end{proof}

\subsection{}
\label{BanachMackey}
\begin{proof}[Proof of the statement made in Example \ref{opfdopdfpofddfop}.\ref{exxxxcon1}]
Let $\|\cdot\|$ denote the Banach norm on $E$. Then, Lemma \ref{fhfhfhffhaaaa} applied to $\compact\equiv\{e\}$ and $\pp\equiv\|\cdot\|$, provides us with an open neighbourhood $V$ of $e$ as well as some $C>0$, such that
\begin{align}
\label{jksdjksdjksd}
	\|\chart(q)-\chart(q')\|\leq C\cdot \|\chart_{h}(q)-\chart_{h}(q')\|\qquad\quad\forall\: q,q',h\in V
\end{align}  
holds. 
We fix an open neighbourhood $U\subseteq G$ of $e$ with $\ovl{U}\subseteq V\cap\U$; and recall that in order to show that $G$ is sequentially complete -- by Remark \ref{confkjdsjdskjsdkjds}.\ref{confev24} -- it suffices to show that each Cauchy sequence $\{g_n\}_{n\in \NN}\subseteq U\subseteq G$ converges in $G$.   
Now, \eqref{jksdjksdjksd} applied to $h\equiv g_{m}$ gives
\begin{align*}
	\|\chart(g_m)-\chart(g_n)\|\leq C\cdot \|\chart(g_m^{-1}\cdot g_n)\|\qquad\quad\forall\: m,n\in \NN,
\end{align*} 
which shows that $\{\chart(g_n)\}_{n\in \NN}\subseteq \chart(U)\subseteq \V$ is a Cauchy sequence in $E$. By assumption, $\lim_n \chart(g_n)=x\in \chart(\he\ovl{U}\he)\subseteq \V$ exists; so that $\{g_n\}_{n\in \NN}$ converges to $\chartinv(x)\in G$.
\end{proof}

\subsection{}
\label{HGMackey}
\begin{proof}[Proof of the statement made in Example \ref{opfdopdfpofddfop}.\ref{exxxxcon2}]
Recall that $\MAU$ is locally $\mu$-convex by Example \ref{exxxcon}.\ref{exxxcon2}; and let 
$\chart\colon \U\cong\MAU\ni a\mapsto  a-\EINS\in \V\equiv\MAU-\EINS$ be as in Appendix \ref{HGConstrained}. 
	Let furthermore $\{a_n\}_{n\in \NN}\subseteq \MAU$ be a fixed sequence. 
\begingroup
\setlength{\leftmargini}{12pt}
\begin{itemize}
\item
We fix $\vv\in \SEM$, choose $\vv\leq \mm\in \SEM$ as in \eqref{invACond} for $\ww\equiv \mm$ there, and obtain
\begin{align}
\label{dfpodfopdfofdpo}
	\vv(a_{n}-a_{n-1})&\textstyle=\vv\big(a_{n-1}\cdot \big(a^{-1}_{n-1}\cdot a_n-\EINS\big)\big)\leq \mm(a_{n-1})\cdot (\mm\cp\chart)\big(a^{-1}_{n-1}\cdot a_n\big)\qquad\quad\forall\: n\geq 1.
\end{align}
\item
We choose $\mm\leq \uu\in \SEM$ as in \eqref{invACond} for $\vv\equiv \mm$ and $\ww\equiv \uu$ there, and let  
$\uu\leq \oo\in \SEM$ be as in \eqref{aaajjhguoiuouo}. Then, passing to a subsequence if necessary, we can achieve that   
\begin{align}
\label{saoqqsaopsaklsaklaskllask}
 	\textstyle\sum_{n=1}^\infty\oo(X_n)\leq 1\qquad\quad\:\:\text{holds for}\qquad\quad\:\: X_n:=\chart\big(a^{-1}_{n-1}\cdot a_{n}\big)\qquad\forall\: n\geq 1.
\end{align}
We obtain 
\begin{align*}
	\textstyle\mm(a_n)&\:\he=\:\he\mm(a_0\cdot\chartinv(X_1)\cdot {\dots}\cdot\chartinv(X_n))\stackrel{\eqref{invACond}}{\leq} \uu(a_0)\cdot \uu(\chartinv(X_1)\cdot {\dots}\cdot\chartinv(X_n))\\
	&\:\he\leq\:\he \uu(a_0)\cdot \big(\uu(\EINS) + (\uu\cp\chart)(\chartinv(X_1)\cdot {\dots}\cdot\chartinv(X_n))\big) \stackrel{\eqref{aaajjhguoiuouo},\eqref{saoqqsaopsaklsaklaskllask}}{\leq} \uu(a_0)\cdot (\uu(\EINS)+1),
\end{align*}
implying $\textstyle\sup\{\mm(a_n)\:|\: n\in \NN\}<\infty$.
\end{itemize}
\endgroup
\noindent
It is thus clear from \eqref{dfpodfopdfofdpo} that: 
\begingroup
\setlength{\leftmargini}{12pt}
\begin{itemize}
\item
If $\MA$ is sequentially complete, and $\{a_n\}_{n\in \NN}\subseteq \MAU$ a Cauchy sequence, then $\lim_n a_n=a\in \MA$ exists.
\item
If $\MA$ is Mackey complete, and $\{a_n\}_{n\in \NN}\subseteq \MAU$ a Mackey-Cauchy sequence, then $\lim_n a_n=a\in \MA$ exists.
\end{itemize}
\endgroup
\noindent
Now, since $\MAU$ is open with $\EINS\in \MAU$, there exists an open neighbourhood $V$ of $\EINS$ with $\ovl{V}\subseteq \MAU$, as well as some $p\geq 0$ with $\{a_{p}^{-1}\cdot a_n\}_{n\geq p}\in V$. Then, 
\begin{align*}
	\textstyle a_p^{-1}\cdot a=\lim_{n}(a_p^{-1}\cdot a_n)\in \ovl{V}\subseteq \MAU, 
\end{align*}
implies $a\in \MAU$; which proves the claim. 
\end{proof}

\section{Appendix to Sect.\ \ref{sec:confined}}

\subsection{}
\label{assaasssaaggggs}
\begin{proof}[Proof of Equation \eqref{asiosaioaois}] 
We fix $q\in \NN$; and 
 choose  $r=t_0<{\dots}<t_n=r'$, as well as $\phi_q[p]$ for $p=0,\dots,n-1$, as in \eqref{opopooppo} for $\phi\equiv \phi_q$ and $\phi[p]\equiv\phi_q[p]$ there. 
 Then, it is clear from \eqref{defpio} that $\mu_q$ is of class $C^1$ on $J:=\bigsqcup_{p=0}^{n-1}(t_p,t_{p+1})$ with 
 $\phi_q=\dd_{\mu_q}\RT_{\mu_q^{-1}}(\dot\mu_q)$ thereon; so that we have
\begin{align}
\label{posdopsdopsdp}
	\dermapinvdiff(\gamma_q,\phi_q)|_J=(\dd_{\mu_q}\chart\cp\dd_e\RT_{\mu_q})(\phi_q)|_J=(\dd_{\mu_q}\chart\cp\dd_e\RT_{\mu_q}\cp\dd_{\mu_q}\RT_{\mu_q^{-1}})(\dot\mu_q)|_J=\dot\gamma_q|_J.
\end{align}
We define $\alpha_p:=\dermapinvdiff(\gamma_q|_{[t_p,t_{p+1}]},\phi_q[p])$ for $p=0,\dots,n-1$; and conclude from \eqref{isdsdoisdiosd} and \eqref{posdopsdopsdp} that
\begin{align}
\label{dspodspodspodspopodspodsaaa}
	\textstyle\gamma_q(\tau')-\gamma_q(\tau) = \int_\tau^{\tau'} \dermapinvdiff(\gamma_q(s),\phi_q(s))\:\dd s
	= \int_\tau^{\tau'} \alpha_p(s)\: \dd s
\end{align}
holds, for each $\COMP\ni[\tau,\tau']\subseteq (t_p,t_{p+1})$. 
Since $\gamma_q,\alpha_0,\dots,\alpha_{n-1}$ are continuous, we obtain
\begin{align*}
	\textstyle\gamma_q(\tau')-\gamma_q(\tau)&\textstyle\:\he=\:\he \lim_{k\rightarrow \infty} \he (\gamma_q(\tau'-1/k)-\gamma_q(\tau+1/k))\textstyle\\
	& \textstyle\stackrel{\eqref{dspodspodspodspopodspodsaaa}}{=}\lim_{k\rightarrow \infty}  \int_{\tau+1/k}^{\tau'-1/k}\alpha_p(s)\: \dd s\textstyle=\int_{\tau}^{\tau'} \alpha_p(s) \:\dd s= \int_{\tau}^{\tau'} \dermapinvdiff(\gamma_q(s),\phi_q[p](s)) \:\dd s
\end{align*}
for each $t_p\leq \tau< \tau'\leq t_{p+1}$, for $p=0,\dots,n-1$. The claim is thus clear from \eqref{opofdpopfd} and $\gamma_q(r)=0$.
\end{proof}

\subsection{}
\label{assaauuuuuuu}
\begin{proof}[Proof of Lemma \ref{omori}]
Let $\frac{\dd}{\dd h}\big|^>_{h=0}$ denote the right derivative; and define $\mu:=\innt_r^\bullet \phi$.
\vspace{6pt}

\noindent
For the implication ``$\Longrightarrow$'', 
\begingroup
\setlength{\leftmargini}{12pt}
\begin{itemize}
\item
we observe that $\alpha:=\Ad_{\mu}(Y)$ is of class $C^1$ with $\alpha(r)=Y$.
\item
we choose an extension $\psi \in \DIDE_{[r,r'+\delta]}$ of $\phi$, for some $\delta>0$; and define $\beta:=\Ad_{\nu}(Y)$ for $\nu:=\innt_r^\bullet \psi$. 
\item
we obtain from \ref{pogfpogf} that
	\begin{align*}
	\dot\alpha(t)=\dot\beta(t)=\textstyle\frac{\dd}{\dd h}\big|^>_{h=0}\: \Ad_{\innt_t^{t+h}\psi\he}(\Ad_{\mu(t)}(Y))=\bl\phi(t),\alpha(t)\br\qquad\quad\forall\: t\in [r,r'].
\end{align*} 
\vspace{-18pt}
\end{itemize}
\endgroup
\noindent
For the implication ``$\Longleftarrow$'', 
\begingroup
\setlength{\leftmargini}{12pt}
\begin{itemize}
\item
we suppose that $\dot\alpha=\bl\phi,\alpha\br$ holds for $\alpha\in C^1([r,r'],\mg)$.
\item
	we choose an extension $\psi \in \DIDE_{[r,r'+\delta]}$ of $\phi$, and an extension $\beta\in C^1([r,r'+\delta],\mg)$ of $\alpha$, for some $\delta>0$; and define 
		$\gamma:=\Ad_{[\innt_r^{\bullet}\psi]^{-1}}(\beta)$. 
\item
	we recall that, cf.\ \ref{pogfpogfaaa} 
\begin{align*}
	\textstyle\big[\innt_t^{t+h}\psi\big]^{-1}=\innt_t^{t+h}-\Ad_{[\innt_t^\bullet \psi]^{-1}}(\psi)\qquad\quad\forall\: 0<h\leq\delta,\:\: t\in [r,r'];
\end{align*}
and conclude from \ref{pogfpogf}, \ref{linear}, \ref{productrule} that 
\begin{align*}
	\textstyle\dot\gamma(t)&\textstyle=\frac{\dd}{\dd h}\big|^>_{h=0}\:\Ad_{[\innt_r^{t+h}\psi]^{-1}}(\beta(t+h))
	\textstyle=\frac{\dd}{\dd h}\big|^>_{h=0}\: \Ad_{\mu^{-1}(t)}\big(\Ad_{[\innt_t^{t+h}\psi]^{-1}}(\beta(t+h))\big)\\
	&=\Ad_{\mu^{-1}(t)}\big([-\phi(t),\alpha(t)]+ \dot\alpha(t)\big)=0 
\end{align*}
holds for each $t\in [r,r']$.
\item
We thus conclude from \eqref{isdsdoisdiosd} that $\Ad_{[\innt_r^{\bullet}\phi]^{-1}}(\alpha)=\alpha(r)=Y$ holds; thus, $\alpha=\Ad_{\mu}(Y)$.
\end{itemize}
\endgroup
\noindent
This proves the claim.
\end{proof}

\section{Appendix to Sect.\ \ref{asopsopdsopsdpoosdp}}

\subsection{}
\label{dsdsdsds}
\begin{proof}[Proof of Equation \eqref{powqopqwopwasasaassaq}]
We choose an open neighbourhood $V\subseteq E$ of $0$, such that  
\begin{align*}
	f\colon V\times V\ni (x,y)\mapsto (\chart\cp\mult)(\chartinv(x),\chartinv(y))
\end{align*}
is defined. Then, shrinking $\delta$ if necessary, we can assume that
\begin{align*}
	\gamma\colon [0,\delta]\ni t\mapsto ((\chart\cp\mu)(t),(\chart\cp\nu)(t))\in V\times V
\end{align*}
holds; and conclude from Lemma \ref{sdsdds} (for $F\equiv E\times E$ and $U=V\times V$ there) that 
\begin{align}
\label{weweeewweweldldlld}
\begin{split}
	\textstyle\lim_{h\rightarrow 0}\rcf{h}\cdot \chart(\mu(h)\cdot\nu(h))&\textstyle=\lim_{h\rightarrow 0}\rcf{h}\cdot (f(\gamma(h))-f(\gamma(0)))\\
	&=\ovl{\dd_{\gamma(0)}f}\he(X,Y)\\
	&=X+Y
\end{split}
\end{align}
holds. For the last step, observe that $\dd_{(e,e)}\mult(v,w)=v+w$  
holds for all $v,w\in \mg$ by \eqref{LGPR}; thus,
\begin{align*}
	\dd_{\gamma(0)}f(Z,Z')=(\dd_e\chart\cp \dd_e\mult)(\dd_0\chartinv(Z),\dd_0\chartinv(Z'))=Z+Z'\qquad\quad\forall\: Z,Z'\in E.
\end{align*}
The last step in \eqref{weweeewweweldldlld} is thus clear from continuity of $\ovl{\dd_{\gamma(0)}f}$.
\end{proof}

\subsection{Appendix}
\label{appLipasasasasas}
\begin{proof}[Proof of Lemma \ref{fdfdfdf}]
By \eqref{absch2}, it suffices to show that 
\begin{align*}
	\textstyle\wt{\Gamma}\colon C^k([r,r'],\mg)\times C^k([r,r'],\mg)\rightarrow C^0([r,r'],\mg),\qquad (\phi,\psi)\mapsto \big[t\mapsto \Gamma\big(\innt_r^{t}\phi,\psi(t)\big)\big]
\end{align*}
is continuous.  
For this, we let $\pp\in\SEM$, $\varepsilon>0$, and $(\phi,\psi)\in C^k([r,r'],\mg)\times C^k([r,r'],\mg)$ be fixed; and have to show that there exist $\qq\in \SEM$ and $\dind\llleq k$, such that
\begin{align}
\label{podspodspodspodsds}
	\qqq^\dind_\infty(\phi'-\phi),\:\: \qqq^\dind_\infty(\psi'-\psi)\leq 1\qquad\quad\Longrightarrow\qquad\quad\ppp_\infty\big(\wt{\Gamma}(\phi',\psi')-\wt{\Gamma}(\phi,\psi)\big)\leq \varepsilon 
\end{align}
for $\phi',\psi'\in C^k([r,r'],\mg)$. 
We let $\mu:=\innt_r^\bullet\phi$, and consider the continuous map
\begin{align*}
	\alpha\colon G\times \mg \times  G\times \mg \rightarrow \mg,\qquad ((g,X),(g',X'))\mapsto \ppp(\Gamma(g,X)-\Gamma(g',X')).
\end{align*}
 Then, for $t\in [r,r']$ fixed, there exists an open neighbourhood $W[t]\subseteq G$ of $e$, as well as $U[t]\subseteq \mg$ open with $0\in U[t]$, such that 
\begin{align}
\label{pofsofspofs}
	\textstyle\alpha((g,X),(g',X'))\leq \varepsilon\qquad\quad\forall\: (g,X),(g',X')\in \big[\mu(t)\cdot W[t]\big]\times \big[\psi(t) + U[t]\big]
\end{align}
holds. We  choose
\begingroup
\setlength{\leftmargini}{12pt}
\begin{itemize}
\item
	$V[t]\subseteq G$ open with $e\in V[t]$ and $V[t]\cdot V[t]\subseteq W[t]$.
\item
	$O[t]\subseteq \mg$ open with $0\in O[t]$ and $O[t]+O[t]\subseteq U[t]$.
\item
	$J[t]\subseteq \RR$ open with $t\in J$, such that for $D[t]:=J[t]\cap [r,r']$, we have 
	\begin{align}
	\label{askljasjaljajskasa}
		\mu(D[t])\subseteq \mu(t)\cdot V[t]\subseteq \mu(t)\cdot W[t]\qquad\text{and}\qquad \psi(D[t])\subseteq \psi(t)+O[t]\subseteq \psi(t)+U[t].
	\end{align}
	\vspace{-22pt}
\end{itemize}
\endgroup
\noindent
Since $[r,r']$ is compact, there exist $t_0,\dots,t_n\in [r,r']$, such that $[r,r']\subseteq D_0\cup{\dots}\cup D_n$ holds.
\begingroup
\setlength{\leftmargini}{12pt}
\begin{itemize}
\item
	We define $V:= V[t_0]\cap{\dots}\cap V[t_n]$.
	 
	Then, Lemma \ref{fddfxxxxfd} provides us with some $\mm$ and $\dind\llleq k$, such that 
	\begin{align}
	\label{asljaljsasaljljsaas}
		\textstyle\innt_r^\bullet \phi'\in \innt_r^\bullet \phi\cdot V\qquad\text{holds for each}\qquad \phi'\in C^k([r,r'],\mg)\qquad\text{with}\qquad \mmm_\infty^\dind(\phi'-\phi)\leq 1. 
	\end{align}
\item
	We define $O:=O[t_0]\cap{\dots}\cap O[t_n]$; and fix some $O\lleq \qq\in \SEM$ with $\mm\leq \qq$.   
\end{itemize}
\endgroup
\noindent
Let now $\phi',\psi'\in C^k([r,r'],\mg)$ be given with $\qqq^\dind_\infty(\phi'-\phi),\:\qqq^\dind_\infty(\psi'-\psi)\leq 1$. 
Then, for $\tau\in D_p$ wit $0\leq p\leq n$, we obtain from \eqref{asljaljsasaljljsaas}, $O\lleq \qq$, and \eqref{askljasjaljajskasa} for $t\equiv t_p$ there that
\begingroup
\setlength{\leftmargini}{12pt}
\begin{itemize}
\item
$\textstyle \mu(t_p)^{-1} \cdot \innt_r^{\tau} \phi'  = \textstyle\big(\mu(t_p)^{-1}\cdot\mu(\tau)\big) \cdot \big([\innt_r^{\tau} \phi]^{-1} [\innt_r^{\tau} \phi']\big)  \in V\cdot V\subseteq W[t_p]$.
\vspace{2pt}
\item
$\psi'(\tau)-\psi(t_p)= (\psi'(\tau)-\psi(\tau))+ (\psi(\tau)-\psi(t_p))\in O + O\subseteq U[t_p]$.
\end{itemize}
\endgroup
\noindent
The claim is thus clear from \eqref{pofsofspofs} and \eqref{askljasjaljajskasa}. 
\end{proof}

\end{document}